\documentclass[11pt, a4paper]{article}

\usepackage{helvet}

\usepackage[utf8]{inputenc}
\usepackage[colorlinks=true, urlcolor=blue]{hyperref}

\usepackage{amsfonts}
\usepackage{amssymb,amsmath,mathrsfs}
\usepackage{amsthm}
\usepackage{bbm}
\usepackage{mathtools}
\usepackage[margin=1.1in, includehead, includefoot]{geometry}
\usepackage{datetime}
\usepackage{dsfont}
\usepackage{verbatim}
\usepackage{enumerate}
\usepackage{MnSymbol}
\usepackage[titletoc, title]{appendix}
\usepackage{tikz}
\usepackage{caption}
\usepackage{float}
\usepackage{arydshln}
\usepackage{fancyhdr}
\usepackage[en-US]{datetime2}
\usepackage[nottoc, notlot, notlof]{tocbibind}
\usepackage{amsmath}
\usepackage{amssymb}
\usepackage{graphicx}

\providecommand{\U}[1]{\protect\rule{.1in}{.1in}}

\newtheorem{theorem}{Theorem}[section]
\newtheorem*{theorem*}{Theorem}
\newtheorem{proposition}[theorem]{Proposition}
\newtheorem{lemma}[theorem]{Lemma}
\newtheorem{corollary}[theorem]{Corollary}
\newtheorem{example}[theorem]{Example}
\newtheorem{condition}{Condition}
\newtheorem{definition}[theorem]{Definition}
\newtheorem{remark}[theorem]{Remark}
\newcommand{\wrt}[1]{\, \mathrm{d} #1}
\newcommand{\strato}[1]{\circ \mathrm{d} #1}
\newcommand{\pd}[2]{\frac{\partial #1}{\partial #2}}
\newcommand{\exptn}[1]{\mathbb{E} \left[ #1\right]}
\newcommand{\fqquad}{\qquad \qquad \qquad \qquad}
\newcommand{\abs}[1]{\left\lvert#1\right\rvert}
\newcommand{\norm}[1]{\left\lVert#1\right\rVert}
\begin{document}

\author{Thomas Cass and Nengli Lim}
\title{{\huge A Stratonovich-Skorohod formula for Volterra Gaussian rough paths}}
\maketitle

\begin{abstract}
Given a solution $Y$ to a rough differential equation (RDE), a recent result \cite{cl2018} extends the classical It\"{o}-Stratonovich formula and provides a closed-form expression for $\int Y \circ \mathrm{d} \mathbf{X} - \int Y \, \mathrm{d} X$, i.e. the difference between the rough and Skorohod integrals of $Y$ with respect to $X$, where $X$ is a Gaussian process with finite $p$-variation less than 3. In this paper, we extend this result to Gaussian processes with finite $p$-variation such that $3 \leq p < 4$. The constraint this time is that we restrict ourselves to Volterra Gaussian processes with kernels satisfying a natural condition, which however still allows the result to encompass many standard examples, including fractional Brownian motion with $H > \frac{1}{4}$. Analogously to \cite{cl2018}, we first show that the Riemann-sum approximants of the Skorohod integral converge in $L^2(\Omega)$ by adopting a suitable characterization of the Cameron-Martin norm, before appending the approximants with higher-level compensation terms without altering the limit. Lastly, the formula is obtained after a re-balancing of terms, and we also show how to recover the standard It\"{o} formulas in the case where the vector fields of the RDE governing $Y$ are commutative.
\end{abstract}

\section{Introduction}

Lyons' rough path theory is a framework for giving a path-wise interpretation
to stochastic differential equations of the form
\begin{align} \label{basic SDE}
\mathrm{d}Y_{t}=V(Y_{t})\circ\mathrm{d}X_{t},\quad
Y_{0}=y_{0},
\end{align}
in particular for a broad class of continuous, vector-valued Gaussian processes $X$. A fundamental contribution of Lyons \cite{lyons98, lcl2006} was to
realize that this needs $X$ to be enriched to a \textit{rough path} $\mathbf{X}$ whose components comprise not only $X$, but also the higher-order iterated integrals up to some finite degree. The model \eqref{basic SDE} then ought to be rewritten as
\begin{align} \label{basic SDE1}
\mathrm{d}Y_{t}=V(Y_{t})\circ\mathrm{d}\mathbf{X}_{t},\quad
Y_{0}=y_{0},
\end{align}
to reflect the dependence of the solution on the enriched rough path. The existence of $\mathbf{X}$ for a given stochastic process cannot be taken for granted. Nevertheless, if $X$ is Gaussian then $\mathbf{X}$ can be constructed in a canonical way as an appropriate limit of iterated integrals of smooth (i.e. bounded variation) approximations to $X$ such as piecewise linear approximations, see \cite{fv2010a, fv2010b}. For this, conditions are needed on the covariance function, but these are flexible enough to encompass a range of examples including fractional Brownian motion with Hurst index $H > \frac{1}{4}$. \par 

In this paper we will assume that $X=\left(  X^{1},...,X^{d}\right)  $ has
i.i.d components, each centered with covariance function $R$, and that $X$ is
defined on a probability space $\left(  \Omega,\mathcal{F},\mathbb{P}\right)
$. For simplicity we assume that $\mathcal{F}$ is generated by $X.$ The
process $X$ then gives rise an isonormal Gaussian process w.r.t. the Hilbert
$\mathcal{H}_{1}^{d}=\oplus_{i=1}^{d}\mathcal{H}_{1}^{\left(  i\right)  }$ where, for all $i=1,...,d,$ $\mathcal{H}_{1}^{\left(  i\right)  } = \mathcal{H}_{1}$ and $\mathcal{H}_{1}$ is the completion of the real vector space
\begin{align*}
\text{span}\left\{  \mathds{1}_{[0,t)}(\cdot):\bigg\vert\;t\in[0,T]\right\}
\end{align*}
endowed with the inner-product $\left\langle \mathds{1}_{[0,t)}(\cdot), \mathds{1}_{[0,s)}(\cdot) \right\rangle_{\mathcal{H}_1} = R(t,s)$. The solution $Y$ to (\ref{basic SDE1}) can also be viewed as a Wiener functional on $\left(\Omega,\mathcal{F},\mathbb{P}\right) $, and its properties can then be studied using Malliavin calculus. A number of recent works have opened up the interplay between Lyons' and Malliavin's calculi, see e.g. \cite{cfv2009}, \cite{cf2011}, \cite{hp2013} and \cite{chlt2015}. In particular, in a recent paper \cite{cl2018} the authors have proven a conversion formula for the difference between the rough path integral of $Y$ w.r.t $\mathbf{X}$ and the Skorohod integral $\delta^{X}$ of $Y$ (i.e. the $L^{2}\left( \Omega\right)$ adjoint of the Malliavin derivative operator). In more detail this result shows, for the case where $Y$ and $X$ are both $\mathbb{R}^{d}$-valued, the following almost sure identity
\begin{align} \label{corr form}
\begin{split}
\int_{0}^{T}\left\langle Y_{t}\circ\mathrm{d}\mathbf{X}_{t}\right\rangle - \delta^{X}\left(  Y\right)  
&= \frac{1}{2}\int_{0}^{T}\mathrm{tr}\left[V(Y_{t})\right]  \,\mathrm{d}R(t) \\
&\qquad+ \int_{[0,T]^{2}}\mathds{1}_{[0,t)}(s) \mathrm{tr}\left[  J_{t}^{\mathbf{X}}\left(  J_{s}^{\mathbf{X}}\right)^{-1} V(Y_{s})-V(Y_{t})\right]  \,\mathrm{d}R(s,t).
\end{split}
\end{align}
Here, $J_{t}^{\mathbf{X}}$ denotes the Jacobian of the flow map $y_{0} \rightarrow Y_{t},$ and the second part of the correction term is a proper 2D Young-Stieltjes integral (see \cite{fv2010a, fv2010b}) with respect to the covariance function of $X.$ When $X$ is standard Brownian motion, this last term vanishes since the integrand is zero on the diagonal and $\mathrm{d}R(s,t) = \delta_{\{s=t\}}\,\mathrm{d}s\,\mathrm{d}t$. This, together with the fact that $R(t)=t$, allows us to recover the classical It\^{o}-Stratonovich conversion formula. This a rather basic tool in stochastic analysis and generalizations are likely to be important especially given the now widespread adoption of Gaussian models, e.g. most recently in mathematical
finance \cite{bfg2016}.

In \cite{cl2018}, conditions need to be imposed in the proof of the formula (\ref{corr form}) which limit the range of applications. An important assumption, for instance, is that the covariance function of $X$ has finite (two-parameter) $\rho$-variation for $\rho \in \left[ 1, \frac{3}{2} \right)$. This implies that the sample paths of $X$ will have finite $p$-variation, for some $p\in [2, 3)$, and this excludes interesting examples such as fractional Brownian motion with $H \in \left( \frac{1}{4}, \frac{1}{3} \right]$. \par 

The purpose of the present paper is to extend the correction formula \eqref{corr formula} to these less regular cases. To do so we will assume that the Gaussian process $X$ is a Volterra process; that is, the covariance function $R$ of each component can be written as
\begin{align*}
R(s,t)=\int_{0}^{t\wedge s}K(t,r)K(s,r)\,\mathrm{d}r,
\end{align*}
for some kernel $K$, a square-integrable function $K: \left[ 0, T \right]^2  \rightarrow\mathbb{R}$ with $K(t,s)=0,\;\forall s\geq t.$ We will present conditions on $K$ that allow us to generalize \eqref{corr form}. In doing so, we need to overcome a number of serious obstacles. We highlight here the three most salient of these, outline the contribution of the present work and, at the same time, provide a road-map for the paper:
\begin{enumerate} [(i)]
\item We need to prove that the solution $Y$ belongs to the domain of the Skorohod integral $\delta^{X}.$ In fact we prove the stronger statement that $Y$ belongs to the Malliavin Sobolev space $\mathbb{D}^{1,2} \left(\mathcal{H}_{1} \right) \subset \mathrm{Dom}(\delta^{X})$ (for definitions see Section \ref{MDsection}). To show that $Y$, a path-valued random variable, can
be understood as a random variable in the Hilbert space $\mathcal{H}_{1}$, we need to identify a class of functions with a subset of $\mathcal{H}_{1}$. This was proved in \cite{cl2018}, by taking advantage of the assumption that $\rho \in \left[1, \frac{3}{2} \right)$, but the less regular cases need a new argument that exploits the structure of the Volterra kernel. To handle the Malliavin derivative $\mathcal{D}Y$, we need a similar result that identifies a class of two-parameter functions as a subset of $\mathcal{H}_{1}\otimes\mathcal{H}_{1}$.

\item For the examples considered in this paper, the Gaussian rough path $\mathbf{X}$ will consist of iterated integrals up to degree three; i.e. $\mathbf{X=}\left(  1,X,\mathbf{X}^{2},\mathbf{X}^{3}\right)  .$ This contrasts with the result in \cite{cl2018}, where only the case $\mathbf{X=} \left(  1,X,\mathbf{X}^{2}\right)$ needs to be considered. This increases the complexity of the the arguments significantly; indeed, the rough integral in the left side of (\ref{corr form}) is now well approximated locally by terms up to third-order
\begin{align*}
\int_{s}^{t}\left\langle Y_{t}\circ\mathrm{d}\mathbf{X}_{t} \right\rangle
\simeq\left\langle Y_{s},X_{s,t}\right\rangle +V\left(  Y_{s}\right) \mathbf{X}_{s,t}^{2} + V^2(Y_{s}) \left(
\mathbf{X}_{s,t}^{3}\right).
\end{align*}
A key step in \cite{cl2018} is the proof that the second-order terms in this approximation satisfy
\begin{align*}
\lim_{\left\Vert \pi(n)\right\Vert \rightarrow 0}\left\Vert \sum_{i:\pi(n)=\left\{  t_{i}^{n}\right\}}V(Y_{t_{i}^{n}})\left(  \mathbf{X}_{t_{i}^{n}, t_{i+1}^{n}}^{2}-\frac{1}{2}\sigma^{2}\left(  t_{i}^{n},t_{i+1}^{n}\right) \mathcal{I}_{d}\right) \right\Vert_{L^{2} (\Omega)} = 0.
\end{align*}
For the present work we need to address the same problem for the third order terms, namely the existence of an $L^{2}(\Omega)$-limit for sums of terms of the form
\begin{align*}
V^2(Y_{t_{i}^{n}}) \left( \mathbf{X}_{t_{i}^{n},t_{i+1}^{n}}^{3}\right),
\end{align*}
possibly after rebalancing, over a sequence of partitions with mesh tending to zero. An important discovery of this paper is the somewhat surprising conclusion that this these terms have vanishing $L^{2}(\Omega)-$limit, without the the need to subtract any rebalancing terms. This is the concluding result of Section \ref{Sko aug}.

\item The proof of point (ii) relies on a rather intricate interplay between estimates from Malliavin's calculus and rough path analysis. From the latter theory, we need estimates on the directional derivatives of RDE solutions. It is well known that an RDE solution of the form \eqref{basic SDE1} can be differentiated in a direction $h\in C^{q-\text{var}}\left(  \left[  0,T\right], \mathbb{R}^{d}\right)$ by considering the perturbed RDE solution driven by the translated rough path $T_{\epsilon h}\mathbf{X}$ and then evaluating the
derivative in $\epsilon$ at zero. For $T_{\epsilon h}\mathbf{X}$ to make sense, $\mathbf{X}$ and $h$ must have Young-complementary regularity, i.e. $\frac{1}{p} + \frac{1}{q} > 1$, in which case Duhamel's formula gives
\begin{align} \label{first o}
D_{g}Y_{t}=\int_{0}^{t}J_{t}^{\mathbf{X}}\left(  J_{s}^{\mathbf{X}}\right)^{-1}V\left(  Y_{s}\right)  \,\mathrm{d}g(s),
\end{align}
a well-defined Young integral. In Malliavin calculus, $g$ will typically be an element of the Cameron-Martin space (written as $\mathcal{H}^{d}$ in this paper), and this has spurred interest in results that prove that $\mathcal{H}^{d}$ can be continuously embedded into $q$-variation spaces, see e.g. \cite{cfv2009}, \cite{fggr2016}. By combining these results with Young's inequality, one can then say e.g. that
\begin{align} \label{embed}
\left\vert D_{g}Y_{t}\right\vert \lesssim\left\vert g\right\vert_{q\text{-var}}\lesssim\left\vert g\right\vert _{\mathcal{H}^{d}},
\end{align}
and these arguments can be generalized to higher order directional derivatives, allowing one control over the Hilbert-Schmidt norm of the Malliavin derivative; see \cite{inahama2014}. Note however, that quality is lost in \eqref{embed} by use of the embedding. For the proof in (ii) we need subtler estimates on the higher order derivatives of the form
\begin{align} \label{high o}
\left\vert D_{g_{i},....,g_{n}}^{n}Y_{t}\right\vert \leq C_{n}\left(\mathbf{X}\right)
{\textstyle\prod\limits_{j=1}^{n}} \left\vert g_{j}\right\vert _{q\text{-var}}.
\end{align}
High order derivatives here complicate matters since derivatives of order $2$ and higher are no longer representable as Young integrals as in \eqref{first o}; instead genuine rough integrals appear. Much of the work underpinning point (ii) goes into deriving closed-form expressions for these high-order derivatives and then estimating them so as to arrive at \eqref{high o}. We must also pay
careful attention to the random variable $C_{n}\left(  \mathbf{X}\right)$ in \eqref{high o} which, for our application, must have finite positive moments of all orders. The first half of Section \ref{Sko aug} is devoted to this material.
\end{enumerate}

The culmination of the these arguments is presented in Section \ref{corr formula}, where we give a set of conditions under which a conversion formula holds for $\int_{0}^{T}\left\langle Y_{t}\circ\mathrm{d}\mathbf{X}_{t}\right\rangle -\delta^{X}\left(  Y\right)$. This formula is reminiscent of the one obtained for the case of second-order rough paths, but there are interesting differences too. Most notably the second term in \eqref{corr form},
\begin{align} \label{non commut}
\int_{\lbrack0,T]^{2}}\mathds{1}_{[0,t)}(s)\mathrm{tr}\left[ J_{t}^{\mathbf{X}}\left(  J_{s}^{\mathbf{X}}\right)  ^{-1}V(Y_{s})-V(Y_{t})\right] \,\mathrm{d}R(s,t),
\end{align}
which exists for $2 \leq p < 3$ as a well-defined 2D Young-Stieltjes integral, can only be identified as an $L^{2}$-limit of a sequence of approximating sums. The difference between the two cases stems from the lack of complementary Young regularity of the integrand and $R.$ Interestingly the integrand, while being continuous on $\left[  0,T\right]^{2}$, is not H\"{o}lder bi-continuous and so we cannot even appeal to the relaxed criteria discussed in point (1) above. It is unknown at present whether the limit is interpretable as a 2D Young-Stieltjes integral. We discuss in detail two important corollaries of our result. The first is where $X$ is a fractional
Brownian motion with $H$ in $\left( \frac{1}{4}, \frac{1}{3} \right]$, and the second is the case where the vector fields defining \eqref{basic SDE1} commute. In this latter case, we show that the second term (\ref{non commut}) in the correction formula disappears and, as a special case, we can recover It\^{o}-type formulas for Gaussian processes, thus connecting our work to a substantial recent corpus e.g. \cite{np1998}, \cite{amn2001}, \cite{ccm2003}, \cite{no2011} and \cite{hjt2013}. \par 

The bulk of the content in this paper can be found in the second-named author's doctoral dissertation \cite{lim2016}.

\section{Preliminaries}

\subsection{Rough path concepts and notation}

We briefly review the basic notation of rough paths theory used in this article; the standard references in this area \cite{lyons98}, \cite{lq2003}, \cite{fh2014} and \cite{fv2010b} can be consulted for more detail. We let $T^{n}\left(  \mathbb{R}^{d}\right)  $ denote the degree $n$ truncated tensor algebra $T^{n}\left(  \mathbb{R}^{d}\right) :=\mathbb{R}\oplus\mathbb{R}^{d}\oplus\cdots\oplus\left(  \mathbb{R}^{d}\right)^{\otimes n}$ equipped with addition and scalar multiplication as defined in the usual fashion. The truncated tensor product of $a=\left(  a^{0},a^{1},\ldots,a^{n}\right) $ (alternatively written as $a^{0}+a^{1}+\ldots+a^{n}$) and $b=\left(  b^{0},b^{1},\ldots,b^{n}\right)\in T^{n}\left(  \mathbb{R}^{d}\right)$ is defined by
\begin{align*}
a\otimes b:=\left(  c^{0},c^{1},\ldots,c^{n}\right), \qquad c^{k}=\sum
_{i=0}^{k}a^{i}\otimes b^{k-i},\quad\forall\,0\leq k\leq n,
\end{align*}
where here we abuse the notation by re-using the same symbol for the tensor product in $\mathbb{R}^{d}$. The unit element is $e=(1,0\ldots,0)$ and the tangent space of $T^{n}(\mathbb{R}^{d})$ at\ $e$ is denoted $A_{T}^{n}\left(\mathbb{R}^{d}\right)$. The exponential and logarithm maps, $\exp:A_{T}^{n}(\mathbb{R}^{d})\rightarrow T^{n}(\mathbb{R}^{d})$ and $\log: T^{n}(\mathbb{R}^{d})\rightarrow A_{T}^{n}(\mathbb{R}^{d}),$ are mutually inverse and defined by
\begin{align*}
\exp(a):=\sum_{i=0}^{n}\frac{a^{\otimes i}}{i!}, \qquad \log(a)=\sum_{i=1}^{n}(-1)^{i+1}\frac{(a-e)^{\otimes i}}{i}.
\end{align*}
The step-$n$ nilpotent group (with $d$ generators), denoted by $G^{n}\left(\mathbb{R}^{d}\right)  $, is the subgroup of $T^{n}\left(  \mathbb{R}^{d}\right)$ corresponding to the sub-Lie algebra of $A_{T}^{n}\left(\mathbb{R}^{d}\right)$ generated by the Lie bracket $[a,b]=a\otimes b-b\otimes a$. We equip it with any symmetric, sub-additive homogeneous norm $\left\Vert \cdot\right\Vert $ (cf. \cite{fv2010a}) which induces the left-invariant metric $d\left(  a,b\right) = \norm{a^{-1} \otimes b}$. Given $\mathbf{x}\in\mathcal{C}\left(  [0,T];G^{n}\left(  \mathbb{R}^{d}\right) \right)$, a continuous $G^{n}\left(  \mathbb{R}^{d}\right)$-valued path, we define the increment over $\left[  s,t\right]  \subseteq\left[  0,T\right]$ by $\mathbf{x}_{s,t}:=\mathbf{x}_{s}^{-1}\otimes\mathbf{x}_{t}$, and then the
$p$-variation distance is
\begin{align} \label{pvarDist}
d_{p-var;[0,T]}(\mathbf{x},\mathbf{y}):=\sup_{\pi}\left( \sum_{i} d(\mathbf{x}_{t_{i},t_{i+1}},\mathbf{y}_{t_{i},t_{i+1}})^{p}\right)^{\frac{1}{p}},
\end{align}
where the supremum runs over all partitions $\pi=\{t_{i}\}$ of $[0, T]$. We let $\left\Vert \mathbf{x}\right\Vert _{p-var;[0,T]}:=d_{p-var;[0,T]} (\mathbf{x},0)$ $,$where $0$ denotes the constant path $\mathbf{y}_{t}=e$ for all $t \in [0, T]$.

\begin{definition}
For $p \geq 1$, the \textit{weakly geometric p-rough paths}, which we will denote by \\ $\mathcal{C}^{p-var}\left([0,T];G^{\lfloor p\rfloor} \left( \mathbb{R}^{d}\right) \right) $, is the set of continuous functions $\mathbf{x}$ from $[0,T]$ onto $G^{\lfloor
p\rfloor} \left( \mathbb{R}^{d} \right)$ such that $\norm{\mathbf{x}}_{p-var; [0, T]} < \infty$.
\end{definition}

The simplest example of a weakly geometric $p$-rough path is as follows, given a bounded-variation path $x$ in $\mathbb{R}^{d}$, we can compute the signature of $x$ in $G^{\lfloor p\rfloor} \left( \mathbb{R}^{d} \right)$: 
\begin{align*}
S_{\lfloor p\rfloor }(x)_{s,t}=\left( 1,\mathbf{x}_{s,t}^{1},\mathbf{x}_{s,t}^{2},\ldots ,\mathbf{x}_{s,t}^{\lfloor p\rfloor }\right),
\end{align*}
where $\mathbf{x}_{s,t}^{k}$ is the conventional $k$-th iterated integral of
the path $x$ over the interval $[s,t]$: 
\begin{align*}
\mathbf{x}_{s,t}^{k} = \sum_{j_1, \ldots, j_k=1}^d \left( \int_{s<r_{1}<\cdots <r_{k}<t}\,\mathrm{d}x^{(j_1)}_{r_{1}}\otimes \cdots \otimes \,\mathrm{d}x^{(j_k)}_{r_{k}} \right) e_{j_1} \otimes \cdots \otimes e_{j_k}.
\end{align*}

\begin{definition}
For $p \geq 1$, the space of \textit{geometric p-rough paths}, which we will denote by \\ $\mathcal{C}^{0,p-var}\left( [0,T]; G^{\lfloor p \rfloor}\left( \mathbb{R}^{d}\right) \right)$, is defined to be the closure of 
\begin{align*}
\mathcal{C}^{\infty} \left( [0, T]; G^{\lfloor p \rfloor} \left(\mathbb{R}^d\right) \right)
:= \left\{ f \in \mathcal{C}^{p-var} \left([0, T]; G^{\lfloor p \rfloor} \left(\mathbb{R}^d\right)\right) \: \Big| \: f \,  \mathrm{smooth} \right\}
\end{align*}
with respect to the topology given by the $p$-variation distance \eqref{pvarDist}.
\end{definition}

The notion of finite $p$-variation extends to (non-necessarily continuous) paths \\
$f:[0,T]\rightarrow E$ taking values in a metric space $(E,d)$ via
\begin{align*}
\left\Vert f\right\Vert _{p-var;[0,T]}:=\sup_{\pi}\left(  \sum_{i}d(f_{t_{i}}, f_{t_{i+1}})_{E}^{p}\right)^{\frac{1}{p}}.
\end{align*}
Letting $V^{p-var}\left( [0,T]; E \right) := \{f \:|\: \left\Vert f\right\Vert_{p-var;[0,T]}<\infty\}$, we use $\mathcal{C}^{p-var}\left(  [0,T];E\right)$ and \\ 
$\mathcal{C}_{pw}^{p-var}\left( [0, T]; E \right) $ to denote the subsets of $V^{p-var}\left( [0,T]; E\right) $ consisting of functions which are also, respectively, continuous functions and piecewise continuous. Given $f$ in \\
$\mathcal{C}^{p-var}\left( [0,T];E\right)$, the function $\omega(s,t):=\left\Vert f\right\Vert _{p-var;[s,t]}^{p}$ defined on the simplex
$\left\{ \left( s, t \right) \:|\: 0 \leq s\leq t\leq T\right\} $ is a \textit{control}, by which we mean it is a continuous, non-negative, function that is super-additive and satisfies $\omega(t,t) = 0$ for all $t \in [0,T]$. When $\left(  E,\left\Vert \cdot\right\Vert _{E}\right) $ is a Banach space we can define a norm on $\mathcal{C}^{p-var}\left(  [0,T];E\right) $ via
\begin{align*}
\left\Vert f\right\Vert _{\mathcal{V}^{p};[0,T]}:=\left\Vert f\right\Vert_{p-var;[0,T]}+\sup_{t \in [0, T]}\left\Vert f_{t}\right\Vert_{E}.
\end{align*}

We will also need the notion of $p$-variation for two-parameter functions. Thus, given $f:[0,T]^{2} \rightarrow E$, we say $f$ is of finite 2D $p$-variation if
\begin{align*} 
\left\Vert f\right\Vert _{p-var;[0,T]^{2}}:=\sup_{\pi}\left(  \sum_{i,j} \left\Vert f
\begin{pmatrix}
u_{i},u_{i+1}\\
v_{j},v_{j+1}
\end{pmatrix} \right\Vert_{E}^{p}\right)  ^{\frac{1}{p}}<\infty,
\end{align*}
where $\pi=\left\{  \left(  u_{i},v_{j}\right)  \right\}  $ is a partition of $[0,T]^{2}$, and the rectangular increment is given by
\begin{align} \label{rectInc}
f \begin{pmatrix}
u_{i},u_{i+1}\\
v_{j},v_{j+1}
\end{pmatrix}
:=f(u_{i},v_{j})+f(u_{i+1},v_{j+1})-f(u_{i},v_{j+1})-f(u_{i+1},v_{j}).
\end{align}
On occasion, we will use the notation
\begin{align*}
f(\Delta_{i},v):=f(u_{i+1},v)-f(u_{i},v)\text{ and }f(u,\Delta_{j} ) := f(u,v_{j+1})-f(u,v_{j}).
\end{align*}

\begin{definition}
Let $f$ and $g$ be functions defined on $[0,T]^{2}.$ We say that the 2D Young-Stieltjes integral of $f$ with respect to $g$ exists if there exists a scalar $I(f,g)\in\mathbb{R}$ such that
\begin{align} \label{quant1}
\lim_{\left\Vert \pi\right\Vert \rightarrow0}\left\vert \sum_{i,j}f\left(
u_{i},v_{j}\right) g
\begin{pmatrix}
u_{i} & u_{i+1}\\
v_{j} & v_{j+1}%
\end{pmatrix}
-I(f,g)\right\vert \rightarrow0,
\end{align}
i.e. for each $\varepsilon > 0$, there exists a $\delta>0$ such that for all partitions $\pi=\{(u_{i},v_{j})\}$ of $[0,T]^{2}$ with $\left\Vert \pi\right\Vert <\delta$, the quantity on the left of \eqref{quant1} is less than $\varepsilon$. In this case, we use $\int_{[0,T]^{2}}f\,\mathrm{d}g$ to denote $I(f,g)$, or $\int_{[s,t]\times\lbrack u,v]}f\,\mathrm{d}g$ whenever we restrict ourselves to any particular subset $[s,t]\times\lbrack u,v]$ of $[0,T]^{2}$.
\end{definition}

\begin{definition}
We say that $f\in\mathcal{C}^{p-var}([s,t]\times\lbrack u,v])$ and $g\in\mathcal{C}^{q-var}([s,t]\times\lbrack u,v])$ have complementary regularity if $p^{-1}+q^{-1}>1$.
\end{definition}

The significance of this definition lies in the following theorem (see \cite{towghi2002}, \cite{fv2010a}), which gives the existence of the Young-Stieltjes integral and Young's inequality in two dimensions; see \cite{lcl2006}, \cite{fh2014}, \cite{fv2010b} for the
one-dimensional version.

\begin{theorem} \label{2Dintegral} 
Let $f \in V^{p-var}([s, t] \times[u, v])$ and $g \in V^{q-var}([s, t] \times[u, v])$ have complementary regularity. We also assume
that $f(s, \cdot)$ and $f(\cdot, u)$ have finite p-variation, and that $f$ and $g$ have no common discontinuities. Then the 2D Young-Stieltjes integral exists and the following Young's inequality holds;
\begin{align} \label{2DYoungIneq}
\begin{split}
\left|  \int_{[s,t] \times[u,v]} f \, \mathrm{d} g\right|  \quad\leq C_{p, q}
\, {\left\vert \kern-0.25ex\left\vert \kern-0.25ex\left\vert f \right\vert
\kern-0.25ex\right\vert \kern-0.25ex\right\vert } \left\|  g\right\| _{q-var,
[s,t] \times[u,v]},
\end{split}
\end{align}
where
\begin{align*}
{\left\vert \kern-0.25ex\left\vert \kern-0.25ex\left\vert f \right\vert
\kern-0.25ex\right\vert \kern-0.25ex\right\vert } = \left|  f(s, u)\right|  +
\left\|  f(s, \cdot)\right\| _{p-var; [u, v]} + \left\|  f(\cdot, u)\right\|
_{p-var; [s, t]} + \left\|  f\right\| _{p-var, [s,t] \times[u,v]}.
\end{align*}

\end{theorem}

\subsection{Gaussian rough paths}
We will work with a stochastic process
\begin{align*}
X_{t} = \left( X_{t}^{(1)},\ldots, X_{t}^{(d)}\right) , \quad t \in[0, T],
\end{align*}
which denotes a centered (i.e. zero-mean), continuous Gaussian process in $\mathbb{R}^{d}$ with i.i.d. components.

This process is defined on the canonical probability space $\left(\Omega,\mathcal{F},\mathbb{P}\right)  $, where $\Omega=\mathcal{C}\left(
[0,T];\mathbb{R}^{d}\right) $, the space of continuous $\mathbb{R}^d$-valued paths equipped with the supremum topology, $\mathcal{F}$ is the completion of the Borel $\sigma$-algebra generated by $X$, and $\mathbb{P}$ is the unique Borel measure under which $X \left( \omega\right) = \left( \omega_{t}\right)_t \in \left[ 0, T\right] $ has the specified Gaussian distribution. We will use
\begin{align*}
R(s,t):=\mathbb{E}\left[ X_{s}^{(1)}X_{t}^{(1)}\right]
\end{align*}
to denote the covariance function common to the components. The variance $R\left( t, t\right) $ will be denoted simply by $R(t)$, and we will also use the notation
\begin{align*}
\sigma^{2}(s,t):=R
\begin{pmatrix}
s & t\\
s & t
\end{pmatrix}
=\mathbb{E}\left[ \left(  X_{s,t}^{(1)}\right)^{2}\right];
\end{align*}
recall the definition of the rectangular increment in \eqref{rectInc}.

The triple $\left(  \Omega,\mathcal{H}^{d},\mathbb{P} \right)  $ denotes the abstract Wiener space associated to $X$, where $\mathcal{H}^{d}=\bigoplus_{i=1}^{d}\mathcal{H}$ is the Cameron-Martin space (or reproducing kernel Hilbert space). The Cameron-Martin space, which is densely and continuously embedded in $\Omega\text{,}$ is the completion of the linear span of the functions
\begin{align*}
\left\{  R(t, \cdot)^{(u)} := R(t,\cdot) \, e_{u} \; \bigg\vert \; t\in[0,T],
\; u = 1, \ldots, d \right\}
\end{align*}
under the inner-product
\begin{align*}
\left\langle R(t, \cdot)^{(u)}, R(s,\cdot)^{(v)} \right\rangle _{\mathcal{H}^{d}} = \delta_{uv} \, R(t,s), \quad u, v = 1, \ldots, d.
\end{align*}
By definition, $\mathcal{H}^{d}$ satisfies the following reproducing property; for any $f=\left(  f^{(1)},\ldots, f^{(d)}\right)  \in\mathcal{H}^{d}$,
\begin{align*}
\left\langle f_{\cdot}, R(t,\cdot)^{(u)} \right\rangle_{\mathcal{H}^{d}} 
= f_{t}^{(u)}, \quad t \in[0,T].
\end{align*}

We assume that there exists $\rho < 2$ such that $R$ has finite 2D $\rho$-variation. The following theorem in \cite{fv2010a} (see also \cite{cq2002} in the case of fractional Brownian motion) then shows that one can canonically lift $X$ via its piecewise linear approximants $X^{\pi}$ to a geometric $p$-rough path for $p > 2 \rho$.

\begin{theorem} \label{gaussianRP} 
Assume $X$ is a centered continuous $\mathbb{R}^d$-valued Gaussian process with i.i.d. components. Let $\rho \in [1, 2)$ and assume that the covariance function has finite 2D $\rho$-variation.
\begin{enumerate} [(i)]
\item (Existence) There exists a random variable $\mathbf{X}=\left( 1, \mathbf{X}^{1}, \mathbf{X}^2, \mathbf{X}^3\right) $ on $\left( \Omega, \mathcal{F},\mathbb{P}\right) $ which takes values in $\mathcal{C}^{0, p-var} \left( [0, T]; G^3(\mathbb{R}^d) \right)$ for $p > 2\rho$ almost surely, and is hence a geometric $p-$rough path for $p \in (2\rho, 4)$. Moreover, $\mathbf{X}$ lifts the Gaussian process $X$ in the sense that $\mathbf{X}_{s,t}^{1}=X_{t}-X_{s}$ almost surely for all $s,t\in [0, T]$.

\item (Uniqueness and consistency) The lift $\mathbf{X}$ is unique in the sense that it is the $d_{p -v a r}$-limit in $L^{q} (\Omega)$, $q \in [1,\infty )$, of any sequence $S_{^{\left \lfloor p\right \rfloor }} (X^{\pi}) $ with $\left\| \pi\right\| \rightarrow 0$. Furthermore, if $X$ has a.s. sample paths of finite $[1 ,2)$-variation, $\mathbf{X}$ coincides with the
signature of $X$.
\end{enumerate}
\end{theorem}

Moreover, Proposition 17 in \cite{fv2010a} shows that for all $h\in\mathcal{H}^{d}$,
\begin{align} \label{CMembedding}
\left\|  h\right\| _{\rho-var; [0, T]} \leq\left\| h\right\| _{\mathcal{H}^{d}} \sqrt{\left\|  R\right\| _{\rho-var; [0, T]^{2}}},
\end{align}
which implies that $\mathcal{H}^{d} \hookrightarrow\mathcal{C}^{\rho-var}([0 ,T] ;\mathbb{R}^{d})$ whenever $R$ has finite 2D $\rho$-variation. Thus if $\rho\in\left[ 1, \frac{3}{2}\right) $, corresponding to $2 \leq p < 3$, we have complementary regularity between $X$ and any path in the Cameron-Martin space. \par 

In the case $\rho\in\left[  \frac{3}{2}, 2 \right)$, there exists a recent result for complementary regularity between the Cameron-Martin paths and $X$ which requires the following definition.

\begin{definition}
We say that a function $f \in V^{(p, q)-var}([0, T]^{2})$ is of mixed (left)
$(p, q)$-variation if
\begin{align*}
\sup_{\pi} \left(  \sum_{i} \left(  \sum_{j} \left|  f
\begin{pmatrix}
s_{i} & s_{i+1}\\
t_{j} & t_{j+1}%
\end{pmatrix}
\right| ^{p} \right) ^{\frac{q}{p}} \right) ^{\frac{1}{q}} < \infty,
\end{align*}
where the supremum runs over all partitions $\pi:= \left\{  \left(  s_{i}, t_{j} \right)  \right\} $ of $[0, T]^{2}$.
\end{definition}

Theorem 1 in \cite{fggr2016} states that if $R$ is of mixed $(1, \rho)$-variation, then
\begin{align*}
\left\|  h\right\| _{q-var; [0, T]} \leq\left\|  h \right\| _{\mathcal{H}^{d}}
\sqrt{\left\|  R\right\| _{(1, \rho)-var; [0, T]^{2}}},
\end{align*}
where $q = \frac{2\rho}{\rho+ 1}$. One can easily verify that this gives us $\frac{1}{p} + \frac{1}{q} > 1$ as long as $p < 4$.

The following condition collects the assumptions we impose on $X$, or equivalently $\mathbf{X}$.
\begin{condition} \label{newCond1} 
Let $X$ be a continuous, centered Gaussian process in $\mathbb{R}^{d}$ with i.i.d. components, and assume that the covariance
function satisfies
\begin{enumerate} [(a)]
\item $\left\|  R\right\| _{\rho-var; [0, T]^{2}} < \infty$ for some $\rho
\in\left[ 1, 2 \right) $.
\end{enumerate}
For $p \in[1, 4)$, let $\mathbf{X} \in\mathcal{C}^{0, p-var} \left( [0, T]; G^{\lfloor p \rfloor} \left( \mathbb{R}^{d}\right) \right) $ denote the geometric rough path constructed from the limit of the piecewise-linear approximations of $X$. Furthermore, assume that $\mathcal{H}^{d} \hookrightarrow\mathcal{C}^{q-var} \left( [0, T]; \mathbb{R}^{d} \right) $, where $q$ satisfies $\frac{1}{p} + \frac{1}{q} > 1$, i.e. for all $h \in\mathcal{H}^{d}$,
\begin{enumerate} [(a)]
\setcounter{enumi}{1}
\item $\left\|  h\right\| _{q-var; [0, T]} \leq C \left\|  h\right\|
_{\mathcal{H}^{d}}$.
\end{enumerate}
\end{condition}

Later on, we will need to impose further conditions on the covariance function. For all $s, t \in[0 ,T]$, we will assume there exists $C < \infty$ such that
\begin{align} \label{R1DBound}
\left\|  R(t, \cdot) - R (s, \cdot) \right\| _{q-var; [0, T]} \leq C \left|  t - s\right| ^{\frac{1}{\rho}}.
\end{align}
This bound will be used to control the $L^{2} (\Omega)$ norm of the iterated integrals. An immediate consequence of the bound is illustrated in the following lemma (Lemma 2.14 in \cite{cl2018}).

\begin{lemma} \label{RdotRhoVar} 
Let $X$ be a continuous, centered Gaussian process in $\mathbb{R}$ and assume its covariance function satisfies
\begin{align*}
\left\|  R (t , \cdot) - R (s , \cdot)\right\| _{q-var; [0, T]} \leq C \left|
t -s\right| ^{\frac{1}{\rho}}, \quad\forall s < t \in[0, T],
\end{align*}
for some $q, \rho\geq1$. Then

\begin{enumerate} [(i)]
\item $R(t) := R(t, t)$ is of bounded $\rho$-variation.

\item For $p > 2 \rho$, $X$ has a $\frac{1}{p}$-H\"{o}lder continuous modification.
\end{enumerate}
\end{lemma}

\subsection{Volterra processes and fractional Brownian motion}
A \textit{Volterra kernel }$K$ is a square-integrable function $K:\left[ 0,T\right]^2 \rightarrow \mathbb{R}$ such that $K(t,s) = 0 \;\forall s\geq t$. Associated with any Volterra kernel is a lower triangular, Hilbert-Schmidt operator $\mathbb{K}: L^{2}\left(  \left[  0,T \right] \right) \rightarrow L^{2}\left(  \left[  0,T\right]  \right)  $ given by
\begin{align*}
\mathbb{K} \left(  f\right)  \left(  \cdot\right)  =\int_{0}^{T}K\left(  \cdot,s\right) f\left(  s\right)  \text{ for all }f\in L^{2}\left(  \left[  0,T\right] \right) .
\end{align*}
Given a standard Brownian motion $B$ and a Volterra kernel $K,$we define a Volterra process $X = \left( X_{t} \right)_{t\in\left[ 0,T\right] }$ as the It\^{o} integral 
\begin{align} \label{voltrep}
X_{t}=\int_{0}^{T}K(t,s)\,\mathrm{d}B_{s};
\end{align}
this is a centered Gaussian process with covariance function
\begin{align*}
R(s,t) = \int_{0}^{t\wedge s}K(t,r)K(s,r)\,\mathrm{d}r.
\end{align*}

\begin{example}
\mbox{}
\begin{enumerate} [(i)]
\item Standard fractional Brownian motion $B^{H}$, with Hurst parameter $H\in(0,1)$, is the centered Gaussian process with covariance function
\begin{align} \label{covfbm}
R\left(  s,t\right)  =\frac{1}{2}\left(  s^{2H}+t^{2H}-\left\vert t-s \right\vert^{2H}\right).
\end{align}
It is well-known that $B^{H}$ has a Volterra representation of the form \eqref{voltrep} where the kernel can be expressed as a particular hypergeometric function; cf. \cite{du1999}.

\item The Riemann-Liouville process with Hurst parameter $H\in\left( 0, 1 \right) $ is determined by the kernel $K(t,s) := C_{H}(t-s)^{H-\frac{1}{2}} \mathds{1}_{[0, t)}(s)$. Like the fractional Brownian motion, it is a self-similar process with variance $t^{2H}$; however, it does not have stationary increments.
\end{enumerate}
\end{example}

We will need the following condition on the kernel $K$.
\begin{condition} \label{amnCond} 
There exists constants $C < \infty$ and $\alpha \in \left[ 0, \frac{1}{4} \right)$ such that
\begin{enumerate} [(i)]
\item $\left|  K(t,s) \right|  \leq C s^{-\alpha} (t - s)^{-\alpha}$ for all $0 < s < t \leq T$.
\item $\frac{\partial K(t,s)}{\partial t}$ exists for all $0 < s < t \leq T$ and satisfies $\left|  \frac{\partial K(t, s)}{\partial t} \right| \leq C \left(  t-s \right)^{-(\alpha + 1)}$.
\end{enumerate}
\end{condition}

The following proposition summarizes the properties of fractional Brownian motion which will be relevant in the sequel.

\begin{proposition} \label{fBMCond} 
Let $B^{H}$ be standard fractional Brownian motion with Hurst index $H\in\left(  \frac{1}{4},1\right)  $, and let $K$ be the
square-integrable kernel associated with it. We have:

\begin{enumerate} [(i)]
\item For any $p>\frac{1}{H}$ the sample paths of $B^{H}$ are almost surely $\frac{1}{p}$-H\"{o}lder continuous. Furthermore, there exists a geometric rough path $\mathbf{X}\in\mathcal{C}^{0,p-var}\left( [0,T];G^{\lfloor p\rfloor}\left(  \mathbb{R}^{d}\right)  \right) $ which is the $d_{p\text{-var }}$-limit of the paths $S_{\left\lfloor p\right\rfloor }\left(  X^{\pi}\right)  $ as $\left\vert \left\vert \pi\right\vert \right\vert \rightarrow 0$.

\item $B^{H}$ satisfies Condition \ref{newCond1} with $\rho=\frac{1}{2H}$ and
\begin{align*}
q = \begin{cases}
\: \frac{2\rho}{\rho+1} \quad \text{if }\frac{1}{4}<H\leq\frac{1}{3},\\
\: \frac{1}{\rho}\wedge1 \quad \text{if }\frac{1}{3} < H \leq 1.
\end{cases}
\end{align*}  

\item If $\frac{1}{4} < H \leq \frac{1}{2}$, then the kernel $K$ satisfies Condition \ref{amnCond} with $\alpha = \frac{1}{2} - H$.

\item The covariance function (\ref{covfbm}) of $B^{H}$ satisfies:
\begin{enumerate}
\item $\left\Vert R(t,\cdot)-R(s,\cdot)\right\Vert _{q-var;[0,T]}
\leq C \left\vert t-s\right\vert ^{\frac{1}{\rho}}$, if $\frac{1}{4}<H\leq\frac{1}{2}$,
\item $R(t)=t^{2H}$ is of bounded variation and thus of finite q-variation for any $q \geq 1$.
\end{enumerate}
\end{enumerate}
\end{proposition}

\begin{proof}
The sample paths of fractional Brownian motion have $H - \varepsilon$-H\"{o}lder regularity for any $\varepsilon > 0$ by Kolmogorov's continuity theorem, and thus have $\frac{1}{p}$-H\"{o}lder regularity for any $p > 2\rho = \frac{1}{H}$. The proof that it has finite 2D $\rho$-variation can be found in \cite{fv2010a}; see also \cite{fv2011}. \par
In the case $1 \leq p < 2$, or $H > \frac{1}{2}$, the geometric rough path is simply $\left( 1, B^H_t \right)$, and for $H \leq \frac{1}{2}$, one can invoke Theorem \ref{gaussianRP} to construct the geometric rough path. Now Condition \ref{newCond1} is satisfied since we have $\mathcal{H}^d \hookrightarrow \mathcal{C}^{q-var} \left([0, T]; \mathbb{R}^d\right)$ from Proposition 17 in \cite{fv2010a} and Theorem 1 in \cite{fggr2016}. Note that the second case, which applies when $\frac{1}{4} < H \leq \frac{1}{3}$, follows from the fact that the covariance function has finite mixed $(1, \rho)$-variation (cf. \cite{fggr2016}). \par
For any $H \in (0, 1)$, we have (see Theorem 3.2 in \cite{du1999})
\begin{align} \label{kBound1}
|K (t,s)| \leq C_{1, H} s^{-\left| H - \frac{1}{2} \right|} (t - s)^{-\left( \frac{1}{2} - H \right)},
\end{align}
for all $0 < s < t \leq T$, and we also have
\begin{align} \label{kBound2}
\frac{\partial K(t, s)}{\partial t} = C_{2, H} \left( \frac{t}{s} \right)^{H - \frac{1}{2}} (t - s)^{- \left(\frac{3}{2} - H \right)};
\end{align}
see \cite{ccm2003} and \cite{nualart2006}. Thus, Condition \ref{amnCond} is satisfied by the kernel as a consequence of the bounds above. \par
Now to prove that
\begin{align*}
\norm{R(t, \cdot) - R (s, \cdot) }_{q-var; [0, T]}  \leq C \left| t - s\right|^{\frac{1}{\rho}},
\end{align*}
when $\frac{1}{4} < H \leq \frac{1}{2}$, we will adopt the method in \cite{fv2011}, and find bounds for
\begin{align*}
\norm{\exptn{B^H_{s,t} B^H_{\cdot}}}_{q-var; [s, t]}, \quad
\norm{\exptn{B^H_{s,t} B^H_{\cdot}}}_{q-var; [0, s]}, \quad \mathrm{and} \;
\norm{\exptn{B^H_{s,t} B^H_{\cdot}}}_{q-var; [t, T]},
\end{align*}
for all $s, t$ in $[0, T]$.
For the first quantity in the preceding line, we use the fact that when $H > \frac{1}{4}$, we have (see \cite{fggr2016})
\begin{align*}
\norm{R}_{1, \rho-var; [s,t]^2} \leq C \abs{t - s}^{2H},
\end{align*}
and
\begin{align*}
\norm{\exptn{B^H_{s,t} B^H_{\cdot}}}_{q-var; [s, t]}
&\leq \norm{R(t, \cdot) - R(s, \cdot)}_{\mathcal{H}} \sqrt{\norm{R}_{1, \rho-var; [s,t]^2}} \\
&\leq C \abs{t - s}^H \abs{t - s}^H = C \abs{t-s}^{2H}.
\end{align*}
Now let $\{r_i\}$ be a partition of $[0, s]$. Then
\begin{align*}
\sum_i \left( \exptn{B^H_{s,t} B^H_{r_i, r_{i+1}}} \right)^q
&\leq \left( \sum_i \abs{\exptn{B^H_{s,t} B^H_{r_i, r_{i+1}}}} \right)^q \\
&= \abs{\exptn{B^H_{s,t}B^H_s} }^q
\end{align*}
since the disjoint increments of fractional Brownian motion have non-positive correlation when $H \leq \frac{1}{2}$. Taking the supremum over all partitions of $[0, s]$, we have
\begin{align*}
\norm{\exptn{B^H_{s,t} B^H_{\cdot}}}_{q-var; [0, s]}
\leq \abs{\exptn{B^H_{s,t} B^H_s}}
\leq \abs{t - s}^{2H},
\end{align*}
where we note that $\abs{\exptn{B^H_{s,t} B^H_u}} < \abs{t - s}^{2H}$ for all $u$ if $H \leq \frac{1}{2}$, cf. Lemma 5 in \cite{nnt2010}.
The bound for $\norm{\exptn{B^H_{s,t} B^H_{\cdot}}}_{q-var; [t, T]}$ is shown in the same manner when we again exploit the fact that the disjoint increments have the same sign.
\end{proof}

Given a Banach space $E$ and a kernel $K$ satisfying Condition \ref{amnCond} for some $\alpha\in\left[  0,\frac{1}{4}\right)  $, we introduce the linear operator $\mathcal{K}^{\ast}$ (see \cite{amn2001}, \cite{dfond2005})
\begin{align} \label{kStarDefn}
\left(  \mathcal{K}^{\ast}\phi\right)  (s):=\phi\left(  s\right)
K(T,s)+\int_{s}^{T}\left[  \phi\left(  r\right)  -\phi\left(  s\right)
\right]  K(\mathrm{d}r,s),
\end{align}
where the signed measure $K(\mathrm{d}r,s):=\frac{\partial K(r,s)}{\partial r}\,\mathrm{d}r$. The domain $D\left( \mathcal{K}^* \right)$ of $\mathcal{K}^*$ consists of measurable functions $\phi:\left[0, T \right] \rightarrow E$ for which the integral on the right-hand side
exists for all $s$ in $\left[  0,T\right] $. In particular, if $\phi$ is a $\lambda$-H\"{o}lder continuous function in the norm of $E$ for some $\lambda > \alpha$, then $\phi \in D\left( \mathcal{K}^* \right) $ and $\mathcal{K}^* \phi$ is in $L^{2}([0,T];E)$. Note also that for any $a$ in $[0,T]$, $\phi\mathds{1}_{[0,a)}$ is in $D\left( \mathcal{K}^* \right) $ whenever $\phi$ is, and we have the identity
\begin{align} \label{truncatedKStar}
\mathcal{K}^* \left( \phi\mathds{1}_{[0,a)}\right) (s)=\mathds{1}_{[0,a)}(s)\left(  \phi(s)K(a,s)+\int_{s}^{a}\left[ \phi(r)-\phi(s)\right] K(\mathrm{d}r,s)\right).
\end{align}

\subsection{Malliavin calculus} \label{MDsection} 
We will primarily work with the following Hilbert space which is isomorphic to $\mathcal{H}^{d}$.

\begin{definition}
Let $\mathcal{H}_{1}^{d}$ denote the completion of the linear span of
\begin{align*}
\left\{  \mathds{1}_{[0,t)}^{(u)}(\cdot):=\mathds{1}_{[0,t)}(\cdot
)\,e_{u}\;\bigg\vert\;t\in\lbrack0,T],\;u=1,\ldots,d\right\}
\end{align*}
(cf. \cite{amn2001}, \cite{nualart2006}) with respect to the inner-product given by
\begin{align*}
\left\langle \mathds{1}_{[0,t)}^{(u)}(\cdot),\mathds{1}_{[0,s)}^{(v)}
(\cdot)\right\rangle _{\mathcal{H}_{1}^{d}}=\delta_{uv}\,R(t,s),
\end{align*}
where $\delta$ denotes the Kronecker delta. Furthermore, let $\Phi:\mathcal{H}_{1}^{d}\rightarrow\mathcal{H}^{d}$ denote the Hilbert space
isomorphism obtained from extending the map $\mathds{1}_{[0,t)}^{(u)}(\cdot) \mapsto R(t,\cdot)^{(u)}$, $t \in [0, T],\;u=1,\ldots,d$.
\end{definition}

We record some basic properties about the Malliavin calculus. For simplicity, we assume here that $d=1;$ the case of $d \in \mathbb{N}$ case needs only minor modifications. First we recall that the map $\mathds{1}_{\left[ 0, t\right) } (\cdot) \mapsto X_{t}$ extends to a unique linear isometry $I$ from $\mathcal{H}_{1}$ to $L^{2}\left(  \Omega\right) $. It follows that $I\left(  h\right)  $ is a mean-zero Gaussian random variable with variance $\left\Vert h\right\Vert_{\mathcal{H}_{1}}^{2}$.

Given a smooth function $f: \mathbb{R}^{n}\rightarrow\mathbb{R}$ of at most polynomial growth, the Malliavin derivative $\mathcal{D}F$ of the functional $F=f\left(  I\left(  h_{1}\right)  ,\ldots,I\left(  h_{n}\right)  \right) $ is the $\mathcal{H}_{1}$-valued random variable given by
\begin{align*}
\mathcal{D}F:=\sum_{i=1}^{n}\frac{\partial f}{\partial x_{i}}\left(  I\left(
h_{1}\right)  ,\ldots,I\left(  h_{n}\right)  \right)  h_{i}.
\end{align*}
We let $\mathbb{D}^{1,2}$ denote the Hilbert space that arises from completing
this subspace of cylinder functionals with respect to
\begin{align*}
\left\Vert F\right\Vert _{1,2}^{2}:=\left\Vert F\right\Vert _{L^{2}\left(
\Omega\right)  }^{2}+\left\Vert \mathcal{D}F\right\Vert _{L^{2}\left(
\Omega;\mathcal{H}_{1}\right)  }^{2},
\end{align*}
whereupon $D$ extends to a bounded linear operator from $\mathbb{D}^{1,2}$ to $L^{2}\left(  \Omega;\mathcal{H}_{1}\right)  .$ More generally, Banach spaces $\mathbb{D}^{1,p}$ $\ $for $p>1$ can be defined by replacing $2$ above with $p.$ For any $F$ in $\mathbb{D}^{1,2}$, we let $\mathcal{D}_{h}F:=\left\langle \mathcal{D}F,h\right\rangle _{\mathcal{H}_{1}}$. The divergence operator
$\delta^{X}$ is the $L^{2}\left(  \Omega\right)  $-adjoint of $D$, and its domain Dom$\left(  \delta^{X}\right)  $ consists of those $G$ in $L^{2}\left( \Omega;\mathcal{H}_{1}\right)  $ for which there exists a $c>0$ with
\begin{align*}
\left\vert \mathbb{E}\left[  \left\langle DF,G\right\rangle _{\mathcal{H}_{1}%
}\right]  \right\vert \leq c\left\Vert F\right\Vert _{L^{2}\left(
\Omega\right) }, \quad \forall F\in\mathbb{D}^{1,2}.
\end{align*}
For such $G$, the Riesz representation theorem then provides that $\delta^{X}\left(  G\right)  $ is characterized by
\begin{align*}
\mathbb{E}\left[ \left\langle DF,G\right\rangle_{\mathcal{H}_{1}}\right]
= \mathbb{E}\left[ F\delta^{X}\left(  G\right) \right].
\end{align*}
Analogous definitions for $DF$ apply when $F$ is an $E-$valued random variable, for any separable Hilbert space $E$, in which case we will denote the Sobolev spaces by $\mathbb{D}^{1,p}\left(  E\right) $. Higher order derivatives $D^{n}F$ and their corresponding Sobolev spaces $\mathbb{D}^{n,p}$ can then be defined iteratively, and the operator $\delta_{n}^{X}:$ Dom$\left(  \delta_{n}^{X}\right)  \subset L^{2}\left( \Omega; \mathcal{H}_{1}^{\otimes n}\otimes E\right)  $ $\rightarrow L^{2}\left(  \Omega\right) $ can be defined as the adjoint of $D^{n}F$ as above. We will use the notation $\delta^{X}(h)$ and $\int_{0}^{T}h_{s}\,\mathrm{d}X_{s}$ interchangeably to denote the divergence
operator. It is well-known that the domain of $\delta^{X} $ contains $\mathbb{D}^{1,2}\left(  \mathcal{H}_{1}\right) $, see e.g. Proposition 1.3.1 in \cite{nualart2006}.

For deterministic $h\in\mathcal{H}_{1}$ we notice that $\delta^{X}\left( h\right)  =I\left(  h\right)  $ as introduced above. More generally, by fixing a multi-index $a=(a_{1},\ldots,a_{M})$ where $\left\vert a\right\vert :=\sum_{i=1}^{M}a_{i}=n$, we can define $I_{n}:\mathcal{H}_{1}^{\otimes n}\rightarrow L^{2}\left(  \Omega\right)  $ as follows:
\begin{align*}
I_{n}\left(  h_{1}^{\otimes a_{1}}\otimes\cdots\otimes h_{M}^{\otimes a_{M}}\right)  =a!\prod_{i=1}^{M}H_{a_{i}}(\delta^{X}(h_{i})),
\end{align*}
where $a!:=\prod_{i=1}^{M}a_{i}!$ and $H_{m}(x)$ denotes the $m^{th}$ Hermite polynomial. Again for deterministic $h\in\mathcal{H}_{1}^{\otimes n} $ we have that $I_{n}\left(  h\right)  =\delta_{n}^{X}\left(  h\right) $. In particular, we have the duality formula:
\begin{align} \label{dualityFormula}
\mathbb{E}\left[  FI_{n}(h)\right]  =\mathbb{E}\left[  \left\langle D^{n}F,h\right\rangle _{\mathcal{H}_{1}^{\otimes n}}\right] .
\end{align}
For $f\in\mathcal{H}_{1}^{\otimes n}$, $g\in\mathcal{H}_{1}^{\otimes m}$, both $f$ and $g$ symmetric, we also have the following product formula (cf. Proposition 1.1.3 in \cite{nualart2006})
\begin{align} \label{productFormula}
I_{n}(f)I_{m}(g)=\sum_{r=0}^{n\wedge m}r!
\begin{pmatrix}
n\\
r
\end{pmatrix}
\begin{pmatrix}
m\\
r
\end{pmatrix}
I_{n+m-2r}\left(  f\tilde{\otimes}_{r}g\right) .
\end{align}
Here $f\tilde{\otimes}_{r}g$ denotes the symmetrization of the tensor $f\otimes_{r}g$, which in turn denotes the contraction of $f$ and $g$ of order $r$ \cite{nnt2010}; i.e. given any orthonormal basis $\{h_{m}\}$ of $\mathcal{H}_{1}$,
\begin{align*}
f\otimes_{r}g:=\sum_{k_{1},\ldots,k_{r}=1}^{\infty}\left\langle f,h_{k_{1}%
}\otimes\cdots\otimes h_{k_{r}}\right\rangle _{\mathcal{H}_{1}^{\otimes r}%
}\otimes\left\langle g,h_{k_{1}}\otimes\cdots\otimes h_{k_{r}}\right\rangle
_{\mathcal{H}_{1}^{\otimes r}}\in\mathcal{H}_{1}^{\otimes(n+m-2r)}.
\end{align*}

\subsection{Rough differential equations}

In this paper we will focus on the RDEs with time-homogeneous vector fields driven by a Gaussian geometric rough path
\begin{align} \label{RDE}
\mathrm{d}Y_{t}=V(Y_{t})\circ\mathrm{d}\mathbf{X}_{t},\quad Y_{0} = y_{0} \in \mathbb{R}^{e},
\end{align}
where $\left\{ V_{1}, \ldots, V_{d} \right\} $ denotes a collection of $\mathbb{R}^e$-valued vector fields which will always be at least continuously differentiable. Recall also from Theorem 2.25 in \cite{cl2018} that for all $s, t \in [0, t]$, $\norm{Y}_{p-var; [s, t]}$ is in $L^q(\Omega)$ for all $q > 0$. \par 

For $h_{1}, \ldots, h_{n} \in\mathcal{H}_{1}^{d}$, we can take the directional derivatives of $Y_{t}$ in the directions \newline$\Phi(h_{1}), \ldots, \Phi(h_{n})$ in $\mathcal{H}^{d} \subset\mathcal{C}^{q-var} \left(  \left[
0,T \right] , \mathbb{R}^{d}\right) $ by setting
\begin{align} \label{nderiv}
\mathcal{D}_{h_{1}, \ldots, h_{n}}^{n} Y_{t} 
:= \frac{\partial^{n}}{\partial\varepsilon_{1} \ldots\partial\varepsilon_{n}}
Y_{t}^{\varepsilon_{1}, \ldots, \varepsilon_{n}} \bigg\vert_{\varepsilon_{1} 
= \ldots = \varepsilon_{n}=0},
\end{align}
where $Y_{t}^{\varepsilon_{1}, \ldots, \varepsilon_{n}}$ solves
\begin{align*}
\mathrm{d} Y^{\varepsilon_{1}, \ldots, \varepsilon_{n}}_{t} = V(Y^{\varepsilon
_{1}, \ldots, \varepsilon_{n}}_{t}) \circ\mathrm{d} \left(  T_{\varepsilon_{1}
\Phi(h_{1}) + \cdots+ \varepsilon_{n} \Phi(h_{n})} \mathbf{X} \right) _{t}.
\end{align*}
Here $T_{\Phi(h)} \mathbf{X}$ denotes the rough path translation of $\mathbf{X}$ by $\Phi(h)$ (see \cite{cf2011}), which is well-defined via Young-Stieltjes integration since $\frac{1}{p} + \frac{1}{q} > 1$. The path \eqref{nderiv} again has finite $p$-variation and can be written as a sum of rough integrals and/or Young-Stieltjes integrals; e.g. when $n=1$ the first-order derivative is given by (cf. \cite{fv2010b}, \cite{cf2011})
\begin{align}
\label{dd}\mathcal{D}_{h} Y_{t} = \int_{0}^{t} J_{t}^{\mathbf{X}} \left(
J_{s}^{\mathbf{X}}\right) ^{-1} V \left(  Y_{s} \right)  \, \mathrm{d}
\Phi(h)(s),
\end{align}
and explicit formulas in the cases $n \geq2$ were derived in \cite{cl2018}. Here $J^{\mathbf{X}}_{t}$ denotes the Jacobian of the flow map $y_{0} \rightarrow Y_{t}$ and satisfies
\begin{align} \label{jacobianRDE}
\mathrm{d} J^{\mathbf{X}}_{t} = \nabla V (Y_{t}) \left( \circ\mathrm{d} \mathbf{X}_{t} \right)  J^{\mathbf{X}}_{t}, \quad
J^{\mathbf{X}}_{0} = \mathcal{I}_{e}.
\end{align}

To bound the Jacobian, we will define
\begin{align} \label{NDefn}
N^{\mathbf{X}}_{\beta;[s,t]} := \sup\left\{  n\in\mathbb{N} \cup\{0\} \: \big\vert \: \tau_{n}(\beta)<t\right\}, \quad s, t \in [0, T], \, \beta > 0,
\end{align}
where $\{\tau_{i}(\beta)\}$ is the "greedy sequence" (see \cite{cll2013}) given by
\begin{align*}
& \tau_{0}(\beta)=s,\\
& \tau_{i+1}(\beta)=\inf\left\{  u\in(\tau_{i},t] \: \big\vert \: \left\Vert
\mathbf{X}\right\Vert _{p-var;[\tau_{i},u]}^{p}\geq\beta\right\} \wedge t.
\end{align*}
We then have the following theorem (Theorem 2.27 in \cite{cl2018}).

\begin{theorem} \label{JBoundThm} 
Let $p \geq 1$. Then for all $s < t \in[0, T]$, $\left\| J^{\mathbf{X}}\right\|_{p-var; [s, t]}$ is in $L^{q}(\Omega)$ for all $q > 0$.
\end{theorem}

\begin{proof}
Using the fact that $N_{1; [s, t]}^{\mathbf{X}}$ has Gaussian tails (see Theorem 6.3 in \cite{cll2013}), we see that $\exptn{\exp \left( C_2 q N_{1; [s, t]}^{\mathbf{X}} \right)} < \infty$ for all $q > 0$, $s < t \in [0, T]$. Now from equation (4.10) in \cite{cll2013}, we have the bound
\begin{align} \label{JBound}
\norm{J^{\mathbf{X}}}_{p-var; [s, t]} \leq C_1 \norm{\mathbf{X}}_{p-var; [s, t]} \exp \left( C_2 N^{\mathbf{X}}_{1; [s, t]} \right).
\end{align}
The statement of the theorem then follows immediately using Cauchy-Schwarz since \\
$\norm{\mathbf{X}}_{p-var; [s, t]}$ also has moments of all orders.
\end{proof}

\begin{remark}
Note that as in \cite{cl2018}, we will abuse the notation and write $\mathcal{D}_{h_{1}, \ldots, h_{n}}^{n} Y_{t}$ and $\mathcal{D}_{\Phi(h_{1}), \ldots, \Phi(h_{n})}^{n} Y_{t}$ interchangeably.
\end{remark}
\section{Convergence in $\mathbb{D}^{1, 2} \left( \mathcal{H}_1^d\right)$} \label{various}

In this section, we will discuss the various isomorphisms and subspaces of the Cameron-Martin space and its tensor product. The motivation is as follows: let $Y$ be a solution to RDE \eqref{RDE} and given a partition $\pi = \{ r_i \}$ of $[0, T]$, denote
\begin{align*}
Y^{\pi} (t) := \sum_i Y_{r_i} \mathds{1}_{\left[ r_i, r_{i+1} \right)} (t).
\end{align*}
Now recall the following inequality from Proposition 1.3.1 in \cite{nualart2006} 
\begin{align} \label{itoskoro}
\exptn{\delta^X\left(Y^{\pi} - Y \right)^2}
\leq \exptn{\norm{ Y^{\pi} - Y }^2_{\mathcal{H}_1^d}} +
\exptn{\norm{ \mathcal{D} Y^{\pi} - \mathcal{D} Y }^2_{\mathcal{H}_1^d \otimes \mathcal{H}_1^d}},
\end{align}
which in particular implies that Dom$\left(  \delta^{X}\right) \supseteq \mathbb{D}^{1,2}(\mathcal{H}_{1}^{d})$. \par 

Thus if we can show that almost surely, $Y$ and $DY$ can be identified as elements of $\mathcal{H}_1^d$ and $\mathcal{H}_1^d \otimes \mathcal{H}_1^d$ respectively, and furthermore $\norm{Y^{\pi} - Y}_{\mathcal{H}_1^d} $ and $\norm{\mathcal{D}Y^{\pi} - \mathcal{D}Y}_{\mathcal{H}_1^d \otimes \mathcal{H}_1^d}$ both vanish as $\norm{\pi} \rightarrow 0$, then with further integrability assumptions one can use \eqref{itoskoro} and dominated convergence to show that $\delta^X (Y^{\pi})$ converges to $\delta^X (Y)$ in $L^2(\Omega)$. 

\subsection{Convergence in $\mathcal{H}_{1}^{d}$}
This main aim of this subsection is to investigate the (almost sure) regularity required of $Y$ to identify it as an element of $\mathcal{H}_1^d$, and to have $\norm{Y^{\pi} - Y}_{\mathcal{H}_1^d} \rightarrow 0$. For Volterra processes, the first issue is to find criteria ensuring that the step-function approximations to a given H\"{o}lder continuous function converge in $\mathcal{H}_{1}^{d}$. We recall the following result from \cite{lim2018} (see also \cite{amn2000}).

\begin{proposition} \label{nualartProp} 
Let $\left(  E,\left\Vert \cdot\right\Vert _{E}\right)$ be a Banach space and $K:\left[  0,T\right]  ^{2}\rightarrow\mathbb{R}$ be a kernel satisfying Condition \ref{amnCond} for some $\alpha\in\left[ 0,\frac{1}{4}\right)  $. Let $\phi:[0,T]\rightarrow E$ be $\lambda$-H\"{o}lder continuous, i.e. there exists $C<\infty$ such that
\begin{align*}
\left\Vert \phi(t_{1})-\phi(t_{2})\right\Vert _{E}\leq C\,|t_{1}
-t_{2}|^{\lambda},\quad\forall t_{1},t_{2} \in [0, T],
\end{align*}
and for any partition $\pi=\{s_{i}\}$ of $[0,T]$, let $\phi^{\pi
}:[0,T]\rightarrow E$ denote
\begin{align*}
\phi^{\pi}(t)=\sum_{i}\phi(s_{i})\mathds{1}_{[s_{i},s_{i+1})}(t).
\end{align*}
Then if $\lambda>\alpha$ we have
\begin{align*}
\lim_{\left\Vert \pi\right\Vert \rightarrow0}\int_{0}^{T}\left\Vert
\mathcal{K}^{\ast}\left(  \phi^{\pi}-\phi\right)  (t)\right\Vert _{E}
^{2}\,\mathrm{d}t=0,
\end{align*}
where $\mathcal{K}^*$ is defined as in \eqref{kStarDefn}.
\end{proposition}

Rather than dealing with the Hilbert space $\mathcal{H}_{1}^{d}$ as an abstract completion, it will be useful to realize it as a closed subspace of an $L^{2}$ space. To this end, we define $\mathcal{H}_{2}^{d}$ to be the closure in $L^{2}([0,T];\mathbb{R}^{d})$ of the linear subspace generated by
\begin{align*}
\left\{  K(t,\cdot)^{(u)}:=K(t,\cdot)\,e_{u}\;\bigg\vert\;t\in\lbrack
0,T],\;u=1,\ldots,d\right\}.
\end{align*}
The inner-product is the usual one in $L^{2}([0,T];\mathbb{R}^{d})$, namely $\left\langle f, g \right\rangle_{\mathcal{H}_{2}^{d}} = \int_{0}^{T} \left\langle f_{s}, g_{s} \right\rangle \wrt{s}$ where $\left\langle \cdot, \cdot \right\rangle$ denotes the Euclidean inner product in $\mathbb{R}^{d}$. The following proposition is more or less immediate.

\begin{proposition} \label{h1 and h2}
$\mathcal{H}_{1}^{d}$ and $\mathcal{H}_{2}^{d}$ are isomorphic as Hilbert spaces
\end{proposition}

\begin{proof}
Since
\begin{align*}
R(s, t) = \int_0^{s \wedge t} K(s, r) K(t, r) \wrt{r},
\end{align*}
we have
\begin{align*}
\left\langle K(t, \cdot)^{(u)}, K(s, \cdot)^{(v)} \right\rangle_{L^2([0, T];
\mathbb{R}^d)} = \delta_{uv} R(s, t) = \left\langle \mathds{1}_{[0, t)}^{(u)} (\cdot), \mathds{1}_{[0, s)}^{(v)} (\cdot) \right\rangle_{\mathcal{H}_1^d}
\end{align*}
for all $s, t \in [0, T], u, v = 1, \ldots, d$. From the definition of $\mathcal{K}^*$ in \eqref{kStarDefn}, we have
\begin{align*}
\mathcal{K}^* \left( \mathds{1}_{[0, t)}^{(u)} (\cdot) \right) (s) = K(t, s)^{(u)}, \quad u = 1, \ldots, d,
\end{align*}
which means that $\mathcal{K}^*$ extends uniquely to a linear isometry from $\mathcal{H}_{1}^d$ onto $\mathcal{H}_{2}^d$ so that
\begin{align} \label{iso3}
\left\langle f, g\right\rangle _{\mathcal{H}_{1}^d} = \left\langle \mathcal{K}^* f, \mathcal{K}^* g \right\rangle_{\mathcal{H}_2^d} \quad \forall f, g\in\mathcal{H}_{1}^d.
\end{align}
\end{proof}

\begin{remark}
In the case of standard Brownian motion the isomorphism $\mathcal{K}^{\ast}$ is the identity operator and $\mathcal{H}_{1}^{d}=\mathcal{H}_{2}^{d} = L^{2}([0,T];\mathbb{R}^{d})$.
\end{remark}

Since the RDE solutions we work with are path-valued, it will be convenient to find subspaces of $\mathcal{H}_{1}^{d}$ whose elements are actual paths. We let
\begin{align*}
\Lambda_{\alpha}^{d}\ :=\bigcup_{\lambda>\alpha}\mathcal{C}_{pw}
^{\lambda-H\ddot{o}l}\left(  [0,T];\mathbb{R}^{d}\right),
\end{align*}
where $\mathcal{C}_{pw}^{\lambda-H\ddot{o}l}\left(  [0,T];\mathbb{R}^{d}\right) $ denotes the space of piecewise $\lambda$-H\"{o}lder continuous functions. By equipping $\Lambda_{\alpha}^{d}$ with the inner product
\begin{align*}
\left\langle f,g\right\rangle _{\Lambda_{\alpha}^{d}}:=\left\langle
\mathcal{K}^{\ast}(f),\,\mathcal{K}^*(g)\right\rangle_{L^{2}\left(
[0,T];\mathbb{R}^{d}\right) },
\end{align*}
whilst suppressing its dependence on $K$ in the notation, the following proposition shows that we can regard $\Lambda_{\alpha}^{d}$ as a dense subspace of $\mathcal{H}_{1}^{d}$.

\begin{proposition} \label{A1eqH1} 
Suppose $K$ is a kernel satisfying Condition \ref{amnCond} for some $\alpha \in \left[0, \frac{1}{4} \right)$. Then $\Lambda_{\alpha}^{d}$ is a dense subspace of $\mathcal{H}_{1}^{d}$, and the inclusion map $i:(\Lambda_{\alpha
}^{d},\,\left\langle \cdot,\cdot\right\rangle _{\Lambda_{\alpha}^{d}})\rightarrow(\mathcal{H}_{1}^{d},\,\left\langle \cdot,\cdot\right\rangle_{\mathcal{H}_{1}^{d}})$ is an isometry.
\end{proposition}

\begin{proof}
Let $f \in \Lambda^d_{\alpha}$ and let $\pi(n) = \left\{ r_i^{(n)}\right\}$ be a sequence of partitions whose mesh vanishes as $n \rightarrow \infty$. We denote
\begin{align*}
f^{\pi(n)} (t) := \sum_i f\left( r^{(n)}_i \right) \mathds{1}_{\left[ r_i^{n}, r_{i+1}^{(n)} \right)} (t),
\end{align*}
and note that for each $n$, $f^{\pi(n)}$ is in $\Lambda_{\alpha}^d \cap \mathcal{H}_1^d$. Moreover, Proposition \ref{nualartProp} tells us that $\norm{\mathcal{K}^* \left( f^{\pi(n)} - f \right)}_{L^2 \left([0, T]; \mathbb{R}^d\right)} \rightarrow 0$. Hence, using \eqref{iso3}, we see that $f^{\pi(n)}$ is Cauchy in $\mathcal{H}^d_1$. We again identify $f$ with the limit of the sequence, and under this identification we have
\begin{align}  \label{LH1norm}
\norm{f}_{\mathcal{H}^d_1}^2 = \norm{\mathcal{K}^*(f)}_{L^2\left([0, T]; \mathbb{R}^d\right)}.
\end{align}
$\Lambda^d_{\alpha}$ contains all the generating functions $\left\{ \mathds{1}^{(u)}_{[0, t)} (\cdot) \right\}$ of $\mathcal{H}^d_1$, and so its closure is $\mathcal{H}^d_1$.
\end{proof}

We recall from \cite{cl2018} a similar result in terms of $p$-variation. In that paper, $\mathcal{H}_{1}^{d}$ was derived from a Gaussian covariance function $R$ which was assumed to be of finite 2D $\rho$-variation, $\rho \in [1, 2)$. It was shown that
\begin{align} \label{Wd}
\mathcal{W}_{\rho}^{d}
:= \bigcup_{q < \frac{\rho}{\rho - 1}} \mathcal{C}_{pw}^{q-var}\left( [0,T];\mathbb{R}^{d}\right),
\end{align}
when equipped with the inner product
\begin{align*}
\left\langle f,g\right\rangle _{\mathcal{W}_{\rho}^{d}}:=\int_{[0,T]^{2}}\left\langle f_{s},\,g_{t}\right\rangle _{\mathbb{R}^{d}}\,\mathrm{d}R(s,t),
\end{align*}
is a dense subspace of $\mathcal{H}_{1}^{d}$ with the inclusion map again being an isometry. In the case when $\lambda > \alpha \wedge \left( 1 - \frac{1}{\rho} \right) $, any $f$ and $g$ belonging to $\mathcal{C}_{pw}^{\lambda-H\ddot{o}l} \left( [0, T]; \mathbb{R}^{d}\right) $ also belong to $\mathcal{W}_{\rho}^{d}\cap\Lambda_{\alpha}^{d}$, and we have
\begin{align*}
\left\langle f,g\right\rangle _{\mathcal{W}_{\rho}^{d}}=\int_{[0,T]^{2}
}\left\langle f_{s},g_{t}\right\rangle _{\mathbb{R}^{d}}\,\mathrm{d}R\left(
s,t\right)  =\int_{0}^{T}\left\langle \mathcal{K}^{\ast}f(r),\mathcal{K}
^{\ast}g(r)\right\rangle _{\mathbb{R}^{d}}\,\mathrm{d}r=\left\langle
f,g\right\rangle _{\Lambda_{\alpha}^{d}}
\end{align*}

The follow figure depicts schematically the relationship between the various subspaces in the case of a Volterra process satisfying Condition \ref{amnCond}. For convenience we have assumed we are in the scalar-valued case $d=1$. 
\begin{figure}[H]
\tikzset{every path/.append style={thick}}
\begin{tikzpicture} [xscale = 0.95, yscale = 0.6]
\draw (0, 0) circle (2.6) (-0.8, 2) node {$\mathcal{W}_{\rho}$};
\draw (1.5, 0) circle (2.6) (2.3, 2) node {$\Lambda_{\alpha}$};
\draw (0.75, 0) circle (3.5) (0.75,3)  node {$\mathcal{H}_1$} (0.75,0) node {$\left\langle \mathds{1}_{[0, t)}(\cdot), \mathds{1}_{[0, s)} (\cdot) \right\rangle_{\mathcal{H}_1}$};
\draw (5, 7.5) circle (3.5) (5,10) node {$\mathcal{H}$} (5, 7.5) node {$\left\langle R(t, \cdot), R(s, \cdot) \right\rangle_{\mathcal{H}}$};
\draw (9.25, 0) circle (3.5) (9.25,3) node {$\mathcal{H}_2$} (9.25,0) node {$\left\langle K(t, \cdot), K(s, \cdot) \right\rangle_{\mathcal{H}_2}$};
\draw (9.25, 0.25) circle (4.1) (9.25, 3.85) node {$L^2\left([0, T]; \mathbb{R} \right)$};
\draw (5, 2.5) node {$R(t, s)$};
\draw [->] ([shift={(5, 2.5)}] 246.5:6) arc[radius=6, start angle=246.5, end angle = 293.5];
\draw [->] ([shift={(5, 2.5)}] 6:6) arc[radius=6, start angle = 6, end angle = 54];
\draw [->] ([shift={(5, 2.5)}] 174:6) arc[radius=6, start angle = 174, end angle = 126];
\draw (5, -4 ) node {$\mathcal{K}^*$} (10.6, 5.75) node {$K$} (-0.6, 5.75) node {$\Phi$};
\draw (4.85, 3) -- (4.85, 7) (5.15, 3) -- (5.15, 7);
\draw (5.45, 2) -- (7.45, 0.3) (5.6, 2.3) -- (7.6, 0.6);
\draw (4.5, 2) -- (2.3, 0.3) (4.35, 2.3) -- (2.15, 0.6);
\end{tikzpicture}
\hspace{1cm}
\caption{}
\end{figure}

Note that $K$ gives an isomorphism from $\mathcal{H}_{2}^{d}$ onto $\mathcal{H}^{d} $ because $R(t, \cdot) = \int_{0}^{T} K(\cdot, r) K(t, r) \,\mathrm{d} r$.

\subsection{Convergence in $\mathcal{H}_{1}^{d}\otimes\mathcal{H}_{1}^{d}$} \label{noCompRegularity} 
The results of the previous subsection allow us to interpret RDE solutions (paths) as $\mathcal{H}_{1}^{d}$-valued random
variables. The Malliavin derivatives of these random variables, when they exist, will take values in $\mathcal{H}_{1}^{d}\otimes\mathcal{H}_{1}^{d}$, and we therefore need similar results which identify suitable function spaces which are subspaces of this tensor product space. We will develop this point in the current subsection. \par 

Throughout, $E$ will denote a general Banach space with norm $\left\Vert \cdot \right\Vert _{E}$. The following operator was defined in \cite{lim2018}.
\begin{definition} \label{kStarTensorOp}
Let $\mathcal{K}^{\ast}\otimes\mathcal{K}^{\ast}$ denote the operator
\begin{align*}
(\mathcal{K}^{\ast}\otimes\mathcal{K}^{\ast})\psi(u,v)
&:= \psi (u,v)K(T,v)K(T,u) + K(T,v)A^{K}\big(\psi(\cdot,v)\big)(u) \\ 
&\qquad+ K(T,u) A^{K} \big(\psi(u,\cdot)\big)(v) + B^{K}(\psi)(u,v),
\end{align*}
where
\begin{align*}
&  A^{K}(\phi)(s):=\int_{s}^{T}\left[  \phi(r)-\phi(s)\right]  K(\mathrm{d}%
r,s)\\
&  B^{K}(\psi)(u,v):=\int_{v}^{T}\int_{u}^{T}\psi%
\begin{pmatrix}
u & r_{1}\\
v & r_{2}%
\end{pmatrix}
K(\mathrm{d}r_{1},u)K(\mathrm{d}r_{2},v),
\end{align*}
which is defined for any measurable function $\psi:[0,T]^{2}\rightarrow E$ for which the
integrals on the right side exist.
\end{definition}

Using Proposition \ref{h1 and h2} and the fact that 
\begin{align} \label{kStarProduct}
\left(  \mathcal{K}^{\ast}\otimes\mathcal{K}^{\ast}\right)  \psi(s,t)=\left(
\mathcal{K}^{\ast}\psi_{1}\right)  (s)\otimes\left(  \mathcal{K}^{\ast}%
\psi_{2}\right) (t)
\end{align}
when $\psi(s, t) = \psi_1(s) \psi_2(t)$, it is also clear that $\mathcal{K}^{\ast}\otimes\mathcal{K}^{\ast}$ maps $\mathcal{H}_{1}^{d}\otimes\mathcal{H}_{1}^{d}$ isometrically onto $\mathcal{H}_{2}^{d} \otimes\mathcal{H}_{2}^{d}$, which is a closed subspace of $L^{2}\left([0,T];\mathbb{R}^{d}\right)  \otimes L^{2}\left(  [0,T];\mathbb{R}^{d}\right)
\cong L^{2}\left(  [0,T]^{2};\mathbb{R}^{d}\otimes\mathbb{R}^{d}\right)$. 

To go beyond product functions in the domain of $\mathcal{K}^* \otimes\mathcal{K}^{\ast}$, we also recall the class of strongly H\"{o}lder bi-continuous functions from \cite{lim2018}.

\begin{definition} \label{biHolderDef}
Let $0 < \lambda \leq 1$. We say that a function $\phi:[0,T]^{2}\rightarrow E$ is strongly $\lambda$-H\"{o}lder bi-continuous in the norm of $E$ (or simply strongly $\lambda $-H\"{o}lder bi-continuous in the case where $E$ is finite-dimensional), if
for all $u_{1},u_{2},v_{1},v_{2}\in\lbrack0,T]$ we have
\begin{align*}
\sup_{v\in\left[  0,T\right]  }\left\Vert \phi(u_{2},v)-\phi(u_{2}
,v)\right\Vert _{E}\leq C\,\left\vert u_{2}-u_{1}\right\vert ^{\lambda}
,\quad\sup_{u\in\left[  0,T\right]  }\left\Vert \phi(u,v_{2})-\phi
(u,v_{1})\right\Vert _{E}\leq C\,\left\vert v_{2}-v_{1}\right\vert ^{\lambda},
\end{align*}
and
\begin{align} \label{biHolder2}
\left\Vert \phi
\begin{pmatrix}
u_{1} & u_{2}\\
v_{1} & v_{2}
\end{pmatrix}
\right\Vert _{E}\leq C\,\left\vert u_{2}-u_{1}\right\vert ^{\lambda}\left\vert v_{2}-v_{1}\right\vert ^{\lambda}.
\end{align}
\end{definition}

The following proposition is one of the main results of \cite{lim2018}. 
\begin{proposition} \label{nualartPropNew} 
Let $\psi:[0,T]^{2}\rightarrow E$ be a function which is strongly $\lambda$-H\"{o}lder bi-continuous in the norm of $E$. For any partition $\{(u_{i},v_{j})\}$ of $[0,T]^{2}$, let $\psi^{\pi}:[0,T]^{2}\rightarrow E$ denote
\begin{align*}
\psi^{\pi}(u,v):=\sum_{i,j}\psi(u_{i},v_{j})\mathds{1}_{[u_{i},u_{i+1}%
)}(u)\mathds{1}_{[v_{j},v_{j+1})}(v).
\end{align*}
In addition, let $\mathcal{K}^{\ast}\otimes\mathcal{K}^*$ denote the operator in Definition \ref{kStarTensorOp}, where the Volterra kernel $K$ satisfies Condition \ref{amnCond} for some $\alpha \in \left[ 0, \frac{1}{4} \right) $.
Then if $\lambda>\alpha$, we have
\begin{align*}
\lim_{\left\Vert \pi\right\Vert \rightarrow0}\int_{[0,T]^{2}}\left\Vert
\left(  \mathcal{K}^* \otimes\mathcal{K}^{\ast}\left(  \psi^{\pi} - \psi\right)  \right) (u,v)\right\Vert _{E}^{2}\,\mathrm{d}u\,\mathrm{d}v,
\end{align*}
and
\begin{align*}
\lim_{\left\Vert \pi\right\Vert \rightarrow0}\int_{0}^{T}\left\Vert \left(
\mathcal{K}^{\ast}\otimes\mathcal{K}^{\ast}\left(  \psi^{\pi}-\psi\right)
\right)  (r,r)\right\Vert _{E}\,\mathrm{d}r=0.
\end{align*}
\end{proposition}
For this paper, the result above, coupled with the fact that $\mathcal{H}_1^d \otimes \mathcal{H}_1^d$ is isomorphic to $\mathcal{H}_2^d \otimes \mathcal{H}_2^d$, shows that the strongly $\lambda$-H\"{o}lder bi-continuous functions are contained in $\mathcal{H}_{1}^{d} \otimes \mathcal{H}_{1}^{d}$ for the class of Volterra kernels we are considering.  For orientation here, contrast this to Proposition \ref{nualartProp}, which showed a similar inclusion in $\mathcal{H}_{1}^{d}$ for the class of $\lambda$-H\"{o}lder continuous functions.

\subsection{The Malliavin derivative and convergence in the tensor norm} \label{mall conv}
Here, we will apply the results of the last subsection to the Malliavin derivatives of RDE solutions.
When $\mathbf{X} \in \mathcal{C}^{0, p-var} \left( [0, T]; G^{\lfloor p \rfloor} \left( \mathbb{R}^d\right) \right)$ satisfies Condition \ref{newCond1}, for all $h \in \mathcal{H}_1^d$, $\Phi(h)$ can be embedded in $\mathcal{C}^{q-var}\left( \left[ 0,T \right]; \mathbb{R}^{d}\right)$ where $\frac{1}{p} + \frac{1}{q} > 1$. Furthermore, the Malliavin derivative of $Y$ satisfying
\begin{align*}
\mathrm{d} Y_t = V(Y_t) \strato{\mathbf{X}_t}, \quad Y_0 = y_0,
\end{align*}
is given by
\begin{align*}
\mathcal{D}_h Y_t = \int_{0}^{t} J_{t}^{\mathbf{X}} \left( J_{s}^{\mathbf{X}}\right)^{-1} V\left( Y_{s} \right) \, \mathrm{d} \Phi(h)(s) 
= \int_{0}^{T} \mathds{1}_{[0,t)} \left( s\right) J_{t}^{\mathbf{X}} \left( J_{s}^{\mathbf{X}}\right)^{-1} V\left( Y_{s} \right) \, \mathrm{d} \Phi(h)(s).
\end{align*}
Denoting
\begin{align} \label{D as path}
\mathcal{D}_{s} Y_{t} = \mathds{1}_{[0, t)} (s) J_{t}^{\mathbf{X}} \left(J_{s}^{\mathbf{X}} \right)^{-1} V\left( Y_{s} \right)
\end{align}
with respect to any partition $\pi = \{ r_i \}$ of $[0, T]$, we will write
\begin{align*}
\mathcal{D}_s Y^{\pi}_t 
&= \sum_i \mathcal{D}_s Y_{r_i} \mathds{1}_{\left[ r_i, r_{i+1} \right)} (t).
\end{align*}
We will proceed to show that 
\begin{enumerate} [(i)]
\item $\mathcal{D} Y^{\pi}$ lies in $\mathcal{H}_1^d \otimes \mathcal{H}_1^d$ almost surely, 
\end{enumerate}
and under suitable regularity assumptions on $\mathcal{D} Y$, we have
\begin{enumerate} [(i)]
\setcounter{enumi}{1}
\item $\norm{\mathcal{D} Y^{\pi} - \mathcal{D}Y}_{\mathcal{H}_1^d \otimes \mathcal{H}_1^d} \rightarrow 0$ as $\norm{\pi} \rightarrow 0$.
\end{enumerate}
Coupled with the results in the previous subsections, this will mean that $Y^{\pi}$ converges to $Y$ in $\mathbb{D}^{1,2} \left(\mathcal{H}_1^d\right)$, and $\delta^X(Y)$ is then the $L^2(\Omega)$ limit of $\delta^X \left( Y^{\pi} \right)$.

A potential problem with \eqref{D as path} is the discontinuity at the diagonal $\{s = t\}$, which prevents it from being H\"{o}lder bi-continuous. The next two propositions show how to handle discontinuities of this form.

\begin{proposition}
Given a Banach space $\left(  E,\left\Vert \cdot\right\Vert _{E}\right) $,
let $\psi:[0,T]^{2}\rightarrow E$ be of the form
\begin{align*}
\psi(u,v)=\mathds{1}_{[0,v)}(u)\tilde{\psi}(u,v),
\end{align*}
where $\tilde{\psi}:[0,T]^{2}\rightarrow E$ is strongly $\lambda$-H\"{o}lder bi-continuous in the norm of $E$. Assume that $K$ is a Volterra kernel which satisfies Condition \ref{amnCond} for some $\alpha \in \left[ 0, \frac{1}{4} \right) $
and let $\mathcal{K}^{\ast}\otimes\mathcal{K}^{\ast}$ be the operator given in Definition \ref{kStarTensorOp}. Then if $\lambda>\alpha$, $\left(  \mathcal{K}^{\ast} \otimes\mathcal{K}^{\ast}\right) \psi$ is in $L^{2}([0,T]^{2};E)$.
\end{proposition}

\begin{proof}
We will investigate the integrability of
\begin{align} \label{kStarTensorOp2}
\begin{split}
\mathcal{K}^* \otimes \mathcal{K}^* \psi (u, v)
&= \psi(u, v) K(T, u) K(T, v) + K(T, v) A^K(\psi(\cdot, v)) (u) \\
&\qquad \qquad+ K(T, u) A^K(\psi(u, \cdot)) (v) + B^K(\psi) (u, v)
\end{split}
\end{align}
in the regions $\{ u < v\}$ and $\{ v < u \}$ separately (ignoring the diagonal as it has zero Lebesgue measure). \par
(i) $u < v$: \par
For the first term on the right of \eqref{kStarTensorOp2} we have
\begin{align*}
\psi(u, v) K(T, u) K(T, v) = \tilde{\psi} (u, v) K(T, u) K(T, v) \in L^2([0, T]^2; E),
\end{align*}
and for the second term, we have
\begin{align*}
&\norm{K(T, v) A^K \big( \psi(\cdot, v) \big) (u)}_E \\
&\qquad \qquad=  \norm{K(T, v) \left( \int_u^v \left[ \tilde{\psi}(r, v) -\tilde{\psi}(u, v) \right] K(\mathrm{d}r, u) - \int_v^T \tilde{\psi}(u, v) K(\mathrm{d}r, u) \right) }_E \\
&\qquad \qquad \leq C \abs{K(T, v)} \left( (v - u)^{\lambda - \alpha} + \left( \frac{1}{(v- u)^{\alpha}} - \frac{1}{(T - u)^{\alpha}} \right) \right) \in L^2([0, T]^2).
\end{align*}
The third term satisfies
\begin{align*}
\norm{ K(T, u) A^K \big( \psi (u, \cdot) \big) (v) }_E
&= \norm{ K(T, u) \int_v^T \left[ \tilde{\psi}(u, r) - \tilde{\psi}(u, v) \right] K(\mathrm{d}r, v) }_E \\
&\leq C \abs{K(T, u)} (T - v)^{\lambda - \alpha} \in L^2([0, T]^2).
\end{align*}
For the fourth term, given $r_1 \in (u, T]$, we have
\begin{align*}
(u, T] \times (v, T] = \big\{ (u, v] \times (v, T] \big\} \bigsqcup \big\{ (v, T] \times (v, r_1] \big\} \bigsqcup \big\{ (v, T] \times (r_1, T] \big\},
\end{align*}
and thus
\begin{align*}
\norm{B^K(\psi) (u, v)}_E
&= \left\| \int_u^v \left( \int_v^T \tilde{\psi} \begin{pmatrix}
u & r_1 \\
v & r_2
\end{pmatrix} K(\mathrm{d}r_2, v) \right) K(\mathrm{d}r_1, u) \right. \\
&\qquad \left. + \int_v^T \left( \int_{r_1}^T \left[ \tilde{\psi}(r_1, r_2) + \tilde{\psi}(u, v) - \tilde{\psi}(u, r_2) \right] K(\mathrm{d}r_2, v) \right) K(\mathrm{d}r_1, u) \right. \\
&\qquad \left. + \int_v^T \left( \int_v^{r_1} \left[ \tilde{\psi}(u, v) - \tilde{\psi}(u, r_2) \right] K(\mathrm{d}r_2, v) \right) K(\mathrm{d}r_1, u) \right\|_E.
\end{align*}
This expression is bounded above by
\begin{align*}
C \left( (v - u)^{\lambda - \alpha} (T - v)^{\lambda - \alpha} + \int_v^T \frac{1}{(r_1 - v)^{\alpha} (r_1 - u)^{\alpha + 1}} \wrt{r_1} + \left( \frac{1}{(v - u)^{\alpha}} - \frac{1}{(T - u)^{\alpha}} \right) \right).
\end{align*}
Since
\begin{align} \label{estL2}
\begin{split}
\int_v^T \frac{1}{ (r_1 - v)^{\alpha} (r_1 - u)^{\alpha + 1}} \wrt{r_1}
&= \int_v^T \frac{1}{(r_1 - v)^{\alpha} (r_1 - u)^{\alpha + \frac{1}{4}} (r_1 - u)^{\frac{3}{4}}} \wrt{r_1} \\
&\leq \frac{1}{(v - u)^{\alpha + \frac{1}{4}}} \int_v^T \frac{1}{(r_1 - v)^{\alpha + \frac{3}{4}}} \wrt{r_1},
\end{split}
\end{align}
and $\alpha < \frac{1}{4}$, the fourth term is also in $L^2([0, T]^2; E)$. \par
(ii) $v < u$: \par
The first two terms on the right of \eqref{kStarTensorOp2} vanish, and the third term obeys the estimate
\begin{align*}
\norm{K(T, u) A^K \big( \psi (u, \cdot) \big) (v) }_E
&= \left\| K(T, u) \int_u^T \tilde{\psi} (u, r) K(\mathrm{d}r, v) \right\|_E \quad (\psi(u,r) = 0 \; \mathrm{when} \; v < r < u) \\
&\leq C \abs{K(T, u)} \left( \frac{1}{(u - v)^{\alpha}} - \frac{1}{(T - v)^{\alpha}} \right),
\end{align*}
and hence it is in $L^2([0, T]^2; E)$. For the fourth term, note that
\begin{align*}
\psi \begin{pmatrix}
u & r_1 \\
v & r_2
\end{pmatrix} = 0 \; \mathrm{when} \; v < r_2 < u,
\end{align*}
and thus we have
\begin{align*}
\norm{B^K(\psi) (u, v)}_E
&\leq \norm{\int_u^T \left( \int_u^{r_2} \left[ \tilde{\psi}(r_1, r_2) - \tilde{\psi} (u, r_2) \right] K(\mathrm{d}r_1, u) \right) K(\mathrm{d}r_2, v) }_E \\
&\qquad+ \norm{\int_u^T \left( \int_{r_2}^T \tilde{\psi} (u, r_2) K(\mathrm{d}r_1, u) \right) K(\mathrm{d}r_2, v)}_E \\
&\leq C \left( \left( \frac{1}{(u - v)^{\alpha}} - \frac{1}{(T - v)^{\alpha}} \right) + \int_u^T \frac{1}{(r_2 - u)^{\alpha} (r_2 - v)^{\alpha + 1}} \wrt{r_2} \right).
\end{align*}
Utilizing \eqref{estL2} again, we see that the fourth term is also in $L^2([0, T]^2; E)$.
\end{proof}

The following proposition then deals with the issue of convergence along discrete approximations.
\begin{proposition} \label{HSprop} 
Let $F$ denote either $\mathbb{R}^{e}$ or $L^{2}(\Omega; \mathbb{R}^{e})$, and let $\psi: [0, T]^{2} \rightarrow F$ be a function of the form $\psi(u, v) = \mathds{1}_{[0, v)}(u) \tilde{\psi}(u,v)$, where $\tilde{\psi}$ is strongly $\lambda$-H\"{o}lder bi-continuous in the norm of $F$. Given a partition $\pi= \{ r_{i} \}$ of $[0, T]$, denote
\begin{align} \label{psiApprox}
\psi^{\pi} (s, t) := \sum_{j} \psi(s, r_{j}) \mathds{1}_{[r_{j}, r_{j+1})}(t).
\end{align}
Moreover, let $\mathcal{K}^{*} \otimes\mathcal{K}^{*}$ be the operator given in Definition \ref{kStarTensorOp}, where the Volterra kernel $K$ satisfies Condition \ref{amnCond} for some $\alpha\in\left[  0, \frac{1}{4} \right) $. Then if $\lambda> \alpha$, we have
\begin{align*}
\int_{[0, T]^{2}} \left\|  \mathcal{K}^{*} \otimes\mathcal{K}^{*} \left(\psi^{\pi} - \psi\right)  (s, t) \right\| ^{2}_{F} \, \mathrm{d} s \, \mathrm{d} t \rightarrow 0.
\end{align*}

\end{proposition}

\begin{proof}
We define
\begin{align*}
h(u, v)
:= \int_0^T \left\langle \mathcal{K}^* \left( \psi (\cdot, u) \right) (s), \, \mathcal{K}^* \left( \psi (\cdot, v) \right) (s) \right\rangle_F \wrt{s},
\end{align*}
and correspondingly,
\begin{align*}
h^{\pi} (u, v)
&:= \int_0^T \left\langle \mathcal{K}^* \left( \psi^{\pi} (\cdot, u) \right) (s), \, \mathcal{K}^* \left( \psi^{\pi} (\cdot, v) \right) (s) \right\rangle_F \wrt{s} \\
&= \sum_{i, j} \left( \int_0^T \left\langle \mathcal{K}^* \left( \psi^{\pi} (\cdot, r_i) \right) (s), \, \mathcal{K}^* \left( \psi^{\pi} (\cdot, r_j) \right) (s) \right\rangle_F \wrt{s} \right) \mathds{1}_{[r_i, r_{i+1})} (u) \mathds{1}_{[r_j, r_{j+1})} (v) \\
&= \sum_{i, j} h(r_i, r_j) \mathds{1}_{[r_i, r_{i+1})} (u) \mathds{1}_{[r_j, r_{j+1})} (v).
\end{align*}
Let $\lambda' := \frac{1}{4} \wedge \lambda$. Since $\alpha < \frac{1}{4}$, $\lambda'$ is greater than $\alpha$, and note that any strongly $\lambda$-H\"{o}lder bi-continuous function is also strongly $\lambda'$-H\"{o}lder bi-continuous. We will begin by first showing that $h(u, v)$ is strongly $\lambda'$-H\"{o}lder bi-continuous. \par
For all $u, v, u_1, u_2, v_1, v_2 \in [0,T]$, we have
\begin{align*}
&\abs{h(u_1, v) - h(u_2, v)}
\leq \left( \int_0^T \left\| \mathcal{K}^* \left( \psi(\cdot, u_1) - \psi(\cdot, u_2) \right) (s) \right\|^2_F \wrt{s} \right)^{\frac{1}{2}} \left( \int_0^T \left\| \mathcal{K}^* \left( \psi(\cdot, v) \right) (s) \right\|^2_F \wrt{s} \right)^{\frac{1}{2}}, \\
&\abs{h(u, v_1) - h(u, v_2)}
\leq \left( \int_0^T \left\| \mathcal{K}^* \left( \psi(\cdot, v_1) - \psi(\cdot, v_2) \right) (s) \right\|^2_F \wrt{s} \right)^{\frac{1}{2}} \left( \int_0^T \left\| \mathcal{K}^* \left( \psi(\cdot, u) \right) (s) \right\|^2_F \wrt{s} \right)^{\frac{1}{2}},
\end{align*}
and $\abs{h \begin{pmatrix}
u_1 & u_2 \\
v_1 & v_2
\end{pmatrix}}$ is bounded above by
\begin{align*}
\left( \int_0^T \left\| \mathcal{K}^* \left( \psi(\cdot, u_1) - \psi(\cdot, u_2) \right) (s) \right\|^2_F \wrt{s} \right)^{\frac{1}{2}} \left( \int_0^T \left\| \mathcal{K}^* \left( \psi(\cdot, v_1) - \psi(\cdot, v_2) \right) (s) \right\|^2_F \wrt{s} \right)^{\frac{1}{2}}.
\end{align*}
Note that for $p \geq 1$, using \eqref{truncatedKStar} and fixing $w \in [0, T]$, we have
\begin{align} \label{estimateE}
\begin{split}
\left\| \mathcal{K}^* \psi (\cdot, w) (s)\right\|_F^p
&= \norm{\tilde{\psi} (s, w) K(w, s) + \int_s^w \left[ \tilde{\psi}(r, w) - \tilde{\psi} (s, w) \right] K(\mathrm{d}r, s) }_F^p \\
&\leq C \, 2^{p-1} \left( \frac{1}{s^{p\alpha} (w - s)^{p\alpha}} + (w - s)^{p (\lambda' - \alpha)} \right).
\end{split}
\end{align}
Since $\alpha < \frac{1}{4}$, $\int_0^T \left\| \mathcal{K}^* \psi (\cdot, w) (s)\right\|_F^p \wrt{s}$ is finite as long as $p \leq 4$. \par
Now, all we have to do is show that
\begin{align} \label{estimateA}
\int_0^T \left\| \mathcal{K}^* \psi(\cdot, w_2) - \psi(\cdot, w_1) (s) \right\|_F^2 \wrt{s} \leq C \, \abs{w_2 - w_1}^{2\lambda'},
\end{align}
for all $w_1, w_2 \in [0, T]$, where without loss of generality, we let $w_1 < w_2$.
Observe that
\begin{align} \label{origExp}
\begin{split}
&\int_0^T \left\| \mathcal{K}^* \left( \psi(\cdot, w_2) - \psi(\cdot, w_1) \right) (s) \right\|^2_F \wrt{s} \\
&\qquad \qquad = \int_0^{w_1} \left\| \mathcal{K}^* \left( \psi(\cdot, w_2) - \psi(\cdot, w_1) \right) (s) \right\|^2_F \wrt{s} + \int_{w_1}^{w_2} \left\| \mathcal{K}^* \left( \psi(\cdot, w_2) - \psi(\cdot, w_1) \right) (s) \right\|^2_F \wrt{s}.
\end{split}
\end{align}
For the first term above, for $s \in [0, w_1)$, we have (using \eqref{truncatedKStar})
\begin{align} \label{expansion1}
\begin{split}
\mathcal{K}^* \left( \psi(\cdot, w_2) - \psi(\cdot, w_1) \right) (s)
&= \left(\psi(s, w_2) - \psi (s, w_1) \right) K(w_2, s) \\
&\qquad+ \int_s^{w_2} \left[ \psi (r, w_2) - \psi (s, w_2) - \psi (r, w_1) + \psi (s, w_1) \right] K(\mathrm{d}r, s) \\
&= \left(\tilde{\psi}(s, w_2) - \tilde{\psi} (s, w_1) \right) K(w_2, s) + \int_s^{w_1}
\tilde{\psi} \begin{pmatrix}
s & r \\
w_1 & w_2
\end{pmatrix} K(\mathrm{d}r, s) \\
&\qquad+ \int_{w_1}^{w_2} \left[ \tilde{\psi} (r, w_2) - \tilde{\psi} (s, w_2) + \tilde{\psi} (s, w_1) \right] K(\mathrm{d}r, s).
\end{split}
\end{align}
Since $\tilde{\psi}$ is strongly $\lambda'$-H\"{o}lder bi-continuous, we have
\begin{align} \label{est1}
\norm{\left( \tilde{\psi}(s, w_2) - \tilde{\psi} (s, w_1) \right) K(w_2, s)}_F
\leq C \abs{w_2 - w_1}^{\lambda'} s^{-\alpha} (w_2 - s)^{-\alpha},
\end{align}
and
\begin{align} \label{est2}
\norm{\int_s^{w_1}
\tilde{\psi} \begin{pmatrix}
s & r \\
w_1 & w_2
\end{pmatrix} K(\mathrm{d}r, s)}_F
\leq C \abs{w_2 - w_1}^{\lambda'} (w_1 - s)^{\lambda' - \alpha}.
\end{align}
For the last integral in \eqref{expansion1}, we let $q_1$ denote $\frac{1}{1 - \lambda'}$ and use H\"{o}lder's inequality to derive
\begin{align} \label{est3}
\begin{split}
\norm{\int_{w_1}^{w_2} \left[ \tilde{\psi} (r, w_2) - \tilde{\psi} (s, w_2) + \tilde{\psi} (s, w_1) \right] K(\mathrm{d}r, s) }_F
&\leq C \abs{w_2 - w_1}^{\lambda'} \left( \int_{w_1}^{w_2} \abs{\pd{K(r, s)}{r}}^{q_1} \wrt{r} \right)^{\frac{1}{q_1}} \\
&\leq C \abs{w_2 - w_1}^{\lambda'} \left( \int_{w_1}^{w_2} \frac{1}{(r - s)^{q_1 (\alpha + 1)}} \wrt{r} \right)^{\frac{1}{q_1}} \\
&\leq C \abs{w_2 - w_1}^{\lambda'} (w_1 - s)^{-(\alpha + \lambda' )}.
\end{split}
\end{align}
Putting estimates \eqref{est1}, \eqref{est2} and \eqref{est3} together, when $s < w_1$ we have
\begin{align} \label{diffEst}
\norm{\mathcal{K}^* \left( \psi(\cdot, w_2) - \psi(\cdot, w_1) \right) (s) }_F
\leq C \abs{w_2 - w_1}^{\lambda'} f(s),
\end{align}
for some $f(s) \in L^2\left([0, T]\right)$ since $\lambda' > \alpha$ and $2(\alpha + \lambda' ) < 1$. This gives
\begin{align*}
\int_0^{w_1} \norm{\mathcal{K}^* \left( \psi(\cdot, w_2) - \psi(\cdot, w_1) \right) (s) }_F^2 \wrt{s}
\leq C \abs{w_2 - w_1}^{2 \lambda'}.
\end{align*}
Returning to the second term in \eqref{origExp}, we let $q_2$ denote $\frac{1}{1 - 2\lambda'}$ and use H\"{o}lder's inequality again to obtain
\begin{align*}
\int_{w_1}^{w_2} \left\| \mathcal{K}^* \left( \psi(\cdot, w_1) - \psi(\cdot, w_2) \right) (s) \right\|^2_F \wrt{s}
\leq |w_2 - w_1|^{2\lambda'} \left( \int_0^T \left\| \mathcal{K}^* \left( \psi(\cdot, w_1) - \psi(\cdot, w_2) \right) (s) \right\|^{2q_2}_F \wrt{s} \right)^{\frac{1}{q_2}}.
\end{align*}
Since $\lambda' < \frac{1}{4}$, we have $2q_2 \leq 4$ and this gives $\left( \int_0^T \left\| \mathcal{K}^* \left( \psi(\cdot, w_1) - \psi(\cdot, w_2) \right) (s) \right\|^{2q_2}_F \wrt{s} \right)^{\frac{1}{q_2}} < \infty$ from \eqref{estimateE}.
Now that we have shown that $h$ is strongly $\lambda'$-H\"{o}lder bi-continuous, we will show that
\begin{align*}
\int_{[0, T]^2} \left\| \mathcal{K}^* \otimes \mathcal{K}^* \left( \psi^{\pi} - \psi \right) (s, t) \right\|^2_F \wrt{s} \wrt{t}
= \int_0^T \left( \mathcal{K}^* \otimes \mathcal{K}^* \left( h^{\pi} - h \right) \right) (t, t) \wrt{t},
\end{align*}
and then invoke Proposition \ref{nualartPropNew} to complete the proof. \par
Let $g(s, t)$ denote $\mathcal{K}^* \left(\psi (\cdot, t) \right) (s)$, and note that $g(s, t) = 0$ when $s \geq t$.
We first compute
\begin{align} \label{pullOutS}
\begin{split}
\mathcal{K}^* \otimes \mathcal{K}^* h (t, t)
&= h(t, t) K(T, t)^2 + K(T, t) A^K\big( h(\cdot, t) \big) (t) + K(T, t) A^K \big( h(t, \cdot) \big) (t) + B^K(h)(t, t) \\
&= \int_0^T \left\langle g(s, t), g(s, t) \right\rangle_F K(T, t)^2 \wrt{s} \\
&\quad+ 2 K(T, t) \int_t^T \left( \int_0^T \left\langle g(s, r) - g(s, t), g(s, t) \right\rangle_F \wrt{s} \right) K(\mathrm{d}r, t) \\
&\quad+ \int_t^T \int_t^T \left( \int_0^T \left\langle g(s, r_1) - g(s, t),  g(s, r_2) - g(s, t) \right\rangle_F \wrt{s} \right) K(\mathrm{d}r_1, t) K(\mathrm{d}r_2, t).
\end{split}
\end{align}
The second term on the right vanishes when $s \geq t$, and when $s < t$, using \eqref{estimateE} and \eqref{diffEst} gives us
\begin{align*}
\abs{\left\langle g(s, r) - g(s, t), g(s, t) \right\rangle_F} \abs{\pd{K(r, t)}{r}}
\leq C \abs{r - t}^{\lambda' - \alpha - 1} \tilde{f}(s)
\end{align*}
for some $\tilde{f}(s) \in L^1 ([0, T])$, and thus we can swap the integral with respect to $s$ outside the integral with respect to $r$.
Similarly, the third term on the right of \eqref{pullOutS} is bounded by
\begin{align*}
C \left( \int_s^T \frac{1}{(r - t)^{\alpha + 1}} \wrt{r} \right)^2
\end{align*}
when $s > t$ since the integrand vanishes when $r_1 \leq s$ or $r_2 \leq s$. Furthermore, when $s < t$, its integrand is bounded by
\begin{align*}
C \abs{r_1 - t}^{\lambda' - \alpha - 1} \abs{r_2 - t}^{\lambda' - \alpha - 1} f^2(s).
\end{align*}
Hence, we can also pull out the integral with respect to $s$, and we get
\begin{align*}
\mathcal{K}^* \otimes \mathcal{K}^* h (t, t)
= \int_0^T \mathcal{K}^* \otimes \mathcal{K}^* \left( \left\langle g(s, \cdot), g(s, \cdot) \right\rangle_F \right) (t, t) \wrt{s}.
\end{align*}
Observe that
\begin{align*}
\mathcal{K}^* \otimes \mathcal{K}^* \left( \left\langle g(s, \cdot), g(s, \cdot) \right\rangle_F \right) (t,t)
&= \left\langle \mathcal{K}^* \left( g(s, \cdot) \right) (t), \mathcal{K}^* \left( g(s, \cdot) \right) (t) \right\rangle_F \\
&= \norm{\mathcal{K}^* (g(s, \cdot))(t)}^2_F,
\end{align*}
where here we use \eqref{kStarProduct}, and Fubini's theorem in the case when $F = L^2 (\Omega; \mathbb{R}^e)$. \par 
Fixing $s$, note that for all $t > s$, $g(s, \cdot)$ is $\lambda'$-H\"{o}lder continuous on $[t, T]$ (with the H\"{o}lder norm depending on $t$) from \eqref{est1}, \eqref{est2} and \eqref{est3}. Thus, $\mathcal{K}^* (g(s, \cdot))(t)$ is well defined for all $t > s$, vanishes when $t < s$, and we can apply Lemma 3.2 in \cite{lim2018} to obtain
\begin{align*}
\mathcal{K}^* \left( g(s, \cdot) \right) (t) = \mathcal{K}^* \otimes \mathcal{K}^* \psi (s, t)
\end{align*}
for all $s \neq t$. This concludes the proof.
\end{proof}
It follows in particular that $\mathcal{H}_{1}^{d}\otimes\mathcal{H}_{1}^{d}$ contains functions $\psi: \left[0, T \right]^{2} \rightarrow \mathbb{R}^{d} \otimes \mathbb{R}^{d}$ of the form $\psi(u,v) = \mathds{1}_{[0,v)}(u)\tilde{\psi}(u,v)$ whenever $\tilde{\psi}$ is strongly $\lambda$-H\"{o}lder bi-continuous.

\begin{proposition} \label{H1tensorH1equiv2}
Let $\psi: [0, T]^2 \rightarrow \mathbb{R}^d$ be of the form $\psi(u, v) = \mathds{1}_{[0, v)}(u) \tilde{\psi}(u,v)$, where $\tilde{\psi}$ is strongly $\lambda$-H\"{o}lder bi-continuous, and let $\mathcal{K}^* \otimes \mathcal{K}^*$ be defined as in Definition \ref{kStarTensorOp}, where the Volterra kernel $K$ satisfies Condition \ref{amnCond} for some $\alpha \in \left[ 0, \frac{1}{4} \right)$. \par 
Then if $\lambda > \alpha$, $\psi$ is an element of $\mathcal{H}_1^d \otimes \mathcal{H}_1^d$, with norm given by
\begin{align} \label{tensorNorm2}
\norm{\psi}_{\mathcal{H}_1^d \otimes \mathcal{H}_1^d}
= \int_{[0, T]^2} \abs{ \mathcal{K}^* \otimes \mathcal{K}^* \psi (s, t) }^2_{\mathbb{R}^d \otimes \mathbb{R}^d} \wrt{s} \wrt{t},
\end{align}
and with $\psi^{\pi}$ defined as in \eqref{psiApprox}, we have
\begin{align} \label{tensorVanishing2}
\norm{\psi^{\pi} - \psi}_{\mathcal{H}_1^d \otimes \mathcal{H}_1^d} \rightarrow 0
\end{align}
as $\norm{\pi} \rightarrow 0$.
\end{proposition}

\begin{proof}
Using the canonical identification
\begin{align} \label{canonID}
A(s) \mathds{1}_{[a, b)} (t) \simeq \sum_{j=1}^d \sum_{k=1}^d a^{(k)}_j (s) e_k \otimes \mathds{1}^{(j)}_{[a, b)} (t), \quad a, b \in [0, T],
\end{align}
it is clear that $\psi^{\pi}$ is a member of $\Lambda^d_{\alpha} \otimes \mathcal{H}_1^d$, and thus lies in $\mathcal{H}_1^d \otimes \mathcal{H}_1^d$ by Proposition \ref{A1eqH1}. Furthermore, $\norm{\psi^{\pi}}^2_{\mathcal{H}_1^d \otimes \mathcal{H}_1^d}$ is equal to
\begin{align*}
&\sum_{k, l} \int_0^T \sum_{j=1}^d \left\langle \mathcal{K}^* \left( \psi_j (\cdot, r_k) \right) (s), \mathcal{K}^* \left( \psi_j(\cdot, r_l) \right) (s) \right\rangle_{\mathbb{R}^d} \wrt{s} \int_0^T \mathcal{K}^* \left( \mathds{1}_{\left[ r_k, r_{k+1} \right)} \right) (t) \mathcal{K}^* \left( \mathds{1}_{\left[ r_l, r_{l+1} \right)} \right) (t) \wrt{t} \\
&\qquad = \sum_{k,l} \int_{[0, T]^2} \left\langle \mathcal{K}^* \otimes \mathcal{K}^* \left( \psi (\cdot, r_k) \mathds{1}_{\Delta_k } (\cdot) \right) (s, t), \mathcal{K}^* \otimes \mathcal{K}^* \left( \psi (\cdot, r_l) \mathds{1}_{\Delta_l } (\cdot) \right) (s, t) \right\rangle_{\mathbb{R}^d \otimes \mathbb{R}^d} \wrt{s} \wrt{t}, \\
&\qquad = \int_{[0, T]^2} \abs{\mathcal{K}^* \otimes \mathcal{K}^* \psi^{\pi} (s, t)}_{\mathbb{R}^d \otimes \mathbb{R}^d}^2 \wrt{s} \wrt{t},
\end{align*}
which we know is Cauchy as $\norm{\pi} \rightarrow 0$ by Proposition \ref{HSprop}. We now take any sequence of partitions $\pi(n)$ with vanishing mesh and identify $\psi$ with the limit of $\psi^{\pi(n)}$ in $\mathcal{H}_1^d \otimes \mathcal{H}_1^d$. Invoking Proposition \ref{HSprop} again then gives us \eqref{tensorNorm2} and \eqref{tensorVanishing2}.
\end{proof}

\subsection{The It\^{o}-Skorohod isometry revisited} \label{itoSkorohodIsometry}
We now give another formulation for the It\^{o}-Skorohod isometry for Volterra processes (see also \cite{ev2003}, where an isometry formula in the specific case of fractional Brownian motion is provided).
\begin{theorem} \label{isometry2}
Let $X$ be a Volterra process which satisfies Condition \ref{newCond1} for some $\rho \in \left[ 1, 2 \right)$, and assume that its kernel satisfies Condition \ref{amnCond} for $\alpha = \frac{1}{2} - \frac{1}{2\rho}$. Given $\lambda > \alpha$, let $Y$ be a random variable which satisfies, almost surely,
\begin{enumerate}[(i)]
\item $Y \in \mathcal{C}^{\lambda-H\ddot{o}l}_{pw} \left([0, T]; \mathbb{R}^d\right)$,
\item $\mathcal{D}Y: [0, T]^2 \rightarrow \mathbb{R}^d \otimes \mathbb{R}^d$ is a function of the form $\mathds{1}_{[0, t)}(s) g(s, t)$, where $g$ is strongly $\lambda$-H\"{o}lder bi-continuous.
\end{enumerate}
Then $\lim_{\norm{\pi} \rightarrow 0} Y^{\pi} = Y$ in $\mathbb{D}^{1, 2} (\mathcal{H}_1^d)$ if and only if
\begin{align*}
\lim_{\norm{\pi} \rightarrow 0} \exptn{\int_0^T \abs{\mathcal{K}^* \left( Y^{\pi} - Y \right) (t) }_{\mathbb{R}^d}^2 \wrt{t}} = 0,
\end{align*}
and
\begin{align*}
\lim_{\norm{\pi} \rightarrow 0} \exptn{\int_{[0, T]^2} \abs{\mathcal{K}^* \otimes \mathcal{K}^* \left( \mathcal{D} Y^{\pi} - \mathcal{D} Y \right) (s, t) }_{\mathbb{R}^d \otimes \mathbb{R}^d}^2 \wrt{s} \wrt{t}} = 0,
\end{align*}
in which case $\lim_{\norm{\pi} \rightarrow 0} \exptn{\delta^X\left(Y^{\pi} - Y \right)^2} = 0$ and $\exptn{\delta^X\left( Y \right)^2} $ is equal to
\begin{align*}
\exptn{\int_0^T \abs{ \mathcal{K}^* Y (t) }_{\mathbb{R}^d}^2 \wrt{t}} + \exptn{\int_{[0, T]^2} \mathrm{tr} \left( \mathcal{K}^* \otimes \mathcal{K}^* \mathcal{D} Y (s, t) \, \mathcal{K}^* \otimes \mathcal{K}^* \mathcal{D} Y (t, s) \right) \wrt{s} \wrt{t}}.
\end{align*}
\end{theorem}

\begin{proof}
From the computation of the trace term in Theorem 4.5 of \cite{cl2018}, we know that $\exptn{\delta^X\left( Y \right)^2}$ is equal to
\begin{align*}
\lim_{\norm{\pi} \rightarrow 0} \exptn{\int_0^T \abs{ \mathcal{K}^* Y^{\pi} (t) }_{\mathbb{R}^d}^2 \wrt{t}}
+ \lim_{\norm{\pi} \rightarrow 0} \exptn{ \sum_{i, j} \sum_{k,l=1}^d \left\langle \mathcal{D}_{\cdot}^{(k)} Y^{(l)}_{t_j}, \mathds{1}_{\Delta_i} (\cdot) \right\rangle_{\mathcal{H}_1} \left\langle \mathcal{D}^{(l)}_{\cdot} Y^{(k)}_{t_i}, \mathds{1}_{\Delta_j} (\cdot) \right\rangle_{\mathcal{H}_1} }.
\end{align*}
The first term converges to $\exptn{\int_0^T \abs{\mathcal{K}^* Y (t) }_{\mathbb{R}^d}^2 \wrt{t}} $ and for the second term we have
\begin{align*}
&\exptn{ \sum_{i, j} \sum_{k,l=1}^d \left\langle \mathcal{D}_{\cdot}^{(k)} Y^{(l)}_{t_j}, \mathds{1}_{\Delta_i} (\cdot) \right\rangle_{\mathcal{H}_1} \left\langle \mathcal{D}^{(l)}_{\cdot} Y^{(k)}_{t_i}, \mathds{1}_{\Delta_j} (\cdot) \right\rangle_{\mathcal{H}_1} } \\
&\qquad \qquad= \exptn{ \sum_{i, j} \sum_{k,l=1}^d \int_0^T \mathcal{K}^* \left( \mathcal{D}_{\cdot}^{(k)} Y^{(l)}_{t_j} \right) (s) K(\Delta_i, s) \wrt{s} \int_0^T \mathcal{K}^* \left( \mathcal{D}_{\cdot}^{(l)} Y^{(k)}_{t_i} \right) (t) K(\Delta_j, t) \wrt{t} }.
\end{align*}
Using Lemma 3.2 in \cite{lim2018}, this expression is equal to
\begin{align*}
\exptn{\int_{[0, T]^2} \mathrm{tr} \left( \mathcal{K}^* \otimes \mathcal{K}^* \mathcal{D} Y^{\pi} (s, t) \, \mathcal{K}^* \otimes \mathcal{K}^* \mathcal{D} Y^{\pi} (t, s) \right) \wrt{s} \wrt{t}},
\end{align*}
which converges as $\norm{\pi} \rightarrow 0$ to
\begin{align*}
\exptn{\int_{[0, T]^2} \mathrm{tr} \left( \mathcal{K}^* \otimes \mathcal{K}^* \mathcal{D} Y (s, t) \, \mathcal{K}^* \otimes \mathcal{K}^* \mathcal{D} Y (t, s) \right) \wrt{s} \wrt{t}}.
\end{align*}
\end{proof}

In the case of Brownian motion both $\mathcal{K}^{\ast}$ and $\mathcal{K}^* \otimes\mathcal{K}^{\ast}$ are identity operators and Theorem \ref{isometry2} recovers the usual It\^{o}-Skorohod isometry:
\begin{align*}
\mathbb{E}\left[  \delta^{X}(Y)^{2}\right]  =\mathbb{E}\left[ \int_{0}^{T}\left\vert Y_{t}\right\vert ^{2}\,\mathrm{d}t\right]  +\mathbb{E}\left[\int_{[0,T]^{2}}\mathrm{tr}\left(  \mathcal{D}_{t}Y_{s}\,\mathcal{D}_{s}
Y_{t}\right)  \,\mathrm{d}s\,\mathrm{d}t\right].
\end{align*}

\section{Approximation of the Skorohod integral} \label{Sko conv}
We will now put together the results of the previous section to show that the Skorohod integral of the discrete approximations to the solution of an RDE converge. Before we proceed, we will introduce additional notation. \par 
Let $Y \in \mathcal{C}^{p-var} \left( [0, T]; \mathcal{L} (\mathbb{R}^d; \mathbb{R}^m) \right)$ denote the path-level solution to 
\begin{align*}
\mathrm{d} Y_t = V(Y_t) \strato{\mathbf{X}_t}, \quad  Y_0 = y_0,
\end{align*}
where $V \in \mathcal{C}^{\lfloor p \rfloor + 1}_b \left( \mathbb{R}^{dm}; \mathbb{R}^{dm} \otimes \mathbb{R}^d \right)$. \par
Recall that 
\begin{align*}
\mathcal{D}_t Y_s = \mathds{1}_{[0, t)} (s) J^{\mathbf{X}}_{t \leftarrow s} V(Y_s), \quad s, t \in [0, T],
\end{align*}
where here and henceforth, we will use the shorthand
\begin{align*}
J^{\mathbf{X}}_{t \leftarrow s} := J^{\mathbf{X}}_t \left(J^{\mathbf{X}}_s\right)^{-1}, \quad 0 \leq s < t \leq T.
\end{align*}
Given a Hilbert space $H$, we will denote an element of $y$ of $\mathbb{R}^m \otimes H$ as 
\begin{align} \label{bracketNotation}
y = \sum_{j=1}^m e_j \otimes [y]_j,
\end{align}
where $[y]_j \in H$ for $j = 1, \ldots, m$. (Note that there may be several ways to perform the decomposition.) \par
Now fix $0 \leq s < t \leq T$. Since $V(Y_s) \in \mathbb{R}^{md} \otimes \mathbb{R}^d \simeq \mathbb{R}^m \otimes \mathbb{R}^d \otimes \mathbb{R}^d$, we will decompose $V(Y_s)$ as
\begin{align*}
V(Y_s) = \sum_{j=1}^m e_j \otimes \left[ V(Y_s) \right]_j,
\end{align*}
where
\begin{align*}
[V(Y_s)]_j := \sum_{i, k=1}^d V^{(d(j-1) + i)}_k  (Y_s) \, e_i \otimes e_k.
\end{align*}
If we canonically identify $\mathbb{R}^{md} \otimes \mathbb{R}^d$ with the space of $md$-by-$d$ matrices, then $\left[ V(Y_s) \right]_j$ simply denotes the $d$-by-$d$ sub-matrix of $V(Y_s)$ which starts at the $(d(j - 1) + 1)^{th}$ row and ends at the $dj^{th}$ row. Contrast this with $V_j (Y_s)$, which denotes the $j^{th}$ column of $V(Y_s)$. \par
We will do the same with $Y_s \in \mathbb{R}^{md} \simeq \mathbb{R}^m \otimes \mathbb{R}^d$, and write
\begin{align*}
Y_s = \sum_{j=1}^m e_j \otimes \left[ Y_s \right]_j, \quad [Y_s]_j := \sum_{i=1}^d Y_s^{(d(j-1) + i)} e_i \in \mathbb{R}^d,
\end{align*}
and for $\displaystyle J^{\mathbf{X}}_{t \leftarrow s} V(Y_s) = \sum_{i=1}^{md} \sum_{k=1}^d a_{i, k} \, e_i \otimes e_k \in \mathbb{R}^{md} \otimes \mathbb{R}^d \simeq \mathbb{R}^m \otimes \mathbb{R}^d \otimes \mathbb{R}^d$, we have
\begin{align*}
&J^{\mathbf{X}}_{t \leftarrow s} V(Y_s) = \sum_{j=1}^m e_j \otimes \left[ J^{\mathbf{X}}_{t \leftarrow s} V(Y_s) \right]_j, \\
&\left[ J^{\mathbf{X}}_{t \leftarrow s} V(Y_s)\right]_j := \sum_{i, k = 1}^d a_{d(j-1) + i, k} \, e_i \otimes e_k \in \mathbb{R}^d \otimes \mathbb{R}^d.
\end{align*}

\begin{proposition} \label{skorohodLimit3rdLevel}
Let $\mathbf{X} \in \mathcal{C}^{0, p-var} \left([0, T]; G^{\lfloor p \rfloor} \left( \mathbb{R}^d\right) \right)$, $1 \leq p < 4$, be a Volterra rough path which satisfies Condition \ref{newCond1}, and assume that its kernel satisfies Condition \ref{amnCond} with $\alpha < \frac{1}{p}$. Let $Y \in \mathcal{C}^{p-var} \left( [0, T]; \mathcal{L} (\mathbb{R}^d; \mathbb{R}^m) \right)$ denote the path-level solution to 
\begin{align*}
\mathrm{d} Y_t = V(Y_t) \strato{\mathbf{X}_t}, \quad  Y_0 = y_0,
\end{align*}
where $V \in \mathcal{C}^{\lfloor p \rfloor + 1}_b \left( \mathbb{R}^{dm}; \mathbb{R}^{dm} \otimes \mathbb{R}^d \right)$. Then $Y \in \mathbb{D}^{1,2} (\mathbb{R}^m \otimes \mathcal{H}_1^d)$ and
\begin{align*}
\int_0^T Y_r \, \mathrm{d} X_r
&= \lim_{\norm{\pi = \{ r_i \} } \rightarrow 0} \sum_i \left[ Y_{r_i} \left( X_{r_i, r_{i+1}} \right) - \sum_{j=1}^m \left( \int_0^{r_i} \mathrm{tr} \, \left[ J^{\mathbf{X}}_{r_i \leftarrow s} V(Y_s) \right]_j \, R(\Delta_i, \wrt{s}) \right) e_j \right],
\end{align*}
where the limit is taken in $L^2\left(\Omega\right)$.
\end{proposition}

\begin{proof}
We first use integration-by-parts to obtain
\begin{align*}
\left\langle \delta^X(Y^{\pi}), \, e_j \right\rangle_{\mathbb{R}^m}
&= \sum_i \left[ \left\langle \left[ Y_{r_i} \right]_j, \, X_{r_i, r_{i+1}} \right\rangle_{\mathbb{R}^d} - \int_0^{r_i} \mathrm{tr} \, \left[ J^{\mathbf{X}}_{r_i \leftarrow s} V(Y_s) \right]_j \, R(\Delta_i, \wrt{s}) \right],
\end{align*}
for all $j = 1, \ldots, m$. Next, we invoke Theorem \ref{isometry2}, which requires us to prove that
\begin{align} \label{conv1}
\exptn{\int_0^T \abs{ \mathcal{K}^* \left( Y^{\pi} - Y \right) (t)}^2_{\mathbb{R}^{md}} \wrt{t}} \rightarrow 0,
\end{align}
and
\begin{align} \label{conv2}
\exptn{\int_{[0, T]^2} \abs{ \mathcal{K}^* \otimes \mathcal{K}^* (\mathcal{D}_s Y^{\pi}_t - \mathcal{D}_s Y_t) (s, t)}_{\mathbb{R}^{md} \otimes \mathbb{R}^d}^2 \wrt{s} \wrt{t} } \rightarrow 0.
\end{align}
We will show that $Y$ is $\frac{1}{p}$-H\"{o}lder continuous in $L^2 \left( \Omega; \mathbb{R}^{md} \right)$, and then invoke Proposition \ref{nualartProp} to obtain \eqref{conv1}. We have
\begin{align} \label{YHolder}
\begin{split}
\abs{Y_{s,t}}
&\leq C \left( \norm{\mathbf{X}}_{p-var; [s, t]} \vee \norm{\mathbf{X}}^p_{p-var; [s, t]} \right) \\
&\leq C \norm{\mathbf{X}}_{\frac{1}{p}-H\ddot{o}l; [s, t]} \left( (t - s) \vee (t - s)^{\frac{1}{p}} \right) \\
&\leq C \norm{\mathbf{X}}_{\frac{1}{p}-H\ddot{o}l; [0, T]} \left( T^{1 - \frac{1}{p}} \vee 1 \right) (t - s)^{\frac{1}{p}}
\end{split}
\end{align}
almost surely, and thus
\begin{align*}
\sqrt{\exptn{\abs{Y_{s,t}}}} \leq C_1 \abs{t - s}^{\frac{1}{p}}
\end{align*}
since $\norm{\mathbf{X}}_{\frac{1}{p}-H\ddot{o}l; [0, T]}$ has moments of all orders. \par
To show \eqref{conv2}, we will apply Proposition \ref{HSprop} with $\psi(s, t) = \mathcal{D}_s Y_t = \mathds{1}_{[0, t)} (s) J_t^{\mathbf{X}} \left( J_s^{\mathbf{X}} \right)^{-1} V(Y_s)$. To do so, we have to show that $\tilde{\psi}(s, t) := J_t^{\mathbf{X}} \left( J_s^{\mathbf{X}} \right)^{-1} V(Y_s)$ is strongly $\frac{1}{p}$-H\"{o}lder bi-continuous in $L^2 \left(\Omega; \mathbb{R}^{md} \otimes \mathbb{R}^d \right)$.
By Lemma 3.6 in \cite{lim2018}, this is equivalent to showing that $J_{\cdot}^{\mathbf{X}}$ and $\left(J_{\cdot}^{\mathbf{X}} \right)^{-1} Y_{\cdot}$ are both $\frac{1}{p}$-H\"{o}lder continuous. \par
Using \eqref{JBound}, we have
\begin{align*}
\abs{ J^{\mathbf{X}}_{s,t}}
&\leq C_1 \norm{\mathbf{X}}_{p-var; [s,t]} \exp \left( C_2 N^{\mathbf{X}}_{1; [s, t]} \right) \\
&\leq C_1 (t - s)^{\frac{1}{p}} \norm{\mathbf{X}}_{\frac{1}{p}-H\ddot{o}l;[0, T]} \exp \left( C_2 N^{\mathbf{X}}_{1; [0, T]} \right),
\end{align*}
which yields $\frac{1}{p}$-H\"{o}lder continuity for $J_{\cdot}^{\mathbf{X}}$ as the expression to the right of $(t-s)^\frac{1}{p}$ is in $L^q (\Omega)$ for all $q > 0$. The same is true for $\left( J^{\mathbf{X}}_{\cdot} \right)^{-1}$ since the inverse obeys the same bound. \par
Finally, $\left( J^{\mathbf{X}}_{\cdot} \right)^{-1} V(Y_{\cdot})$ is also $\frac{1}{p}$-H\"{o}lder continuous, since $V$ is $\mathcal{C}^1$ smooth and both $Y$ and $\left( J^{\mathbf{X}}_{\cdot} \right)^{-1}$ are $\frac{1}{p}$-H\"{o}lder continuous.
\end{proof}

\section{Augmenting the Skorohod integral with higher-level terms} \label{Sko aug}

The main purpose of this section is to show that the usual Riemann-sum approximation to the Skorohod integral can be augmented with (suitably corrected) second-level and third-level rough path terms which vanish in $L^{2}(\Omega)$ as the mesh of the partition goes to zero.

Before we do so, we will extend the theory of controlled rough paths to the case $3 \leq p < 4$, and give bounds on the higher-directional derivatives of a controlled rough path satisfying an RDE.

\subsection{Estimates for controlled rough paths of lower regularity} \label{controlledLR} 
To construct the rough integral of controlled rough paths for $3 \leq p < 4$, we need the following new definition.

\begin{definition}
Let $\mathbf{x}=\left(  1, x, \mathbf{x}^{2}, \mathbf{x}^{3} \right)
\in\mathcal{C}^{p-var}\left(  [0, T]; G^{3}(\mathbb{R}^{d})\right) $, where $3
\leq p < 4$, and let $q$ be such that $\frac{1}{p} + \frac{1}{q} > 1$. Let
$\left(  \phi, \phi^{\prime}, \phi^{\prime\prime}\right) $ satisfy
\begin{align*}
& \phi\in\mathcal{C}^{p-var} \left( [0, T]; \mathcal{U} \right) ,\\
& \phi^{\prime}\in\mathcal{C}^{p-var} \left( [0, T]; \mathcal{L} (\mathbb{R}^{d}; \mathcal{U}) \right) , \quad\mathrm{and}\\
& \phi^{\prime\prime}\in\mathcal{C}^{p-var} \left( [0, T]; \mathcal{L} (\mathbb{R}^{d} \otimes\mathbb{R}^{d}; \mathcal{U}) \right) .
\end{align*}
Then we say that $\left(  \phi,\phi^{\prime}, \phi^{\prime\prime}\right) $ (or $\phi$) is controlled by $\mathbf{x}$ if for all $s,t \in[0, T]$ we have
\begin{align} \label{controlled2}
\begin{split}
& \phi_{s,t} = \phi^{\prime}_{s} x_{s,t} + \phi^{\prime\prime}_{s}
\mathbf{x}^{2}_{s, t} + R_{s,t}^{\phi},\\
& \phi^{\prime}_{s,t} = \phi^{\prime\prime}_{s} x_{s,t} + R^{\phi^{\prime}}_{s,t},
\end{split}
\end{align}
where the remainder terms satisfy
\begin{align*}
R^{\phi} \in\mathcal{C}^{q-var} \left(  \left[  0,T\right] ; \mathcal{U} \right) , \quad R^{\phi^{\prime}} \in\mathcal{C}^{\frac{p}{2}-var} \left(\left[  0,T\right] ; \mathcal{U} \right) .
\end{align*}

\end{definition}

Thus, $\phi$ is controlled by $\mathbf{x}$ if $\left\| \phi\right\|_{p,q-cvar} < \infty$, where the controlled variation norm is defined as
\begin{align*}
\left\|  \phi\right\| _{p,q-cvar} := \left\|  \phi\right\| _{\mathcal{V}^{p};
[0, T]} + \left\|  \phi^{\prime}\right\| _{\mathcal{V}^{p}; [0, T]} + \left\|
\phi^{\prime\prime}\right\| _{\mathcal{V}^{p}; [0, T]} + \left\|  R^{\phi
}\right\| _{q-var; \left[  0,T\right] } + \left\|  R^{\phi^{\prime}}\right\|
_{\frac{p}{2}-var; \left[  0,T\right] }.
\end{align*}
Before we continue, note that $3 \leq p < 4$ implies that $\frac{p}{3} < \frac{p}{p-1} \leq\frac{p}{2}$. Since $q$-variation decreases with increasing $q$, we can always, if necessary, increase $q$ such that
\begin{align}
\label{qBounds}\frac{p}{3} \leq q < \frac{p}{p-1} \leq\frac{p}{2}
\end{align}
when we are working with $p$ in the interval $[3, 4)$.

The following theorem and the next two propositions are the lower-regularity analogues of Theorem 2.20, Proposition 2.22 and Proposition 2.21 respectively from \cite{cl2018}.

\begin{theorem} \label{controlledThm2} 
Let $\mathbf{x}=\left(  1, x,\mathbf{x}^{2}, \mathbf{x}^{3} \right)  \in\mathcal{C}^{p-var}\left(  [0, T]; G^{3} (\mathbb{R}^{d})\right) $, where $3 \leq p < 4$, and let $q$ be such that $\frac{1}{p} + \frac{1}{q} > 1 $. Let $\left(  \phi, \phi^{\prime}, \phi^{\prime\prime}\right) $ satisfy
\begin{align*}
& \phi\in\mathcal{C}^{p-var}\left( [0,T]; \mathcal{L}(\mathbb{R}^{d};\mathbb{R}^{e})\right) ,\\
& \phi^{\prime}\in\mathcal{C}^{p-var} \left(  [0,T]; \mathcal{L} (\mathbb{R}^{d}; \mathcal{L}(\mathbb{R}^{d};\mathbb{R}^{e}))\right) ,
\quad\mathrm{and} \\
& \phi^{\prime\prime}\in\mathcal{C}^{p-var} \left(  [0,T]; \mathcal{L} \left(
\mathbb{R}^{d} \otimes\mathbb{R}^{d}; \mathcal{L} (\mathbb{R}^{d}; \mathbb{R}^{e}) \right)  \right) .
\end{align*}
If $\left(  \phi, \phi^{\prime}, \phi^{\prime\prime}\right)  $ is controlled by $\mathbf{x}$ with remainder terms $R^{\phi}$ and $R^{\phi^{\prime}}$ of bounded $q$-variation and $\frac{p}{2}$-variation respectively, we can define
the rough integral
\begin{align} \label{controlledRPdefn2}
\int_{0}^{t} \phi_{r}\circ\mathrm{d}\mathbf{x}_{r}:=\underset{\left\Vert \pi\right\Vert \rightarrow 0,\pi=\left\{
0=r_{0}<\ldots<r_{n}=t\right\}  }{\lim}\sum_{i=0}^{n-1}\left(  \phi_{r_{i}} x_{r_{i},r_{i+1}}+\phi_{r_{i}}^{\prime} \mathbf{x}_{r_{i},r_{i+1}}^{2} + \phi^{\prime\prime}_{r_{i}} \mathbf{x}^{3}_{r_{i}, r_{i+1}} \right) ,
\end{align}
where we have made use of the canonical identification $\mathcal{L} (\mathbb{R}^{d}; \mathcal{L}(\mathbb{R}^{d};\mathbb{R}^{e}))\simeq \mathcal{L}(\mathbb{R}^{d}\otimes\mathbb{R}^{d}; \mathbb{R}^{e})$ and
$\mathcal{L} \left(  \mathbb{R}^{d} \otimes\mathbb{R}^{d}; \mathcal{L} (\mathbb{R}^{d};\mathbb{R}^{e}) \right)  \simeq\mathcal{L}(\mathbb{R}^{d} \otimes\mathbb{R}^{d}\otimes\mathbb{R}^{d}; \mathbb{R}^{e})$. Furthermore, if
$q \geq\frac{p}{3}$, then denoting
\begin{align*}
z_{t}:=\int_{0}^{t}\phi_{r}\circ\mathrm{d}\mathbf{x}_{r},\quad z_{t}^{\prime}
:=\phi_{t}, \quad z^{\prime\prime}_{t} := \phi^{\prime}_{t},
\end{align*}
$(z, z^{\prime}, z^{\prime\prime})$ is again controlled by $\mathbf{x}$, and
we have
\begin{align} \label{controlledBound2}
\left\|  z\right\| _{p,q} \leq C_{p, q} \left\|
\phi\right\| _{p,q-cvar} \left(  1 + \left\|  x\right\| _{p-var; [0, T]} +
\left\|  \mathbf{x}^{2}\right\| _{\frac{p}{2}-var;[0, T]} + \left\|
\mathbf{x}^{3}\right\| _{\frac{p}{3}-var; [0, T]} \right) .
\end{align}

\end{theorem}

\begin{proof}
Let $s < u < t$ and define
\begin{align*}
\Xi_{s,t} := \phi_s x_{s,t} + \phi'_s \mathbf{x}^2_{s,t} + \phi''_s \mathbf{x}^3_{s,t}.
\end{align*}
Then we have
\begin{align*}
\Xi_{s, u} + \Xi_{u, t} - \Xi_{s, t}
&= \left( \phi_s x_{s,u}  + \phi'_s \mathbf{x}^2_{s,u} + \phi''_s \mathbf{x}^3_{s,u} \right)
+ \left( \phi_u x_{u,t}  + \phi'_u \mathbf{x}^2_{u,t} + \phi''_u \mathbf{x}^3_{u,t} \right) - \phi_s \left( x_{s,u} + x_{u,t} \right) \\
&\qquad- \phi'_s \left( \mathbf{x}^2_{s,u} + \mathbf{x}^2_{u,t} + x_{s,u} \otimes x_{u,t} \right) - \phi''_s \left( \mathbf{x}^3_{s,u} + \mathbf{x}^3_{u,t} + \mathbf{x}^2_{s,u} \otimes x_{u,t} + x_{s,u} \otimes \mathbf{x}^2_{u,t} \right) \\
&= \phi_{s,u} x_{u,t} + \phi'_{s,u} \mathbf{x}^2_{u,t} + \phi''_{s,u} \mathbf{x}^3_{u,t} - \phi'_s \left( x_{s,u} \otimes x_{u,t} \right)
-\phi''_s \left( \mathbf{x}^2_{s,u} \otimes x_{u,t} + x_{s,u} \otimes \mathbf{x}^2_{u,t} \right) \\
&= \left( \phi'_s x_{s,u} + \phi''_s \mathbf{x}^2_{s,u} + R^{\phi}_{s,u} \right) x_{u,t} + \left(\phi''_s x_{s,u} + R^{\phi'}_{s,u} \right) \mathbf{x}^2_{u,t} + \phi''_{s,u} \mathbf{x}^3_{u,t} - \phi'_s \left( x_{s,u} \otimes x_{u,t} \right) \\
&\qquad- \phi''_s \left( \mathbf{x}^2_{s,u} \otimes x_{u,t} + x_{s,u} \otimes \mathbf{x}^2_{u,t} \right) \\
&= R^{\phi}_{s,u} x_{u,t} + R^{\phi'}_{s,u} \mathbf{x}^2_{u,t} + \phi''_{s,u} \mathbf{x}^3_{u,t}.
\end{align*}
Now let $\theta := \min \left( \frac{1}{p} + \frac{1}{q}, \frac{4}{p} \right)$, and let $\omega(s,t)$ denote the function
\begin{align*}
\norm{R^{\phi}}^{\frac{1}{\theta}}_{q-var; [s,t]} \norm{x}^{\frac{1}{\theta}}_{p-var; [s,t]} + \norm{R^{\phi'}}^{\frac{1}{\theta}}_{\frac{p}{2}-var; [s,t]} \norm{\mathbf{x}^2}^{\frac{1}{\theta}}_{\frac{p}{2}-var; [s,t]} + \norm{\phi''}^{\frac{1}{\theta}}_{p-var; [s,t]} \norm{\mathbf{x}^3}^{\frac{1}{\theta}}_{\frac{p}{3}-var; [s,t]}.
\end{align*}
This is a control as $\theta \leq \frac{4}{p}$ gives $\frac{1}{\theta} \left( \frac{4}{p} \right) \geq 1$ \cite{fv2010a}. Following Theorem 3.3 in \cite{cl2018}, for any partition $\pi = \{r_i\}$ of $[s, t]$ with $k$ sub-intervals, there necessarily exists some $r_j \in \pi$ such that
\begin{align*}
\abs{\Xi_{r_{j-1}, r_j} + \Xi_{r_j, r_{j+1}} - \Xi_{r_{j-1}, r_{j+1}}}
&\leq \abs{R^{\phi}_{r_{j-1}, r_j} x_{r_j, r_{j+1}}} + \abs{R^{\phi'}_{r_{j-1}, r_j} \mathbf{x}^2_{r_j, r_{j+1}}} + \abs{\phi''_{r_{j-1}, r_j} \mathbf{x}^3_{r_j, r_{j+1}}} \\
&\leq 3 \, \omega(r_{j-1}, r_{j+1})^{\theta}
\leq 3 \left( \frac{2}{k-1} \right)^{\theta} \omega(s, t)^{\theta}.
\end{align*}
Appropriately extracting points from the partition until $[s, t]$ remains gives us the bound
\begin{align} \label{controlledOldBound2}
\abs{\int_{\pi} \phi_r \strato{\mathbf{x}_r} - \left( \phi_s x_{s,t} + \phi'_s \mathbf{x}^2_{s,t}  + \phi''_s \mathbf{x}^3_{s,t} \right)}
< C \, \zeta(\theta) \, \omega(s, t)^{\theta},
\end{align}
where
\begin{align*}
\int_{\pi} \phi_r \strato{\mathbf{x}_r}
:= \sum_i \phi_{r_i} x_{r_i, r_{i+1}} + \phi'_{r_i} \mathbf{x}^2_{r_i, r_{i+1}} + \phi''_{r_i} \mathbf{x}^3_{r_i, r_{i+1}},
\end{align*}
and \eqref{controlledRPdefn2} is proved as in Theorem 3.3 of \cite{cl2018}. \par
If we define
\begin{align} \label{remainder2}
R^z_{s,t} := \int_s^t \phi_r \strato{x_r} - \left( \phi_s x_{s,t} + \phi'_s \mathbf{x}^2_{s,t} \right),
\end{align}
we get
\begin{align*}
z_{s,t} = z'_s x_{s,t} + z''_s \mathbf{x}^2_{s,t} + R^z_{s,t},
\end{align*}
and from \eqref{controlledOldBound2}, $\abs{R^z_{s,t}}^q$ is bounded above by
\begin{align*}
C_{p,q} \left[ \norm{\phi''}^q_{\mathcal{V}^p} \left( \norm{\mathbf{x}^3}^{\frac{p}{3}}_{\frac{p}{3}-var; [s,t]} \right)^{\frac{3q}{p}} +  \norm{x}^q_{p-var;[0, T]} \norm{R^{\phi}}^q_{q-var; [s,t]} + \norm{R^{\phi'}}^q_{\frac{p}{2}-var;[s,t]} \norm{\mathbf{x}^2}_{\frac{p}{2}-var;[s,t]}^q \right].
\end{align*}
Since $q \geq \frac{p}{3}$, the right side of the above expression is a control and is thus super-additive. Furthermore, $\norm{R^z}_{q-var; [0, T]}$ is bounded above by
\begin{align*}
C_{p, q} \left( \norm{\phi''}_{\mathcal{V}^p} \norm{\mathbf{x}^3}_{\frac{p}{3}-var; [0, T]} + \norm{x}_{p-var;[0, T]} \norm{R^{\phi}}_{q-var; [0, T]} + \norm{R^{\phi'}}_{\frac{p}{2}-var;[0, T]} \norm{\mathbf{x}^2}_{\frac{p}{2}-var;[0, T]} \right).
\end{align*}
Continuing, we define
\begin{align*}
R^{z'}_{s,t} := \phi''_s \mathbf{x}^2_{s, t} + R_{s,t}^{\phi},
\end{align*}
which gives
\begin{align*}
z'_{s,t} = z''_s x_{s,t} + R^{z'}_{s,t},
\end{align*}
as well as
\begin{align*}
\norm{R^{z'}}_{\frac{p}{2}-var; [0, T]}
\leq \norm{\phi''}_{\mathcal{V}^p; [0, T]} \norm{\mathbf{x}^2}_{\frac{p}{2}-var; [0, T]} + \norm{R^{\phi}}_{\frac{p}{2}-var; [0, T]}.
\end{align*}
\end{proof}

For the next proposition, given maps $A \in\mathcal{L}(\mathbb{R}^{d}; \mathcal{L}(\mathcal{U}; \mathcal{V}))$ and $B \in\mathcal{L}(\mathbb{R}^{d}; \mathcal{U})$, we will identify them as tensors (either $\mathcal{L} (\mathcal{U}; \mathcal{V})$-valued or $\mathcal{U}$-valued)
\begin{align*}
& A = \sum_{j=1}^{d} a_{j} \, \mathrm{d} e_{j}, \quad a_{j} \in\mathcal{L}%
(\mathcal{U}; \mathcal{V}),\\
& B = \sum_{j=1}^{d} b_{j} \, \mathrm{d} e_{j}, \quad b_{j} \in\mathcal{U},
\end{align*}
and adopt the following notation
\begin{align*}
AB := a_{i} (b_{j}) \, \mathrm{d} e_{i} \otimes\mathrm{d} e_{j},\\
BA := a_{j} (b_{i}) \, \mathrm{d} e_{i} \otimes\mathrm{d} e_{j},\\
\mathrm{Sym} (AB) := \frac{1}{2} (AB + BA).
\end{align*}

\begin{proposition}
\label{controlledLeibniz2} (Leibniz rule) For $3 \leq p < 4$, let
\begin{align*}
& \phi\in\mathcal{C}^{p-var} \left( [0, T]; \mathcal{L}(\mathcal{U};
\mathcal{V}) \right) ,\\
& \phi^{\prime}\in\mathcal{C}^{p-var} \left( [0, T]; \mathcal{L}%
(\mathbb{R}^{d}; \mathcal{L}(\mathcal{U}; \mathcal{V})) \right) ,
\quad\mathrm{and}\\
& \phi^{\prime\prime}\in\mathcal{C}^{p-var} \left( [0, T]; \mathcal{L}%
(\mathbb{R}^{d} \otimes\mathbb{R}^{d}; \mathcal{L}(\mathcal{U}; \mathcal{V}))
\right) .
\end{align*}
Assume that $\left( \phi, \phi^{\prime}, \phi^{\prime\prime}\right) $ is
controlled by $\mathbf{x} \in\mathcal{C}^{p-var} \left(  [0, T]; G^{3}\left(
\mathbb{R}^{d} \right)  \right) $, with remainder terms $R^{\phi}$ and
$R^{\phi^{\prime}}$ of bounded $q$-variation and $\frac{p}{2}$-variation
respectively, where $\frac{1}{p} + \frac{1}{q} > 1$ and $q \geq\frac{p}{3}$.

\begin{enumerate} [(i)]
\item Let
\begin{align*}
& \psi\in\mathcal{C}^{p-var}\left(  [0,T]; \mathcal{U} \right) ,\\
& \psi^{\prime}\in\mathcal{C}^{p-var}\left(  [0,T]; \mathcal{L}(\mathbb{R}%
^{d}; \mathcal{U} )\right) , \quad\mathrm{and}\\
& \psi^{\prime\prime}\in\mathcal{C}^{p-var}\left(  [0,T]; \mathcal{L}%
(\mathbb{R}^{d} \otimes\mathbb{R}^{d}; \mathcal{U} )\right) .
\end{align*}
If $\left( \psi,\psi^{\prime}, \psi^{\prime\prime}\right)  $ is controlled by
$\mathbf{x}$, then the path $\phi\psi\in\mathcal{C}^{p-var}([0,T];
\mathcal{V})$ given by the composition of $\phi$ and $\psi$ is also controlled
by $\mathbf{x}$, with derivative process
\begin{align*}
(\phi\psi)^{\prime\prime}\psi+\phi\psi^{\prime}%
\end{align*}
and second derivative process
\begin{align*}
(\phi\psi)^{\prime\prime}= \phi^{\prime\prime}\psi+ 2 \, \mathrm{Sym}
(\phi^{\prime}\psi^{\prime}) + \phi\psi^{\prime\prime}.
\end{align*}
In addition, we have the bound
\begin{align} \label{leibnizBound2a}
\left\|  \phi\psi\right\| _{p,q-cvar} \leq4 \left\|
\phi\right\| _{p,q-cvar} \left\|  \psi\right\| _{p,q-cvar} \left(  1 +
\left\|  x\right\| _{p-var; [0, T]} + \left\|  \mathbf{x}^{2}\right\|
_{\frac{p}{2}-var; [0,T]} \right) .
\end{align}

\item Suppose that $\psi\in\mathcal{C}^{q-var}([0,T]; \mathcal{U})$. Then
$\phi\psi\in\mathcal{C}^{p-var}([0,T]; \mathcal{V})$ is also controlled by
$\mathbf{x}$, with derivative process
\begin{align*}
(\phi\psi)^{\prime}= \phi^{\prime}\psi
\end{align*}
and second derivative process
\begin{align*}
(\phi\psi)^{\prime\prime}= \phi^{\prime\prime}\psi.
\end{align*}
Moreover, we have the bound
\begin{align}
\label{leibnizBound2b}\left\|  \phi\psi\right\| _{p,q-cvar} \leq\left\|
\phi\right\| _{p,q-cvar} \left\|  \psi\right\|  _{\mathcal{V}^{q}; [0, T]}.
\end{align}

\end{enumerate}
\end{proposition}

\begin{proof}
(i) It is trivial to see that $\norm{\phi'\psi + \phi\psi'}_{p-var;[0, T]}$ and $\norm{\phi'' \psi + 2 \, \mathrm{Sym} (\phi' \psi') + \phi \psi''}_{p-var; [0, T]}$ satisfy \eqref{leibnizBound2a}. First we denote
\begin{align*}
\tilde{R}^{\phi}_{s,t} := \phi''_s \mathbf{x}^2_{s,t} + R^{\phi}_{s,t}, \\
\tilde{R}^{\psi}_{s,t} := \psi''_s \mathbf{x}^2_{s,t} + R^{\psi}_{s,t},
\end{align*}
and since $\norm{\cdot}_{\frac{p}{2}-var;[s, t]} \leq \norm{\cdot}_{q-var;[s, t]}$, we have the bounds
\begin{align*}
\norm{\tilde{R}^{\phi}}_{\frac{p}{2}-var;[s, t]}
\leq \norm{\phi''}_{\infty} \norm{\mathbf{x}^2}_{\frac{p}{2}-var; [s, t]} + \norm{R^{\phi}}_{q-var; [s,t]}, \\
\norm{\tilde{R}^{\psi}}_{\frac{p}{2}-var; [s, t]}
\leq \norm{\psi''}_{\infty} \norm{\mathbf{x}^2}_{\frac{p}{2}-var; [s, t]} + \norm{R^{\psi}}_{q-var; [s, t]},
\end{align*}
for all $s, t$ in $[0, T]$. Continuing, we compute
\begin{align} \label{phipsiInc}
\begin{split}
(\phi \psi)_{s,t}
&= \phi_{s,t} \psi_{s,t} +\phi_{s,t} \psi_s + \phi_s \psi_{s,t} \\
&= \left(\phi'_s x_{s,t} + \tilde{R}^{\phi}_{s,t} \right) \left(\psi'_s x_{s,t} + \tilde{R}^{\psi}_{s,t} \right) + \left(\phi'_s x_{s,t} + \phi''_s \mathbf{x}_{s,t}^2 + R^{\phi}_{s,t} \right) \psi_s + \phi_s \left(\psi'_s x_{s,t} + \psi''_s \mathbf{x}_{s,t}^2 + R^{\psi}_{s,t} \right) \\
&= \phi'_s x_{s,t} \left( \psi'_s x_{s,t} \right) + \left( \phi'_s \psi_s + \phi_s \psi'_s \right) x_{s,t} + \left(\phi''_s \psi_s + \phi_s \psi''_s \right) \mathbf{x}^2_{s,t} \\
&\qquad + \phi'_s x_{s,t} \tilde{R}^{\psi}_{s,t} + \tilde{R}^{\phi}_{s,t} \psi'_s x_{s,t} + \tilde{R}^{\phi}_{s,t} \tilde{R}^{\psi}_{s,t} + R^{\phi}_{s,t} \psi_s + \phi_s R^{\psi}_{s,t}.
\end{split}
\end{align}
Denoting
\begin{align*}
&\phi'_s = \sum_{i=1}^d \phi'_i (s) \, \mathrm{d} e_i, \quad \phi'_i(s) \in \mathcal{L}(\mathcal{U}; \mathcal{V}) \:\: \forall i \in \{1, \ldots, d \} \, \mathrm{and} \, s \in [0, T], \\
&\psi'_s = \sum_{i=1}^d \psi'_i (s) \, \mathrm{d} e_i, \quad \psi'_i(s) \in \mathcal{U} \:\: \forall i \in \{1, \ldots, d \} \, \mathrm{and} \, s \in [0, T],
\end{align*}
we have
\begin{align*}
\phi'_s x_{s,t} \left( \psi'_s x_{s,t} \right)
&= \sum_{i,j=1}^d \phi'_i (s) ( \psi'_j (s)) \, x^{(i)}_{s,t} x^{(j)}_{s,t} \\
&= \left( \sum_{i, j =1}^d \left( \phi'_i (s) (\psi'_j(s)) + \phi'_j(s) (\psi'_i(s)) \right) \mathrm{d} e_i \otimes \mathrm{d} e_j \right) \frac{1}{2} \left( x_{s,t} \otimes x_{s,t} \right) \\
&= 2 \, \mathrm{Sym} (\phi'_s \psi'_s) \, \mathbf{x}^2_{s,t}
\end{align*}
Thus, continuing from \eqref{phipsiInc}, we have
\begin{align*}
(\phi \psi)_{s,t}
= \left( \phi'_s \psi_s + \phi_s \psi'_s \right) x_{s,t} + \left(\phi''_s \psi_s + 2 \, \mathrm{Sym} (\phi'_s \psi'_s) + \phi_s \psi''_s \right) \mathbf{x}^2_{s,t} + R^{\phi\psi}_{s,t},
\end{align*}
where
\begin{align*}
R^{\phi\psi}_{s,t}
:= \phi'_s x_{s,t} \tilde{R}^{\psi}_{s,t} + \tilde{R}^{\phi}_{s,t} \psi'_s x_{s,t} + \tilde{R}^{\phi}_{s,t} \tilde{R}^{\psi}_{s,t} + R^{\phi}_{s,t} \psi_s + \phi_s R^{\psi}_{s,t}.
\end{align*}
We can use the fact that $\frac{4q}{p} > \frac{3q}{p} \geq 1$ to show that $\norm{R^{\phi\psi}}_{q-var; [0, T]}$ is bounded above by
\begin{align*}
\norm{x}_{p} \left( \norm{\phi'}_{\infty} \norm{\tilde{R}^{\psi}}_{\frac{p}{2}} + \norm{\psi'}_{\infty} \norm{\tilde{R}^{\phi}}_{\frac{p}{2}} \right)
+ \norm{\tilde{R}^{\phi}}_{\frac{p}{2}} \norm{\tilde{R}^{\psi}}_{\frac{p}{2}} + \norm{\psi}_{\infty} \norm{R^{\phi}}_{q} + \norm{\phi}_{\infty} \norm{R^{\psi}}_{q},
\end{align*}
where here we use $\norm{\cdot}_p$ as short-hand for $\norm{\cdot}_{p-var; [0, T]}$. \par
Moving on, we need to show that
\begin{align*}
(\phi \psi)'_{s,t} = (\phi \psi)''_s x_{s,t} + R^{(\phi\psi)'}_{s,t}.
\end{align*}
We have
\begin{align*}
\left( \phi \psi \right)'_{s,t}
&= \phi'_t \psi_t + \phi_t \psi'_t - \phi'_s \psi_s - \phi_s \psi'_s \\
&= \phi'_{s,t} \psi_s + \phi_s \psi'_{s,t} + \phi'_s \psi_{s,t} + \phi_{s,t} \psi'_s + \phi'_{s,t} \psi_{s,t} + \phi_{s,t} \psi_{s,t} \\
&= \left( \phi''_s x_{s,t} + R^{\phi'}_{s,t} \right) \psi_s + \phi_s \left( \psi''_s x_{s,t} + R^{\psi'}_{s,t} \right) + \phi'_s \left( \psi'_s x_{s,t} + \tilde{R}^{\psi}_{s,t} \right) \\
&\qquad \qquad+ \left( \phi'_s x_{s,t} + \tilde{R}^{\phi}_{s,t} \right) \psi'_s + \phi'_{s,t} \psi_{s,t} + \phi_{s,t} \psi'_{s,t} \\
&= \left( \phi''_s \psi_s + \phi_s \psi''_s \right) x_{s,t} + \phi'_s (\psi'_s x_{s,t}) + \phi'_s x_{s,t} (\psi'_s) \\
&\qquad \qquad+ \underset{\displaystyle =: R_{s,t}^{(\phi\psi)'}}{\underbrace{ \phi'_s \tilde{R}^{\psi}_{s,t} + \tilde{R}^{\phi}_{s,t} \psi'_s + R^{\phi'}_{s,t} \psi_s + \phi_s R^{\psi'}_{s,t} + \phi'_{s,t} \psi_{s,t} + \phi_{s,t} \psi'_{s,t}}}.
\end{align*}
We have
\begin{align*}
\phi'_s (\psi'_s x_{s,t}) + \phi'_s x_{s,t} (\psi'_s)
&= \sum_{i,j=1}^d \phi'_j(s) (\psi'_i(s)) x^{(i)}_{s,t} \, \mathrm{d} e_j + \sum_{i,j=1}^d \phi'_i(s) (\psi'_j(s)) x^{(i)}_{s,t} \, \mathrm{d} e_j \\
&= 2 \, \mathrm{Sym} (\phi'_s \psi'_s) x_{s,t},
\end{align*}
and again $\norm{R^{(\phi\psi)'}}_{\frac{p}{2}-var; [0, T]}$ is bounded above by
\begin{align*}
\norm{\phi'}_{\infty} \norm{\tilde{R}^{\psi}}_{\frac{p}{2}} + \norm{\psi'}_{\infty} \norm{\tilde{R}^{\phi}}_{\frac{p}{2}} + \norm{\psi}_{\infty} \norm{R^{\phi'}}_{\frac{p}{2}}
+ \norm{\phi}_{\infty} \norm{R^{\psi'}}_{\frac{p}{2}} + \norm{\phi'}_{p} \norm{\psi}_{p} + \norm{\phi}_{p} \norm{\psi'}_{p}.
\end{align*}
(ii) Note that $\norm{\phi' \psi}_{p-var; [0, T]}$ and $\norm{\phi'' \psi}_{p-var; [0, T]}$ satisfy \eqref{leibnizBound2b}. Moreover, we have
\begin{align*}
\left( \phi \psi \right)_{s,t}
&= \phi_{s,t} \psi_s + \phi_t \psi_{s,t} \\
&= \left( \phi'_s x_{s,t} + \phi''_s \mathbf{x}^2_{s,t} + R^{\phi}_{s, t} \right) \psi_s + \phi_t \psi_{s,t} \\
&= \phi'_s \psi_s x_{s,t} + \phi''_s \psi_s \mathbf{x}^2_{s,t} + \underset{\displaystyle =: R^{\phi\psi}_{s,t}}{\underbrace{R^{\phi}_{s,t} \psi_s + \phi_t \psi_{s,t}}},
\end{align*}
and thus
\begin{align*}
\norm{ R^{\phi\psi}}_{q-var; [0, T]}
\leq \norm{\psi}_{\infty} \norm{R^{\phi}}_{q-var; [0, T]} + \norm{\phi}_{\infty} \norm{\psi}_{q-var; [0, T]}.
\end{align*}
Continuing, we have
\begin{align*}
(\phi \psi)'_{s,t}
&= \phi'_{s,t} \psi_s + \phi'_t \psi_{s,t} \\
&= \left(\phi''_s x_{s,t} + R^{\phi'}_{s,t} \right) \psi_s + \phi'_t \psi_{s,t} \\
&= \phi''_s \psi_s x_{s,t} + \underset{\displaystyle =: R^{(\phi\psi)'}_{s,t}}{\underbrace{R^{\phi'}_{s,t} \psi_s + \phi'_t \psi_{s,t}}},
\end{align*}
which implies that
\begin{align*}
\norm{R^{(\phi \psi)'}}_{\frac{p}{2}-var; [0, T]}
\leq \norm{\psi}_{\infty} \norm{R^{\phi'}}_{\frac{p}{2}-var; [0, T]} + \norm{\phi'}_{\infty} \norm{\psi}_{q-var; [0, T]}.
\end{align*}
\end{proof}

\begin{proposition}
\label{controlledSmoothMap2} Let $\mathbf{x} \in\mathcal{C}^{p-var} \left(
[0, T]; G^{3}\left(  \mathbb{R}^{d} \right)  \right) $ where $3 \leq p < 4$.
We assume that
\begin{align*}
& y \in\mathcal{C}^{p-var} \left( [0, T]; \mathcal{U} \right) ,\\
& y^{\prime}\in\mathcal{C}^{p-var} \left( [0, T]; \mathcal{L} \left(
\mathbb{R}^{d}, \mathcal{U} \right)  \right) ,\\
& y^{\prime\prime}\in\mathcal{C}^{p-var} \left( [0, T]; \mathcal{L} \left(
\mathbb{R}^{d} \otimes\mathbb{R}^{d}; \mathcal{U} \right) \right) ,
\end{align*}
and $\left( y, y^{\prime}, y^{\prime\prime}\right) $ is controlled by
$\mathbf{x}$ with remainder terms $R^{y}$ and $R^{y^{\prime}}$ of bounded
$q$-variation and $\frac{p}{2}$-variation respectively, where $\frac{1}{p} +
\frac{1}{q} > 1$ and $q \geq\frac{p}{3}$. Let $\phi\in\mathcal{C}_{b}^{3}
\left(  \mathcal{U}, \mathcal{V} \right)  $ and define
\begin{align*}
\left(  z_{s}, z_{s}^{\prime}, z_{s}^{\prime\prime}\right)  := \left(
\phi\left(  y_{s}\right)  ,\nabla\phi\left(  y_{s}\right)  y_{s}^{\prime
},\nabla\phi\left(  y_{s}\right)  y_{s}^{\prime\prime}+\nabla^{2}\phi\left(
y_{s}\right)  \left(  y_{s}^{\prime},y_{s}^{\prime}\right)  \right)
\end{align*}
for all $s \in[0, T]$. Then $\left(  z, z^{\prime}, z^{\prime\prime}\right)  $
is controlled by $\mathbf{x}$, and we have the following bounds
\begin{align*}
\left\|  z\right\| _{p-var; \left[  0,T \right]  }  & \leq\left\|
\phi\right\| _{\mathcal{C}_{b}^{3}} \left\|  y\right\| _{p-var; \left[
0,T\right]  },\\
\left\|  z^{\prime}\right\| _{p-var; \left[  0,T\right]  }  & \leq\left\|
\phi\right\| _{\mathcal{C}_{b}^{3}} \left\|  y\right\| _{\mathcal{V}^{p}; [0,
T]} \left\|  y^{\prime}\right\| _{\mathcal{V}^{p}; [0, T]},\\
\left\|  z^{\prime\prime}\right\| _{p-var; \left[  0,T\right]  }  &
\leq\left\|  \phi\right\| _{\mathcal{C}_{b}^{3}} \left\|  y\right\|
_{\mathcal{V}^{p}; [0, T]} \left(  \left\|  y^{\prime\prime}\right\|
_{\mathcal{V}^{p}; [0, T]} + \left\|  y^{\prime}\right\| ^{2}_{\mathcal{V}%
^{p}; [0, T]} \right) ,
\end{align*}
and
\begin{align}
\label{est Rz}\left\|  R^{z}\right\| _{q-var; \left[  0,T\right]  }, \left\|
R^{z^{\prime}}\right\| _{\frac{p}{2}-var; \left[  0,T\right]  } \leq\left\|
\phi\right\| _{\mathcal{C}_{b}^{3}} \left(  1 + \left\|  y\right\| _{p,q-cvar}
\right) ^{3} \left(  1 + \left\|  x\right\| _{p-var; \left[  0,T\right]  } +
\left\|  \mathbf{x}^{2}\right\| _{\frac{p}{2}-var; \left[  0,T\right]  }
\right) ^{2}.
\end{align}

\end{proposition}

\begin{proof}
From Taylor's theorem we have
\begin{align} \label{z2}
z_{s,t}=\nabla \phi \left( y_{s}\right) y_{s,t}+\frac{1}{2}\nabla ^{2}\phi
\left( y_{s}\right) \left( y_{s,t}, y_{s,t} \right) + R_{s,t}^{Taylor}
\end{align}
for all $s < t$ in $\left[ 0,T\right]$, where $\abs{R_{s,t}^{Taylor}} \leq \norm{\phi}_{\mathcal{C}_b^3} \abs{y_{s,t}}^3$. From this it follows that
\begin{align} \label{Tremainder2}
\norm{R^{Taylor}}_{q-var; \left[ 0,T\right] }
\leq \norm{\phi}_{\mathcal{C}_b^3} \norm{y}_{p-var; \left[ 0, T \right] }^3.
\end{align}
As in the previous proposition, we define
\begin{align*}
\tilde{R}^{y}_{s, t} := y''_s \mathbf{x}^2_{s,t} + R^y_{s,t},
\end{align*}
and note that $y_{s,t} = y'_s x_{s,t} + \tilde{R}^y_{s,t}$ and
\begin{align} \label{rTildeEst}
\norm{\tilde{R}^y}_{\frac{p}{2}-var; [s, t]} \leq \norm{y''}_{\infty} \norm{\mathbf{x}^2}_{\frac{p}{2}-var; [s, t]} + \norm{R^y}_{\frac{p}{2}-var; [s, t]}
\end{align}
for all $s, t$ in $[0, T]$. We next use the fact that $(y, y' , y'')$ is controlled by $\mathbf{x}$ in equation \eqref{z2}, which yields
\begin{align*}
z_{s,t}
&= \nabla \phi \left( y_{s}\right) \left( y'_s x_{s,t} + y''_s \mathbf{x}_{s,t}^2 + R_{s,t}^y \right) +\frac{1}{2} \nabla^2 \phi \left( y_{s}\right) \left( y'_s x_{s,t} + \tilde{R}^y_{s,t}, y'_s x_{s,t} + \tilde{R}^y_{s,t} \right) + R_{s,t}^{Taylor} \\
&= \underset{\displaystyle = z'_s}{\underbrace{\nabla \phi \left( y_{s}\right) y'_s}} x_{s,t} + \underset{\displaystyle = z''_s}{\underbrace{\left( \nabla \phi \left( y_{s}\right) y''_s + \nabla^2 \phi \left( y_{s}\right) \left( y_{s}^{\prime }, y_{s}^{\prime
}\right) \right) }}\mathbf{x}_{s,t}^{2}+\underset{\displaystyle =: R_{s,t}^z}{\underbrace{\nabla \phi \left( y_{s}\right) R_{s,t}^{y}+E_{s,t}+R_{s,t}^{Taylor}}},
\end{align*}
where
\begin{align*}
E_{s,t}
&:= \nabla^2 \phi \left( y_s\right) \left( y'_{s} x_{s,t}, \tilde{R}^y_{s,t} \right) + \frac{1}{2} \nabla^2 \phi \left( y_{s}\right) \left( \tilde{R}_{s,t}^y, \tilde{R}_{s,t}^y \right),
\end{align*}
and we have used the fact that $\frac{1}{2} \nabla ^{2}\phi \left( y_{s}\right) \left( y'_{s} x_{s,t}, y'_{s} x_{s,t} \right) = \nabla^2 \phi \left( y_{s}\right) \left( y'_s, y'_s \right) \mathbf{x}_{s,t}^2$ (recall that $\mathbf{x}$ is assumed to be weakly-geometric). The stated estimates on the $p^{th}$-variation of $(z, z', z'')$ are then easily derived. We get
\begin{align*}
\norm{E}_{q-var; [0, T]}
\leq \norm{\phi}_{\mathcal{C}^3_b} \left( \norm{y'}_{\infty} \norm{x}_{p-var; [0, T]} \norm{\tilde{R}^y}_{\frac{p}{2}-var; [0, T]} + \norm{\tilde{R}^y}^2_{\frac{p}{2}-var; [0, T]} \right).
\end{align*}
After using \eqref{rTildeEst} and adding the $q^{th}$-variation bounds of $\nabla \phi(y) R^y$ and $R^{Taylor}$, we get bound \eqref{est Rz} for $R^z$. \par
Proceeding, we can apply Lemma 3.5 from \cite{cl2018} to $\nabla \phi (y)$ to obtain
\begin{align*}
R^{\nabla \phi(y)} \leq \norm{\phi}_{\mathcal{C}^3_b} \left( \norm{y}^2_{p-var; [0, T]} + \norm{\tilde{R}^y}_{\frac{p}{2}-var; [0, T]} \right).
\end{align*}
Furthermore, if we apply Lemma 3.6 from \cite{cl2018} with $\phi$ replaced with $\nabla \phi(y)$ and $\psi$ replaced by $y'$, we obtain
\begin{align*}
z'_{s,t} = \left( \nabla \phi(y) \, y' \right)_{s,t}
&=: (\phi \psi)_{s,t} \\
&= (\phi'_s \psi_s + \phi_s \psi'_s) \, x_{s,t} + R^{\phi\psi}_{s,t} \\
&= \underset{\displaystyle = z''_s}{\underbrace{\left( \nabla^2 \phi (y_s) \left(y'_s, y'_s \right) + \nabla \phi(y_s) \, y''_s \right)}} \, x_{s,t} + \underset{\displaystyle =: R^{z'}_{s,t}}{\underbrace{R^{\phi\psi}_{s,t}}}
\end{align*}
and
\begin{align*}
\norm{R^{z'}}_{\frac{p}{2}-var; [0, T]}
&\leq \left( \norm{\nabla \phi (y)}_{\mathcal{V}^p; [0, T]}  + \norm{R^{\nabla \phi(y)}}_{\frac{p}{2}-var; [0, T]} \right) \left( \norm{y'}_{\mathcal{V}^p; [0, T]} + \norm{R^{y'}}_{\frac{p}{2}-var; [0, T]} \right) \\
&\leq \norm{\phi}_{\mathcal{C}^3_b} \left( \left( 1 + \norm{y}_{p-var; [0, T]} \right)^2 + \norm{\tilde{R}^y}_{\frac{p}{2}-var; [0, T]} \right) \left( \norm{y'}_{\mathcal{V}^p; [0, T]} + \norm{R^{y'}}_{\frac{p}{2}-var; [0, T]} \right).
\end{align*}
\end{proof}

The following theorem extends Theorem 3.1 in \cite{cl2018}.
\begin{theorem} \label{controlledRDE} 
Consider the system of RDEs
\begin{align*}
& \mathrm{d} y_{t} = V(y_{t}) \circ\mathrm{d} \mathbf{x}_{t}, \quad y_{0} = a
\in\mathbb{R}^{e}, \\
& \mathrm{d} J^{\mathbf{x}}_{t} = \nabla V (y_{t}) \left( \circ\mathrm{d} \mathbf{x}_{t} \right)  J^{\mathbf{x}}_{t}, \quad J^{\mathbf{x}}_{0} = \mathcal{I}_{e},
\end{align*}
where $V=(V_{1},\ldots,V_{d})$ is a collection of $\mathbb{R}^{e}$-valued vector fields. If $\mathbf{x} = \left(  1, x,\mathbf{x}^{2}, \mathbf{x}^{3} \right)  \in\mathcal{C}^{p-var} \left( [0, T]; G^{3} \left( \mathbb{R}^{d}\right) \right) $, $3 \leq p < 4$, and $V$ is in $\mathcal{C}^{4}_{b}$, then both $(y, V(y), V^2(y) )$ and \\
$\left(  J^{\mathbf{x}}, \left( J^{\mathbf{x}} \right) ^{\prime}, \left(  J^{\mathbf{x}} \right)^{\prime\prime}\right) $ are controlled by $\mathbf{x}$. In addition,
\begin{align} \label{boundedVF2}
\left\|  y\right\| _{p,q-cvar} \leq C_{p} \left( 1 + \left\Vert V\right\Vert _{C_{b}^{3}} \right) ^{10} \left(  1 + \left\Vert \mathbf{x}\right\Vert_{p-var;[0,T]}\right)^{8},
\end{align}
and
\begin{align} \label{linearVF2}
\left\|  J^{\mathbf{x}}\right\| _{p,q-cvar} \leq C_{1} \left( 1 + \exp\left(  C_{2} N^{\mathbf{x}}_{1; [0, T]} \right)  \right) ^{10} \left(  1 + \left\Vert \mathbf{x}\right\Vert _{p-var;[0,T]}\right)^{8},
\end{align}
where $C_{1}$, $C_{2}$ depend on $p$ and $\left\|  V\right\| _{\mathcal{C}^{4}_{b}}$.
\end{theorem}

\begin{proof}
Using Corollary 10.15 in \cite{fv2010b}, for $\gamma > p$ and $s,t \in [0, T]$, we have
\begin{align*}
\abs{y_{s,t} - V(y_s) x_{s,t} - V^2(y_s) \mathbf{x}^2_{s,t} - V^3(y_s) \mathbf{x}^3_{s, t}}
\leq C_p \left( \norm{V}_{\mathcal{C}^3_b} \norm{\mathbf{x}}_{p-var;[s, t]} \right)^{\gamma},
\end{align*}
where $V^2(y_s)$ and $V^3(y_s)$ denote the following tensors
\begin{align*}
&V^2(y_s) := \nabla V(y_s) (V(y_s)) \in \mathbb{R}^e \otimes \mathbb{R}^d \otimes \mathbb{R}^d \quad \mathrm{and} \\
&V^3(y_s) := \nabla^2 V(y_s) (V(y_s), V(y_s)) + \nabla V(y_s) \left[ \nabla V(y_s) (V(y_s)) \right] \in \mathbb{R}^e \otimes \mathbb{R}^d \otimes \mathbb{R}^d \otimes \mathbb{R}^d
\end{align*}
respectively. \par 
This implies that
\begin{align} \label{REst2}
\begin{split}
\abs{R^y_{s,t}}^q
&\leq C_q \left( \abs{ V^3(y_s) \mathbf{x}^3_{s,t}}^q + \left( \norm{V}_{\mathcal{C}^3_b} \norm{\mathbf{x}}_{p-var; [s,t]} \right)^{\gamma q} \right) \\
&\leq C_q \left( \norm{V}_{\mathcal{C}^3_b}^{3q} \norm{\mathbf{x}}_{p-var;[s,t]}^{3q} + \norm{V}_{\mathcal{C}^3_b}^{\gamma q} \norm{\mathbf{x}}_{p-var;[s,t]}^{\gamma q} \right),
\end{split}
\end{align}
and thus
\begin{align*}
\norm{R^y}_{q-var; [0, T]} \leq C_q \left( \norm{V}^3_{\mathcal{C}^3_b} \norm{\mathbf{x}}_{p-var; [0, T]}^3 \vee \norm{V}_{\mathcal{C}^3_b}^{\gamma} \norm{\mathbf{x}}_{p-var; [0, T]}^{\gamma} \right),
\end{align*}
from the super-additivity of the right side of \eqref{REst2}.
Observe that
\begin{align*}
&\norm{V(y)}_{p-var; [0, T]} \leq \norm{V}_{\mathcal{C}^3_b} \norm{y}_{p-var; [0, T]}, \\
&\norm{V^2(y)}_{p-var; [0, T]} \leq \norm{V}_{\mathcal{C}^3_b}^2 \left(1 + \norm{y}_{p-var; [0, T]} \right)^2,
\end{align*}
and
\begin{align*}
\norm{y}_{p-var;[0, T]}
\leq C_p \left( \norm{V}_{\mathcal{C}^3_b} \norm{\mathbf{x}}_{p-var; [0, T]} \vee \norm{V}_{\mathcal{C}^3_b}^p \norm{\mathbf{x}}^p_{p-var; [0, T]} \right).
\end{align*}
Applying Proposition 2.21 from \cite{cl2018}, we also have
\begin{align*}
V(y)_{s, t} = V^2(y_s) x_{s,t} + R^{V(y)}_{s, t}, \quad \forall \, s < t \in [0, T],
\end{align*}
where
\begin{align*}
\norm{R^{V(y)}}_{\frac{p}{2}-var;[0,T]}
\leq \norm{V}_{\mathcal{C}_{b}^3} \left( \norm{y}_{p-var;[0,T]}^2 + \norm{R^y}_{q-var;[0,T]}\right),
\end{align*}
and thus \eqref{controlled2} is satisfied. \par 
Note that $q \leq \frac{p}{p-1}$ implies that $q \leq \frac{p}{2}$ for all $p \geq 3$.
Collecting all the estimates above and choosing $\gamma$ to be in $(p, 4)$ gives us the bound in \eqref{boundedVF2}. \par
The proof for the Jacobian is the same as that in Theorem 3.1 of \cite{cl2018}. From Proposition 5 in \cite{fr2013}, we can obtain the bound
\begin{align} \label{Jsup}
\norm{J^{\mathbf{x}}}_{\infty} \leq 1 + \exp \left( C_{p, \norm{V}_{\mathcal{C}^3_b}} \left( N^{\mathbf{x}}_{1; [0, T]} + 1 \right) \right) =: 1 + M,
\end{align}
and construct $\mathcal{C}^4_b$ vector fields $U_i(y_t)$ which equal the linear vector field $z \mapsto \nabla V_i (y_t) z$ on the set $\mathcal{W} = \left\{ z \in \mathbb{R}^{e^2} \; \big\vert \; \abs{z} < M + 2 \right\}$, and which satisfy
\begin{align*}
\norm{U_i}_{\mathcal{C}^4_b}
\leq \norm{V}_{\mathcal{C}^4_b} (M + 3), \quad i = 1, \ldots, d.
\end{align*}
Then the system of RDEs can be rewritten as a bounded RDE
\begin{align*}
\mathrm{d} \tilde{y}_t = \tilde{V} (\tilde{y}_t) \strato{\mathbf{x}_t}, \quad \tilde{y}_0 = (a, \mathcal{I}_e),
\end{align*}
where $\tilde{y} = \left( y, J^{\mathbf{x}} \right) \in \mathbb{R}^{e + e^2}$ and $\norm{\tilde{V}}_{\mathcal{C}^4_b} \leq \norm{V}_{\mathcal{C}^4_b} (M + 3)$. \par
Finally, an application of \eqref{boundedVF2} to the above equation yields \eqref{linearVF2}.
\end{proof}

\subsection{Upper bounds on the high-order Malliavin derivatives} \label{HOMD} 
We now use the results from the proceeding section to obtain upper bounds on the directional derivative. We first recall the following results from \cite{cl2018} for the formula of $\mathcal{D}_{h_{1},\ldots, h_{n}}^{n}y_{t}$.

\begin{theorem} \label{mdFormula} 
Let $p \geq1$ and $q \geq 1$ be such that $\frac{1}{p} + \frac{1}{q} > 1$, and let $n \in \mathbb{N}$ such that $n \geq 2$. Assume $\mathbf{x} \in\mathcal{C}^{0, p-var} \left( [0, T]; G^{\lfloor p \rfloor} \left( \mathbb{R}^{d}\right) \right) $ and suppose $y$ is the path-level solution to the RDE
\begin{align}
\mathrm{d}y_{t}=V\left(  y_{t}\right)  \circ\mathrm{d}\mathbf{x}_{t}, \quad y_{0}\in\mathbb{R}^{e}\text{ given},
\end{align}
where $V \in\mathcal{C}^{\lfloor p \rfloor+ n}_{b} \left( \mathbb{R}^{e}; \mathbb{R}^{e} \otimes\mathbb{R}^{d} \right) $. Suppose that $g_{1}, \ldots, g_{n} \in \mathcal{C}^{q-var}([0, T]; \mathbb{R}^d)$. Then $\mathcal{D}_{g_{1}, \ldots, g_{n}}^{n} y_{t}$ satisfies the RDE
\begin{align} \label{df}
\begin{split}
&\mathrm{d} \mathcal{D}_{g_{1}, \ldots, g_{n}}^{n} y_{t} 
= \sum_{i=1}^{n} \nabla^{i} V\left(  y_{t} \right)  A_{i}^{n}\left(  t\right) \circ\mathrm{d} \mathbf{x}_{t} + \sum_{i=0}^{n-1}\sum_{j=1}^{n} \nabla^{i} V\left(  y_{t} \right)  B_{i,j}^{n} \left( t \right) \, \mathrm{d} g_{j} (t), \\
&\mathcal{D}^n_{g_1, \ldots, g_n} y_0 = 0,
\end{split}
\end{align}
where $A_{i}^{n}$ and $B_{i,j}^{n}$ are respectively defined as
\begin{align}
\label{A}A_{i}^{n}\left(  t\right)  :=\sum_{\pi=\left\{  \pi_{1},\ldots ,\pi_{i}\right\}  \in\mathcal{P}\left(  \left\{  g_{1},\ldots, g_{n}\right\} \right)  }\mathcal{D}_{\pi_{1}}^{\left\vert \pi_{1}\right\vert }y_{t} \tilde{\otimes}\cdots\tilde{\otimes}\mathcal{D}_{\pi_{i}}^{\left\vert \pi_{i}\right\vert } y_{t}, \, t \in \left[ 0, T\right] ,
\end{align}
and
\begin{align} \label{Bb}
B_{i,j}^{n}\left(  t\right)  :=\sum_{\pi=\left\{  \pi_{1},\ldots, \pi_{i}\right\}  \in\mathcal{P}\left(  \left\{ g_{1}, \ldots, g_{j-1}, g_{j+1}, \ldots, g_{n} \right\} \right)  }\mathcal{D}_{\pi_{1}}^{\left\vert \pi_{1}\right\vert }y_{t}\tilde{\otimes}\cdots\tilde{\otimes}\mathcal{D}_{\pi_{i}}^{\left\vert \pi_{i}\right\vert } y_{t}.
\end{align}

\end{theorem}

\begin{corollary} \label{cf} 
Under the conditions of the preceding theorem, $\mathcal{D}_{g_{1}, \ldots, g_{n}}^{n} y_{t}$ equals
\begin{align} \label{exp2}
\sum_{i=2}^{n}\int_{0}^{t} J_{t}^{\mathbf{x}} \left(J_{s}^{\mathbf{x}} \right) ^{-1} \nabla^{i}V\left(  y_{s}\right) A_{i}^{n} \left(  s\right)  \circ\mathrm{d} \mathbf{x}_{s} + \sum_{i=1}^{n-1} \sum_{j=1}^{n} \int_{0}^{t} J_{t}^{\mathbf{x}} \left( J_{s}^{\mathbf{x}} \right)^{-1} \nabla^{i} V\left(  y_{s} \right)  B_{i,j}^{n}\left(s \right) \, \mathrm{d} g_{j}(s)
\end{align}
for all $n\geq2$.
\end{corollary}

We now arrive at the main result of this section, which extends Proposition 3.5 in \cite{cl2018}.

\begin{proposition} \label{dir der est} 
Let $p \in[2, 4)$ and $q$ be such that $\frac{1}{p} + \frac{1}{q} > 1$. Let $\mathbf{x} \in\mathcal{C}^{0, p-var} \left( [0, T]; G^{\lfloor p \rfloor} \left( \mathbb{R}^{d}\right) \right) $, and $y$ be the solution to the RDE
\begin{align*}
\mathrm{d}y_{t}=V\left(  y_{t}\right)  \circ\mathrm{d}\mathbf{x}_{t}, \quad y_{0}\in\mathbb{R}^{e}\text{ given},
\end{align*}
where $V \in\mathcal{C}^{\lfloor p \rfloor+ n}_{b} (\mathbb{R}^{e}; \mathbb{R}^{e} \otimes\mathbb{R}^{d})$. Then there exists a polynomial $P_{d(n)}: \mathbb{R}_{+} \times\mathbb{R}_{+} \rightarrow\mathbb{R}_{+}$ of finite degree $d(n)$ for which
\begin{align} \label{ubound}
\left\|  \mathcal{D}_{g_{1}, \ldots, g_{n}}^{n} y_{\cdot} \right\|_{\mathcal{V}^{p}; \left[  0,T\right]  } \leq P_{d(n)} \left( \left\|  \mathbf{x}\right\| _{p-var; [0, T]}, \exp\left(  C \, N^{\mathbf{x}}_{1; [0, T]} \right)  \right)  \prod\limits_{i=1}^{n} \left\| g_i \right\| _{q-var; [0, T]},
\end{align}
for any $g_{1}, \ldots, g_{n} \in \mathcal{C}^{q-var}([0, T]; \mathbb{R}^d)$. Here $N_{1}^{\mathbf{x}}$ is defined as in \eqref{NDefn}, and the constant $C$ as well as the coefficients of $P_{d(n)}$ depend only on $\left\|  V\right\|_{\mathcal{C}^{\lfloor p \rfloor+ n}_{b}}$, $p$ and $q$ ($= \frac{p}{2}$ when $2 \leq p < 3 $).
\end{proposition}

\begin{proof}
We shall omit the proof as it proceeds in virtually the same manner as that of Proposition 3.5 in \cite{cl2018}, for which the reader is invited to consult. The only difference is that for the case $p \geq 3$, one will have to use Theorems 5.2 and 5.5, as well as Propositions 5.3 and 5.4 in lieu of Theorem 2.20 and 3.5, and Propositions 2.22 and 2.21 respectively from \cite{cl2018}.
\end{proof}

\subsection{Augmenting the higher-order iterated integrals}

For this section, we will use $\pi(n) := \left\{  t^{n}_{i} \right\} $ to denote the $n^{th}$ dyadic partition of $[0, T]$, i.e. $t^{n}_{i} = \frac{iT}{2^{n}}$ for $i = 0, \ldots, 2^{n}$, and $\Delta^{n}_{i}$ to denote the interval $\left[ t^{n}_{i}, t^{n}_{i+1}\right] $.

In addition, $\rho^{\prime}$ will denote the H\"{o}lder conjugate of $\rho$, i.e. $\frac{1}{\rho} + \frac{1}{\rho^{\prime}} = 1$.

The following proposition giving bounds for the compensated second-order terms was proven in \cite{cl2018}.

\begin{proposition} \label{2ndlevel} 
Let $X$ be a continuous, centered Gaussian process in $\mathbb{R}^{d}$ with i.i.d. components, and for $p \in[2, 4)$, let
$\mathbf{X} \in\mathcal{C}^{0, p-var} \left( [0, T]; G^{\lfloor p \rfloor} \left( \mathbb{R}^{d}\right) \right) $ denote the geometric rough path constructed from the limit of the piecewise-linear approximations of $X$.

Let $\rho$ and $q$ be such that $\rho\in\left[ 1, 2 \right) $ and $\frac{1}{p} + \frac{1}{q} > 1$. We assume that the covariance function of $X$ satisfies
\begin{enumerate} [(a)]
\item $\left\|  R\right\| _{\rho-var; [0, T]^{2}} < \infty$,
\item $\left\|  R(t, \cdot) - R (s, \cdot)\right\| _{q-var; [0, T]} \leq C
\left|  t - s\right| ^{\frac{1}{\rho}}$, for all $s, t \in[0, T]$.
\end{enumerate} 
Now let $\psi: \Omega\times[0, T] \rightarrow\mathbb{R}^{d} \otimes \mathbb{R}^{d}$ be a stochastic process satisfying $\displaystyle \psi_{t} = \sum_{a, b = 1}^{d} \psi_{t}^{(a, b)} \mathrm{d} e_{a} \otimes\mathrm{d} e_{b}
\in\mathbb{D}^{4,2} (\mathbb{R}^{d} \otimes\mathbb{R}^{d})$ for all $t \in [0, T]$. Furthermore, assume there exists $C < \infty$ such that for all $s, t \in [0, T]$ and $a, b = 1, \ldots, d$, we have
\begin{align} \label{propCond0}
\left|  \mathbb{E} \left[  \psi^{(a, b)}_{s} \psi^{(a, b)}_{t}\right] \right| \leq C,
\end{align}
and for $k = 2, 4$, we have
\begin{align} \label{propCond}
\left|  \mathbb{E} \left[  \mathcal{D}^{k}_{h_{1}, \ldots, h_{k}} \left(  \psi^{(a, b)}_{s} \psi^{(a, b)}_{t} \right)  \right] \right| 
\leq C \prod_{i=1}^{k} \left\|  \Phi(h_{i})\right\|_{q-var; [0, T]},
\end{align}
for all $h_{1}, \ldots, h_{k} \in\mathcal{H}_{1}^{d}$. Then
\begin{align}
\lim_{n \rightarrow \infty} \left\| \sum_{i=0}^{2^{n} - 1} \psi_{t^{n}_{i}}
\left( \mathbf{X}^{2}_{t^{n}_{i}, t^{n}_{i+1}} - \frac{1}{2} \sigma^{2}
\left( t^{n}_{i}, t^{n}_{i+1} \right)  \mathcal{I}_{d} \right) \right\|_{L^{2}(\Omega)} = 0.
\end{align}
\end{proposition}

We will now proceed to give a similar estimate for the third-order terms. We first begin with the following lemma.

\begin{lemma} \label{productExp} 
For any $h_{1}, \ldots, h_{6} \in\mathcal{H}_{1}^{d}$, we have
\begin{align*}
\mathbb{E} \left[  \psi_{t^{n}_{i}} \psi_{t^{n}_{j}} \prod_{k=1}^{6}
I_{1}(h_{k})\right]  = \mathbb{E} \left[  \mathcal{D}_{h_{1}, h_{2}, h_{3},
h_{4}, h_{5}, h_{6}}^{6} \psi_{t^{n}_{i}} \psi_{t^{n}_{j}}\right]  +
\sum_{\sigma\in\mathcal{S}_{6}} C_{\sigma, 1} A_{\sigma, 1} + C_{\sigma, 2}
A_{\sigma, 2} + C_{\sigma, 3} A_{\sigma, 3},
\end{align*}
where
\begin{align*}
& A_{\sigma, 1} := \mathbb{E} \left[  \mathcal{D}_{h_{\sigma(1)},
h_{\sigma(2)}, h_{\sigma(3)}, h_{\sigma(4)}}^{4} \psi_{t^{n}_{i}} \psi
_{t^{n}_{j}}\right]  \left\langle h_{\sigma(5)}, h_{\sigma(6)} \right\rangle
_{\mathcal{H}_{1}^{d}},\\
& A_{\sigma, 2} := \mathbb{E} \left[  \mathcal{D}_{h_{\sigma(1)},
h_{\sigma(2)}}^{2} \psi_{t^{n}_{i}} \psi_{t^{n}_{j}}\right]  \left\langle
h_{\sigma(3)}, h_{\sigma(4)} \right\rangle _{\mathcal{H}_{1}^{d}} \left\langle
h_{\sigma(5)}, h_{\sigma(6)} \right\rangle _{\mathcal{H}_{1}^{d}},\\
& A_{\sigma, 3} := \mathbb{E} \left[  \psi_{t^{n}_{i}} \psi_{t^{n}_{j}
}\right]  \left\langle h_{\sigma(1)}, h_{\sigma(2)} \right\rangle
_{\mathcal{H}_{1}^{d}} \left\langle h_{\sigma(3)}, h_{\sigma(4)} \right\rangle
_{\mathcal{H}_{1}^{d}} \left\langle h_{\sigma(5)}, h_{\sigma(6)} \right\rangle
_{\mathcal{H}_{1}^{d}},
\end{align*}
$\mathcal{S}_{6}$ denotes the symmetric group of permutations on $\{1, \ldots, 6\}$, and the $C_{\sigma, k}$'s are constants that depend on the permutation $\sigma$.
\end{lemma}

\begin{proof}
From the product formula \eqref{productFormula} we have
\begin{align} \label{product6}
\begin{split}
\prod_{k=1}^6 I_1(h_k)
&= I_6 (h_1 \otimes h_2 \otimes h_3 \otimes h_4 \otimes h_5 \otimes h_6) \\
&\qquad + \sum_{\sigma \in \mathcal{S}_6} C_{\sigma, 1} \, I_4 \left( h_{\sigma(1)} \otimes h_{\sigma(2)} \otimes h_{\sigma(3)} \otimes h_{\sigma(4)} \right) \left\langle h_{\sigma(5)}, h_{\sigma(6)} \right\rangle_{\mathcal{H}_1^d} \\
&\qquad \qquad+ C_{\sigma, 2} \, I_2 \left( h_{\sigma(1)} \otimes h_{\sigma(2)} \right) \left\langle h_{\sigma(3)}, h_{\sigma(4)} \right\rangle_{\mathcal{H}_1^d} \left\langle h_{\sigma(5)}, h_{\sigma(6)} \right\rangle_{\mathcal{H}_1^d} \\
&\qquad\qquad + C_{\sigma, 3} \, \left\langle h_{\sigma(1)}, h_{\sigma(2)} \right\rangle_{\mathcal{H}_1^d} \left\langle h_{\sigma(3)}, h_{\sigma(4)} \right\rangle_{\mathcal{H}_1^d} \left\langle h_{\sigma(5)}, h_{\sigma(6)} \right\rangle_{\mathcal{H}_1^d}.
\end{split}
\end{align}
Applying integration-by-parts \eqref{dualityFormula} finishes the proof.
\end{proof}

\begin{proposition} \label{3rdlevel} 
Let $\mathbf{X} \in\mathcal{C}^{0, p-var} \left( [0, T]; G^{3} \left( \mathbb{R}^{d}\right) \right) $, $3 \leq p < 4$, be a geometric Gaussian rough path which satisfies Condition \ref{newCond1}, and assume that its covariance function satisfies, for all $s, t \in[0, T]$,
\begin{align*}
\left\|  R(t, \cdot) - R (s, \cdot) \right\| _{q-var; [0, T]} \leq C \left| t - s\right| ^{\frac{1}{\rho}},
\end{align*}
for some finite constant $C$ and $\rho\in[1, 2)$.

Let $\psi: \Omega\times[0, T] \rightarrow\mathbb{R}^{d} \otimes\mathbb{R}^{d} \otimes\mathbb{R}^{d}$ be a stochastic process satisfying $\displaystyle \psi_{t} = \sum_{a,b,c = 1}^{d} \psi_{t}^{(a,b,c)} \mathrm{d} e_{a} \otimes \mathrm{d} e_{b} \otimes \mathrm{d} e_{c} \in \mathbb{D}^{6,2} (\mathbb{R}^{d} \otimes \mathbb{R}^{d} \otimes \mathbb{R}^{d})$ for all $t \in [0, T]$. Furthermore, assume there exists $C < \infty$ such that we have
\begin{align}
\left|  \mathbb{E} \left[  \psi^{(a, b)}_{s} \psi^{(a, b)}_{t}\right] \right| \leq C,
\end{align}
and for $k = 2, 4, 6$, we have
\begin{align} \label{propCond2}
\left|  \mathbb{E} \left[  \mathcal{D}^{k}_{h_{1}, \ldots, h_{k}} \left( \psi^{(a, b, c)}_{s} \psi^{(a, b, c)}_{t} \right)  \right] \right|  
\leq C \prod_{i=1}^{k} \left\| \Phi(h_{i})\right\|_{q-var; [0, T]},
\end{align}
for all $h_{1}, \ldots, h_{k} \in\mathcal{H}_{1}^{d}$, $s, t \in[0, T]$ and $a, b, c = 1, \ldots, d$. Then
\begin{align}
\lim_{n \rightarrow\infty} \left\|  \sum_{i=0}^{2^{n} - 1} \psi_{t^{n}_{i}}
\left(  \mathbf{X}^{3}_{t^{n}_{i}, t^{n}_{i+1}} \right) \right\|_{L^{2}(\Omega)} = 0.
\end{align}

\end{proposition}

\begin{proof}
First note that
\begin{align} \label{mainQuant2}
\begin{split}
\left\| \sum_{i=0}^{2^n - 1} \psi_{t^n_i} \left( \mathbf{X}^3_{t^n_i, t^n_{i+1}} \right) \right\|_{L^2(\Omega)}
\leq \left\| \sum_{i=0}^{2^n - 1} \psi_{t^n_i} \left( \left(\mathbf{X}^3_{t^n_i, t^n_{i+1}} \right)^S \right) \right\|_{L^2(\Omega)}
+ \left\| \sum_{i=0}^{2^n - 1} \psi_{t^n_i} \left( \left( \mathbf{X}^3_{t^n_i, t^n_{i+1}} \right)^{NS} \right) \right\|_{L^2(\Omega)},
\end{split}
\end{align}
where $\left( \mathbf{X}^3 \right)^S$ denotes the symmetric part of $\mathbf{X}^3$ and
\begin{align*}
\left( \mathbf{X}^3_{s,t} \right)^{NS} = \mathbf{X}^3_{s,t} - \left( \mathbf{X}^3_{s,t} \right)^S
\end{align*}
denotes the non-symmetric part. The two parts will be tackled separately, and since
\begin{align*}
\norm{\sum_{i=0}^{2^n - 1} \psi_{t^n_i} \left( \left( \mathbf{X}^3_{t^n_i, t^n_{i+1}} \right)^S \right)}_{L^2(\Omega)}
\leq \sum_{a, b, c = 1}^d \norm{ \sum_{i=0}^{2^n - 1} \psi_{t^n_i}^{(a, b, c)} \left( \left( \mathbf{X}^3_{t^n_i, t^n_{i+1}} \right)^S \right)^{(a, b, c)} }_{L^2(\Omega)},
\end{align*}
and similarly for the non-symmetric part, we will study the convergence of each fixed $(a, b, c)^{th}$ tensor component individually and henceforth suppress the notation for the component in the superscript of $\psi$. \par
(a) To begin, we will prove that
\begin{align}
\lim_{n \rightarrow \infty} \norm{\sum_{i=0}^{2^n - 1} \psi_{t^n_i} \left(
\mathbf{X}^3_{t^n_i, t^n_{i+1}} \right)^S}_{L^2(\Omega)} = 0.
\end{align}
Since
\begin{align*}
\left( \mathbf{X}^3_{s,t} \right)^S = \frac{1}{6} X_{s, t} \otimes X_{s,t} \otimes X_{s,t},
\end{align*}
this is equivalent to showing that
\begin{align*}
\exptn{\left( \sum_{i=0}^{2^n - 1} \psi_{t^n_i} \left( \left( \mathbf{X}^3_{t_i^n, t_{i+1}^n} \right)^{S} \right)^{(a, b, c)} \right)^2}
= \frac{1}{36} \sum_{i, j=0}^{2^n - 1} \exptn{ \psi_{t^n_i} \psi_{t^n_j} X^{(a)}_{\Delta^n_i} X^{(b)}_{\Delta^n_i} X^{(c)}_{\Delta^n_i} X^{(a)}_{\Delta^n_j} X^{(b)}_{\Delta^n_j} X^{(c)}_{\Delta^n_j}}
\end{align*}
converges to zero as $n \rightarrow \infty$.
First define
\begin{align*}
&h_1 := \mathds{1}_{\Delta^n_i}^{(a)}, \quad h_2 := \mathds{1}_{\Delta^n_i}^{(b)}, \quad h_3 := \mathds{1}_{\Delta^n_i}^{(c)}, \\
&h_4 := \mathds{1}_{\Delta^n_j}^{(a)}, \quad h_5 := \mathds{1}_{\Delta^n_j}^{(b)}, \quad h_6 := \mathds{1}_{\Delta^n_j}^{(c)}.
\end{align*}
Note that for $k = 1, \ldots, 6$, we have
\begin{align*}
\norm{\Phi(h_k)}_{q-var; [0, T]}
&= \norm{R\left( t^n_{i+1}, \cdot\right) - R\left( t^n_i, \cdot\right)}_{q-var; [0, T]} \quad \mathrm{or} \quad \norm{R\left( t^n_{j+1}, \cdot\right) - R\left( t^n_j, \cdot\right)}_{q-var; [0, T]} \\
&\leq C \, 2^{-\frac{n}{\rho}}
\end{align*}
and
\begin{align*}
\norm{h_k}_{\mathcal{H}_1^d}
&= \sqrt{\sigma^2 \left( t_i^n, t_{i+1}^n \right)} \quad \mathrm{or} \quad \sqrt{\sigma^2 \left( t_j^n, t_{j+1}^n \right)} \\
&\leq \sqrt{\norm{R\left( t^n_{i+1}, \cdot\right) - R\left( t^n_i, \cdot\right)}_{q-var; [0, T]}} \quad \mathrm{or} \quad \sqrt{\norm{R\left( t^n_{j+1}, \cdot\right) - R\left( t^n_j, \cdot\right)}_{q-var; [0, T]}} \\
&\leq C \, 2^{-\frac{n}{2\rho}}.
\end{align*}
Recall from Lemma \ref{productExp} that
\begin{align*}
&\exptn{\psi_{t^n_i} \psi_{t^n_j} \prod_{k=1}^6 I_1(h_k)}
=: \exptn{\mathcal{D}_{h_1, h_2, h_3, h_4, h_5, h_6}^6 \psi_{t^n_i} \psi_{t^n_j}} + \sum_{\sigma \in \mathcal{S}_6} C_{\sigma, 1} A_{\sigma, 1} + C_{\sigma, 2} A_{\sigma, 2} + C_{\sigma, 3} A_{\sigma, 3}.
\end{align*}
For the first term on the right side of the above expression, we have
\begin{align*}
\sum_{i, j = 0}^{2^n - 1} \exptn{\mathcal{D}_{h_1, h_2, h_3, h_4, h_5, h_6}^6 \psi_{t^n_i} \psi_{t^n_j}}
&\leq C \sum_{i, j = 0}^{2^n - 1} \prod_{k=1}^6 \norm{\Phi(h_k)}_{q-var; [0, T]} \\
&\leq C\, 2^{-2n \left( \frac{3}{\rho} - 1 \right)},
\end{align*}
which vanishes as $n \rightarrow \infty$ since $\rho < 2$. \par
For the $A_{\sigma, 1}$ terms we have
\begin{align*}
&\sum_{i, j= 0}^{2^n-1} \exptn{\mathcal{D}^4_{h_{\sigma(1)}, h_{\sigma(2)}, h_{\sigma(3)}, h_{\sigma(4)}} \psi_{t^n_i} \psi_{t^n_j}} \left\langle h_{\sigma(5)}, h_{\sigma(6)} \right\rangle_{\mathcal{H}_1^d} \\
&\fqquad \leq \sum_{i, j = 0}^{2^n-1} \prod_{k=1}^4 \norm{\Phi(h_{\sigma(k)})}_{q-var; [0, T]} \norm{h_{\sigma(5)}}_{\mathcal{H}_1^d} \norm{h_{\sigma(6)}}_{\mathcal{H}_1^d} \\
&\fqquad \leq C \, 2^{-2n \left( \frac{5}{2\rho} -1 \right)} \rightarrow 0,
\end{align*}
and similarly for the $A_{\sigma, 2}$ terms we have
\begin{align*}
&\sum_{i, j= 0}^{2^n-1} \exptn{\mathcal{D}^2_{h_{\sigma(1)}, h_{\sigma(2)}} \psi_{t^n_i} \psi_{t^n_j}} \left\langle h_{\sigma(3)}, h_{\sigma(4)} \right\rangle_{\mathcal{H}_1^d} \left\langle h_{\sigma(5)}, h_{\sigma(6)} \right\rangle_{\mathcal{H}_1^d} \\
&\qquad \qquad \leq \sum_{i, j = 0}^{2^n-1} \norm{\Phi(h_{\sigma(1)})}_{q-var; [0, T]} \norm{\Phi(h_{\sigma(2)})}_{q-var; [0, T]} \prod_{k=3}^6 \norm{h_{\sigma(k)}}_{\mathcal{H}_1^d}  \\
&\qquad \qquad \leq C \, 2^{-2n \left( \frac{2}{\rho} -1 \right)} \rightarrow 0.
\end{align*}
Finally for the $A_{\sigma, 3}$ terms we have two cases:
either
\begin{align*}
\left\langle h_{\sigma(1)}, h_{\sigma(2)} \right\rangle_{\mathcal{H}_1^d} \left\langle h_{\sigma(3)}, h_{\sigma(4)} \right\rangle_{\mathcal{H}_1^d} \left\langle h_{\sigma(5)}, h_{\sigma(6)} \right\rangle_{\mathcal{H}_1^d}
= R \begin{pmatrix}
t^n_i & t^n_{i+1} \\
t^n_j & t^n_{j+1}
\end{pmatrix}^3,
\end{align*}
or
\begin{align*}
\left\langle h_{\sigma(1)}, h_{\sigma(2)} \right\rangle_{\mathcal{H}_1^d} \left\langle h_{\sigma(3)}, h_{\sigma(4)} \right\rangle_{\mathcal{H}_1^d} \left\langle h_{\sigma(5)}, h_{\sigma(6)} \right\rangle_{\mathcal{H}_1^d}
= R \begin{pmatrix}
t^n_i & t^n_{i+1} \\
t^n_j & t^n_{j+1}
\end{pmatrix}
\sigma^2 \left( t^n_i, t^n_{i+1} \right) \sigma^2 \left( t^n_j, t^n_{j+1} \right).
\end{align*}
In either case, since
\begin{align*}
\abs{R \begin{pmatrix}
t^n_i & t^n_{i+1} \\
t^n_j & t^n_{j+1}
\end{pmatrix}},
\sigma^2 \left( t^n_i, t^n_{i+1} \right), \sigma^2 \left( t^n_j, t^n_{j+1} \right)
\leq \frac{C}{2^{\frac{n}{\rho}}},
\end{align*}
we have
\begin{align*}
\sum_{i, j = 0}^{2^n-1}
\exptn{\psi_{t^n_i} \psi_{t^n_j}} \left\langle h_{\sigma(1)}, h_{\sigma(2)} \right\rangle_{\mathcal{H}_1^d} &\left\langle h_{\sigma(3)}, h_{\sigma(4)} \right\rangle_{\mathcal{H}_1^d} \left\langle h_{\sigma(5)}, h_{\sigma(6)} \right\rangle_{\mathcal{H}_1^d} \\
&\leq C \left( \sum_{i,j = 0}^{2^n-1}
R \begin{pmatrix}
t^n_i & t^n_{i+1} \\
t^n_j & t^n_{j+1}
\end{pmatrix}^{\rho} \right)^{\frac{1}{\rho}}
\left( \sum_{i, j=0}^{2^n-1} 2^{\frac{-2n \rho'}{\rho}} \right)^{\frac{1}{\rho'}} \\
&\leq C \, \norm{R}_{\rho-var;[0, T]^2} 2^{-2n\left( \frac{1}{\rho} - \frac{1}{\rho'}\right)} \\
&\leq C \, \norm{R}_{\rho-var;[0, T]^2} 2^{-2n \left( \frac{2}{\rho} - 1\right)} \rightarrow 0.
\end{align*}
(b) We will now move on to show that
\begin{align}
\lim_{n \rightarrow \infty} \norm{\sum_{i=0}^{2^n - 1} \psi_{t^n_i} \left(
\mathbf{X}^3_{t^n_i, t^n_{i+1}} \right)^{NS}}_{L^2(\Omega)} = 0.
\end{align}
Let $X^{\pi(k)}$ denote the piece-wise linear approximation of $X$ over $\pi(k)$, and let \\
$\mathbf{X}^{\pi(k)} = \left(1, \, \mathbf{X}^1 (\pi(k)), \, \mathbf{X}^2 (\pi(k)), \, \mathbf{X}^3 (\pi(k)) \right) = S_3 \left( X^{\pi(k)} \right)$ denote its canonical lift to a geometric rough path. \par
Next, define
\begin{align*}
\left( \mathbf{X}^3_{s,t} \right)^{NS} (\Delta_l) := \left( \mathbf{X}^3_{s,t} \right)^{NS} \left( \pi (l+1) \right)  - \left( \mathbf{X}^3_{s,t} \right)^{NS} \left( \pi (l) \right).
\end{align*}
Since $\left( \mathbf{X}^3_{t_i^n, t_{i+1}^n} \right)^{NS} \left( \pi(n) \right) = 0$, we have
\begin{align*}
\left( \mathbf{X}^3_{t_i^n, t_{i+1}^n} \right)^{NS} = \lim_{m\rightarrow\infty}\sum_{k=1}^m
\left( \mathbf{X}^3_{t_i^n, t_{i+1}^n} \right)^{NS} \left( \Delta_{n+k} \right)  \text{ for every }n\in \mathbb{N} \; \mathrm{and} \; i = 0, 1, \ldots, 2^n - 1,
\end{align*}
where the limit is taken in $L^2 (\Omega)$. \par
We want to show that
\begin{align*}
\norm{\sum_{i=0}^{2^n - 1} \psi_{t^n_i} \left( \left( \mathbf{X}^3_{t_i^n, t_{i+1}^n} \right)^{NS} (\pi (n+m)) \right)^{(a, b, c)}}_{L^2(\Omega)} \rightarrow 0
\end{align*}
uniformly for all $m$ as $n \rightarrow \infty$. To begin, let
\begin{align} \label{sDefn}
s_u^{k, i} := t_i^n + \frac{u}{2^{n+k}} = t_{u + i 2^k}^{n+k},
\end{align}
and we will denote the intervals
\begin{align} \label{LRDefn}
\begin{split}
&\Delta_{u^L}^i := \left[ s^{k+1, i}_{2u}, s^{k+1, i}_{2u+1} \right], \quad \Delta_{u^R}^i := \left[ s^{k+1, i}_{2u+1}, s^{k+1, i}_{2u+2} \right], \\
&\Delta_u^i := \Delta_{u^L}^i \cup \Delta_{u^R}^i = \left[ s^{k, i}_u, s^{k, i}_{u+1} \right] \subset \left[ t^n_i, t^n_{i+1} \right], \quad \forall u = 0, \ldots 2^k - 1.
\end{split}
\end{align}
Note that we suppress the dependence on $k$ and $n$ in the notation for the variables on the left.
The following computation on $G^3(\mathbb{R}^d)$ can be verified easily; for $f, g \in \mathbb{R}^d$ we have
\begin{align*}
&\exp(f) \otimes \exp(g)
= \left( 1, \, f + g, \, \frac{(f + g)^{\otimes 2}}{2} + \frac{1}{2} [f, g], \, \frac{(f + g)^{\otimes 3}}{6} + N(f, g) \right),
\end{align*}
where
\begin{align*}
N(f, g) := \frac{1}{4} \big( (f+g) \otimes [f, g] + [f, g] \otimes (f + g) \big) + \frac{1}{12} \big( [f, [f, g]] + [g, [g, f]] \big).
\end{align*}
Using the above expression with $f = X_{\Delta_{u^L}^i}$ and $g = X_{\Delta_{u^R}^i}$, for $k = 1, \ldots, m$ we obtain the following identity on $T^3 \left( \mathbb{R}^d \right)$:
\begin{align*}
&\bigotimes_{u=0}^{2^k-1} \exp \left( X_{\Delta_{u^L}^i} \right) \otimes \exp \left( X_{\Delta^i_{u^R}} \right) - \bigotimes_{u=0}^{2^k-1} \exp \left( X_{\Delta_u^i} \right) \\
&\quad \qquad = \bigotimes_{u=0}^{2^k-1} \left( 1, X_{\Delta_u^i}, \frac{1}{2} X_{\Delta_u^i}^{\otimes 2}, \frac{1}{6} X_{\Delta_u^i}^{\otimes 3} \right)
+ \left(0, 0, \frac{1}{2} \left[ X_{\Delta^i_{u^L}}, X_{\Delta_{u^R}^i} \right], 0 \right)
+ \left(0, 0, 0, N \left( X_{\Delta^i_{u^L}}, X_{\Delta^i_{u^R}} \right) \right) \\
&\fqquad- \bigotimes_{u = 0}^{2^k-1} \left( 1, X_{\Delta_u^i}, \frac{1}{2} X_{\Delta_u^i}^{\otimes 2}, \frac{1}{6} X_{\Delta_u^i}^{\otimes 3} \right) \\
&\quad \qquad = \sum_{u=0}^{2^k-1} \left( 0, \, 0, \, \frac{1}{2} \left[ X_{\Delta^i_{u^L}}, X_{\Delta^i_{u^R}} \right], \, M\left( X_{\Delta^i_{u^L}}, X_{\Delta^i_{u^R}} \right) + N\left( X_{\Delta^i_{u^L}}, X_{\Delta^i_{u^R}} \right) \right)
\end{align*}
where
\begin{align*}
M\left( X_{\Delta^i_{u^L}}, X_{\Delta^i_{u^R}} \right)
&:= \sum_{r=0}^{u-1} X_{\Delta^i_r} \otimes \frac{1}{2} \left[ X_{\Delta^i_{u^L}}, X_{\Delta^i_{u^R}} \right] + \frac{1}{2} \left[ X_{\Delta^i_{u^L}}, X_{\Delta^i_{u^R}} \right] \otimes \sum_{r=u+1}^{2^k-1} X_{\Delta^i_r} \\
&= X_{t^n_i, s^{k, i}_u} \otimes \frac{1}{2} \left[ X_{\Delta^i_{u^L}}, X_{\Delta^i_{u^R}} \right] + \frac{1}{2} \left[ X_{\Delta^i_{u^L}}, X_{\Delta^i_{u^R}} \right] \otimes X_{s^{k,i}_{u+1}, t^n_{i+1}}.
\end{align*}
This means that
\begin{align*}
\mathbf{X}^3_{t^n_i, t^n_{i+1}} (\pi(n+k+1)) - \mathbf{X}^3_{t^n_i, t^n_{i+1}} (\pi(n+k))
&= \sum_{u=0}^{2^k-1} M_u + N_u,
\end{align*}
where we use $M_u$ and $N_u$ as short-hand for $M\left( X_{\Delta^i_{u^L}}, X_{\Delta^i_{u^R}} \right)$ and $N\left( X_{\Delta^i_{u^L}}, X_{\Delta^i_{u^R}} \right)$ respectively.
This in turn gives us
\begin{align*}
\left( \mathbf{X}^3_{t_i^n, t_{i+1}^n} \right)^{NS} \left( \Delta_{n+k} \right) = \sum_{u=0}^{2^k-1} M_u + N_u,
\end{align*}
since
\begin{align*}
\mathbf{X}^3_{t^n_i, t^n_{i+1}} (\pi(n+k+1))& - \mathbf{X}^3_{t^n_i, t^n_{i+1}} (\pi(n+k)) \\
&= \left( \mathbf{X}^3_{t^n_i, t^n_{i+1}} \right)^S (\pi(n+k+1)) - \left( \mathbf{X}^3_{t^n_i, t^n_{i+1}} \right)^S (\pi(n+k)) + \left( \mathbf{X}^3_{t_i^n, t_{i+1}^n} \right)^{NS} \left( \Delta_{n+k} \right) \\
&= \exp \left( X_{t^n_i, t^n_{i+1}} \right) - \exp \left( X_{t^n_i, t^n_{i+1}} \right) + \left( \mathbf{X}^3_{t_i^n, t_{i+1}^n} \right)^{NS} \left( \Delta_{n+k} \right) \\
&= \left( \mathbf{X}^3_{t_i^n, t_{i+1}^n} \right)^{NS} \left( \Delta_{n+k} \right).
\end{align*}
Thus, we obtain
\begin{align} \label{3rdQuant}
\begin{split}
&\exptn{\left( \sum_{i=0}^{2^n - 1} \psi_{t^n_i} \left( \left( \mathbf{X}^3_{t_i^n, t_{i+1}^n} \right)^{NS} (\pi (n+m)) \right)^{(a, b, c)} \right)^2} \\
&\quad \qquad \qquad \qquad= \exptn{\left( \sum_{i=0}^{2^n - 1} \psi_{t^n_i} \sum_{k=1}^m \left( \left( \mathbf{X}^3_{t_i^n, t_{i+1}^n} \right)^{NS} (\Delta_{n+k}) \right)^{(a, b, c)} \right)^2} \\
&\quad \qquad \qquad \qquad= \sum_{i, j=0}^{2^n - 1} \exptn{ \psi_{t^n_i} \psi_{t^n_j} \sum_{k=1}^m \left( \left( \mathbf{X}^3_{t_i^n, t_{i+1}^n} \right)^{NS} \left( \Delta_{n+k} \right) \right)^{(a, b, c)}  \sum_{l=1}^m \left( \left( \mathbf{X}^3_{t_j^n, t_{j+1}^n} \right)^{NS} \left( \Delta_{n+l} \right) \right)^{(a, b, c)}} \\
&\quad \qquad \qquad \qquad= \sum_{i,j=0}^{2^n-1} \sum_{k,l=1}^m \sum_{u=0}^{2^k-1} \sum_{v=0}^{2^l-1} \exptn{\psi_{t^n_i} \psi_{t^n_j} \left( M_u + N_u\right)^{(a, b, c)} \left( M_v + N_v \right)^{(a, b, c)}}.
\end{split}
\end{align}
In what follows, it does not matter to the analysis which particular subinterval of $\Delta_u^i$, $\Delta_v^j$, $\Delta^n_i$ or $\Delta^n_j$ is present in the terms. Hence we will use the notation
\begin{align*}
&\Delta_{u^*} = \Delta_{u^L}^i, \Delta_{u^R}^i \: \mathrm{or} \: \Delta_u^i, \quad \Delta_{v^*} = \Delta_{v^L}^j, \Delta_{v^R}^j \: \mathrm{or} \: \Delta_v^j, \\
&\Delta_{i^*} = \left[t^n_i, s^{k,i}_u \right] \: \mathrm{or} \: \left[ s^{k,i}_{u+1}, t^n_{i+1} \right], \quad \Delta_{j^*} = \left[t^n_j, s^{l,j}_v \right] \: \mathrm{or} \: \left[ s^{l,j}_{v+1}, t^n_{j+1} \right],
\end{align*}
and
\begin{align*}
R \begin{pmatrix}
\Delta_{u^*} \\
\Delta_{v^*}
\end{pmatrix}
:= \left\langle \mathds{1}_{\Delta_u^*}, \mathds{1}_{\Delta_v^*} \right\rangle_{\mathcal{H}_1}
= R \begin{pmatrix}
a_1 & a_2 \\
b_1 & b_2
\end{pmatrix},
\end{align*}
where $[a_1, a_2] = \Delta_{u^L}^i, \Delta_{u^R}^i \: \mathrm{or} \: \Delta_u^i$, and $[b_1, b_2] = \Delta_{v^L}^j, \Delta_{v^R}^j \: \mathrm{or} \: \Delta_v^j$. \par
$R \begin{pmatrix}
\Delta_{u^*} \\
\Delta_{i^*}
\end{pmatrix}$,
$R \begin{pmatrix}
\Delta_{u^*} \\
\Delta_{j^*}
\end{pmatrix}$,
$R \begin{pmatrix}
\Delta_{u^*} \\
\Delta_{u^*}
\end{pmatrix}$,
$R \begin{pmatrix}
\Delta_{v^*} \\
\Delta_{i^*}
\end{pmatrix}$,
$R \begin{pmatrix}
\Delta_{v^*} \\
\Delta_{j^*}
\end{pmatrix}$ and
$R \begin{pmatrix}
\Delta_{v^*} \\
\Delta_{v^*}
\end{pmatrix}$ can be defined in the same manner, and we have the bounds
\begin{align} \label{qrBounds}
\begin{split}
\abs{R \begin{pmatrix}
\Delta_{u^*} \\
\Delta_{v^*}
\end{pmatrix}},
\abs{R \begin{pmatrix}
\Delta_{u^*} \\
\Delta_{u^*}
\end{pmatrix}},
\abs{R \begin{pmatrix}
\Delta_{u^*} \\
\Delta_{i^*}
\end{pmatrix}},
\abs{R \begin{pmatrix}
\Delta_{u^*} \\
\Delta_{j^*}
\end{pmatrix}}
&\leq \norm{R \left( \Delta_{u^*}, \cdot \right)}_{q-var; [0, T]} \\
&= \norm{\Phi \left( \mathds{1}_{\Delta_{u^*}} \right)}_{q-var; [0, T]}
\leq \frac{C}{2^{\frac{n+k}{\rho}}}, \\
\abs{R \begin{pmatrix}
\Delta_{u^*} \\
\Delta_{v^*}
\end{pmatrix}},
\abs{R \begin{pmatrix}
\Delta_{v^*} \\
\Delta_{v^*}
\end{pmatrix}},
\abs{R \begin{pmatrix}
\Delta_{v^*} \\
\Delta_{i^*}
\end{pmatrix}},
\abs{R \begin{pmatrix}
\Delta_{v^*} \\
\Delta_{j^*}
\end{pmatrix}}
&\leq \norm{R \left( \Delta_{v^*}, \cdot \right)}_{q-var; [0, T]} \\
&= \norm{\Phi \left( \mathds{1}_{\Delta_{v^*}} \right)}_{q-var; [0, T]}
\leq \frac{C}{2^{\frac{n+l}{\rho}}}.
\end{split}
\end{align}
Using the notation 
\begin{align*}
R_{\Delta_u^i \times \Delta_v^j} : = \abs{R \begin{pmatrix}
s^{k, i}_{2u} & s^{k, i}_{2u+1} \\
s^{l, j}_{2v} & s^{l, j}_{2v+1}
\end{pmatrix}} + 
\abs{R \begin{pmatrix}
s^{k, i}_{2u+1} & s^{k, i}_{2u+2} \\
s^{l, j}_{2v} & s^{l, j}_{2v+1}
\end{pmatrix}} + 
\abs{R \begin{pmatrix}
s^{k, i}_{2u} & s^{k, i}_{2u+1} \\
s^{l, j}_{2v+1} & s^{l, j}_{2v+2}
\end{pmatrix}} + 
\abs{R \begin{pmatrix}
s^{k, i}_{2u+1} & s^{k, i}_{2u+2} \\
s^{l, j}_{2v+1} & s^{l, j}_{2v+2}
\end{pmatrix}},
\end{align*}
note that
\begin{align*}
\sum_{i, j=0}^{2^n-1} \sum_{u=0}^{2^k-1} \sum_{v=0}^{2^l-1} \abs{R \begin{pmatrix}
\Delta_{u^*} \\
\Delta_{v^*}
\end{pmatrix}}^{\rho}
\leq \sum_{i, j=0}^{2^n-1} \sum_{u=0}^{2^k-1} \sum_{v=0}^{2^l-1} R^{\rho}_{\Delta^i_u \times \Delta_v^j}
\leq 4^{\rho} \norm{R}_{\rho-var; [0, T]^2}^{\rho}
\end{align*}
for all $k,l \in \mathbb{N}$. \par
For $k = 1, \ldots 6$, let $y_k$ denote $a, b$ or $c$. Returning to \eqref{3rdQuant}, we see that the last line can be expanded to include terms of the type $M_u^{(a, b, c)} M_v^{(a, b, c)}$:
\begin{align} \label{typeMM}
\exptn{\psi_{t^n_i} \psi_{t^n_j} X^{(y_1)}_{\Delta_{u^*}} X^{(y_2)}_{\Delta_{u^*}} X^{(y_3)}_{\Delta_{i^*}} X^{(y_4)}_{\Delta_{v^*}} X^{(y_5)}_{\Delta_{v^*}} X^{(y_6)}_{\Delta_{j^*}} },
\end{align}
terms coming from $N_u^{(a, b, c)} N_v^{(a, b, c)}$:
\begin{align} \label{typeNN}
\exptn{\psi_{t^n_i} \psi_{t^n_j} X^{(y_1)}_{\Delta_{u^*}} X^{(y_2)}_{\Delta_{u^*}} X^{(y_3)}_{\Delta_{u^*}} X^{(y_4)}_{\Delta_{v^*}} X^{(y_5)}_{\Delta_{v^*}} X^{(y_6)}_{\Delta_{v^*}} },
\end{align}
and terms arising from $M_u^{(a, b, c)} N_v^{(a, b, c)}$:
\begin{align} \label{typeMN}
\exptn{\psi_{t^n_i} \psi_{t^n_j} X^{(y_1)}_{\Delta_{u^*}} X^{(y_2)}_{\Delta_{u^*}} X^{(y_3)}_{\Delta_{i^*}} X^{(y_4)}_{\Delta_{v^*}} X^{(y_5)}_{\Delta_{v^*}} X^{(y_6)}_{\Delta_{v^*}} }.
\end{align}
To account for the remaining $N_u^{(a, b, c)} M_v^{(a, b, c)}$ terms, we simply swap $u$ and $v$ in the third case.
Note also that with our short-hand notation, as an example, $X_{\Delta_{u^*}}^{(y_1)}$ may not be equal to $X_{\Delta_{u^*}}^{(y_2)}$ even if $y_1 = y_2$ since $\Delta_{u^*}$ may be one of several intervals. \par
Since $M_u$ is anti-symmetric with respect to $X_{\Delta_{u^L}}$ and $X_{\Delta_{u^R}}$, we can assume that $y_1 \neq y_2$ in \eqref{typeMM} and \eqref{typeMN}, and $y_4 \neq y_5$ in \eqref{typeMM}.
In each of the three cases, we will use $I_1(h_k)$ to denote $X^{(y_k)}$ for $k = 1, \ldots, 6$; for example in \eqref{typeMM}, $h_1 := \mathds{1}_{\Delta_u^*}^{(y_1)}$ and $I_1 (h_1) = X^{(y_1)}_{\Delta_{u^*}}$. Now applying Lemma \ref{productExp}, we have
\begin{align*}
&\exptn{\psi_{t^n_i} \psi_{t^n_j} \prod_{k=1}^6 I_1(h_k)}
=: \exptn{\mathcal{D}_{h_1, h_2, h_3, h_4, h_5, h_6}^6 \psi_{t^n_i} \psi_{t^n_j}} + \sum_{\sigma \in \mathcal{S}_6} C_{\sigma, 1} A_{\sigma, 1} + C_{\sigma, 2} A_{\sigma, 2} + C_{\sigma, 3} A_{\sigma, 3},
\end{align*}
where we recall that
\begin{align*}
&A_{\sigma, 1} := \exptn{\mathcal{D}_{h_{\sigma(1)}, h_{\sigma(2)}, h_{\sigma(3)}, h_{\sigma(4)}}^4 \psi_{t^n_i} \psi_{t^n_j}} \left\langle h_{\sigma(5)}, h_{\sigma(6)} \right\rangle_{\mathcal{H}_1^d}, \\
&A_{\sigma, 2} := \exptn{\mathcal{D}_{h_{\sigma(1)}, h_{\sigma(2)}}^2 \psi_{t^n_i} \psi_{t^n_j}} \left\langle h_{\sigma(3)}, h_{\sigma(4)} \right\rangle_{\mathcal{H}_1^d} \left\langle h_{\sigma(5)}, h_{\sigma(6)} \right\rangle_{\mathcal{H}_1^d}, \\
&A_{\sigma, 3} := \exptn{\psi_{t^n_i} \psi_{t^n_j}} \left\langle h_{\sigma(1)}, h_{\sigma(2)} \right\rangle_{\mathcal{H}_1^d} \left\langle h_{\sigma(3)}, h_{\sigma(4)} \right\rangle_{\mathcal{H}_1^d} \left\langle h_{\sigma(5)}, h_{\sigma(6)} \right\rangle_{\mathcal{H}_1^d}.
\end{align*}
We will show that for each of these terms, the sum over all the sub-intervals converges to zero as $n \rightarrow \infty$. \par
For the first term, from \eqref{propCond} we have
\begin{align*}
\exptn{\mathcal{D}_{h_1, h_2, h_3, h_4, h_5, h_6}^6 \psi_{t^n_i} \psi_{t^n_j}}
\leq C \prod_{i=1}^6 \norm{\Phi(h_i)}_{q-var; [0, T]}.
\end{align*}
Looking at each of the three types of terms \eqref{typeMM}, \eqref{typeNN} and \eqref{typeMN}, we see that at least two of the $h_i$'s are $\mathds{1}_{\Delta_{u^*}}$, and another two of the $h_i$'s are $\mathds{1}_{\Delta_{v^*}}$. Thus we obtain
\begin{align*}
\sum_{i, j = 0}^{2^n - 1} \sum_{k,l = 1}^m \sum_{u = 0}^{2^k - 1} \sum_{v = 0}^{2^l-1} \exptn{\mathcal{D}_{h_1, h_2, h_3, h_4, h_5, h_6}^6 \psi_{t^n_i} \psi_{t^n_j}}
& \leq C \sum_{i, j = 0}^{2^n - 1} \sum_{k,l = 1}^m \sum_{u = 0}^{2^k - 1} \sum_{v = 0}^{2^l-1} \frac{1}{2^{(n+k)\frac{2}{\rho}} } \frac{1}{2^{(n+l)\frac{2}{\rho}} } \\
&\leq C \sum_{i, j = 0}^{2^n - 1} \frac{1}{2^{2n \left( \frac{2}{\rho}\right)}} \sum_{k,l=1}^{\infty} \frac{1}{2^{k \left( \frac{2}{\rho} - 1 \right)}} \frac{1}{2^{l \left( \frac{2}{\rho} - 1 \right)}} \\
&  \leq \frac{C}{2^{2n \left( \frac{2}{\rho} - 1\right)}} \rightarrow 0
\end{align*}
since $\rho < 2$. \par
For the $A_{\sigma, 1}$ terms, we have two cases:
\begin{enumerate}[(i)]
\item
$\left\langle h_{\sigma(5)}, h_{\sigma(6)} \right\rangle_{\mathcal{H}_1^d} = R \begin{pmatrix}
\Delta_{u^*} \\
\Delta_{v^*}
\end{pmatrix}$: \par
In all three types of terms \eqref{typeMM}, \eqref{typeNN} and \eqref{typeMN}, at least one of $\left\{ h_{\sigma(1)}, h_{\sigma(2)}, h_{\sigma(3)}, h_{\sigma(4)} \right\}$ equals $\mathds{1}_{\Delta_{u^*}}$, and another one in the set equals $\mathds{1}_{\Delta_{v^*}}$. Applying the bounds in \eqref{qrBounds}, we get
\begin{align*}
\sum_{i, j = 0}^{2^n - 1} \sum_{k,l = 1}^m \sum_{u = 0}^{2^k - 1} \sum_{v=0}^{2^l-1} &\exptn{\mathcal{D}_{h_{\sigma(1)}, h_{\sigma(2)}, h_{\sigma(3)}, h_{\sigma(4)}}^4 \psi_{t^n_i} \psi_{t^n_j}} \left\langle h_{\sigma(5)}, h_{\sigma(6)} \right\rangle_{\mathcal{H}_1^d}, \\
&\leq C \sum_{i, j = 0}^{2^n - 1} \sum_{k,l = 1}^m \sum_{u = 0}^{2^k - 1} \sum_{v=0}^{2^l-1} \prod_{r=1}^4 \norm{\Phi(h_{\sigma(r)})}_{q-var; [0, T]} \abs{\left\langle h_{\sigma(5)}, h_{\sigma(6)} \right\rangle_{\mathcal{H}_1^d}} \\
&\leq C \sum_{i, j = 0}^{2^n - 1} \sum_{k,l = 1}^m \sum_{u = 0}^{2^k - 1} \sum_{v=0}^{2^l-1}
2^{\frac{-(n+k)}{\rho}} 2^{\frac{-(n+l)}{\rho}}
\abs{ R \begin{pmatrix}
\Delta_{u^*} \\
\Delta_{v^*}
\end{pmatrix}} \\
&\leq C \sum_{k,l = 1}^m \left( \sum_{i, j = 0}^{2^n - 1} \sum_{u = 0}^{2^k - 1} \sum_{v = 0}^{2^l-1} R \begin{pmatrix}
\Delta_{u^*} \\
\Delta_{v^*}
\end{pmatrix}^{\rho} \right)^{\frac{1}{\rho}}
\left( \sum_{i, j = 0}^{2^n - 1} \sum_{u = 0}^{2^k - 1} \sum_{v = 0}^{2^l-1} 2^{-(n+k)\left(\frac{\rho'}{\rho}\right)} 2^{-(n+l)\left(\frac{\rho'}{\rho}\right)} \right)^{\frac{1}{\rho'}} \\
&\leq C \, 2^{-2n \left( \frac{1}{\rho} - \frac{1}{\rho'} \right)} \norm{R}_{\rho-var;[0, T]^2} \sum_{k,l=1}^m 2^{-k \left( \frac{1}{\rho} - \frac{1}{\rho'} \right)} 2^{-l \left( \frac{1}{\rho} - \frac{1}{\rho'} \right)} \\
&\leq C \, 2^{-2n \left( \frac{2}{\rho} - 1 \right)} \norm{R}_{\rho-var;[0, T]^2} \sum_{k,l=1}^{\infty} 2^{-k \left( \frac{2}{\rho} - 1 \right)} 2^{-l \left( \frac{2}{\rho} - 1 \right)} \rightarrow 0
\end{align*}
\item
$\left\langle h_{\sigma(5)}, h_{\sigma(6)} \right\rangle_{\mathcal{H}_1^d} \neq R \begin{pmatrix}
\Delta_{u^*} \\
\Delta_{v^*}
\end{pmatrix}$: \par
We will go through each of the three types of terms \eqref{typeMM}, \eqref{typeNN} and \eqref{typeMN} to count the number of quantities with increments $\Delta_{u^*}$ or $\Delta_{v^*}$, which yield the factors $2^{\frac{-(n+k)}{\rho}}$ and $2^{\frac{-(n+l)}{\rho}}$ respectively.
\begin{enumerate} [(a)]
\item
$M_u^{(a, b, c)} M_v^{(a, b, c)}$ terms: \par
We have five possibilities:
\begin{align} \label{56choices}
\left\langle h_{\sigma(5)}, h_{\sigma(6)} \right\rangle_{\mathcal{H}_1^d}
= R \begin{pmatrix}
\Delta_{u^*} \\
\Delta_{i^*}
\end{pmatrix},
R \begin{pmatrix}
\Delta_{u^*} \\
\Delta_{j^*}
\end{pmatrix},
R \begin{pmatrix}
\Delta_{v^*} \\
\Delta_{i^*}
\end{pmatrix},
R \begin{pmatrix}
\Delta_{v^*} \\
\Delta_{j^*}
\end{pmatrix} \: \mathrm{or} \:
R \begin{pmatrix}
\Delta_{i^*} \\
\Delta_{j^*}
\end{pmatrix};
\end{align}
we need not consider the cases $R \begin{pmatrix}
\Delta_{u^*} \\
\Delta_{u^*}
\end{pmatrix}$ or $R \begin{pmatrix}
\Delta_{v^*} \\
\Delta_{v^*}
\end{pmatrix}$ since $y_1 \neq y_2$ and $y_4 \neq y_5$ in \eqref{typeMM}. \par
If $\left\langle h_{\sigma(5)}, h_{\sigma(6)} \right\rangle_{\mathcal{H}_1^d}$ is equal to either of the first two quantities on the right of \eqref{56choices}, then one of $\left\{ h_{\sigma(1)}, h_{\sigma(2)}, h_{\sigma(3)}, h_{\sigma(4)} \right\}$ must be equal to $\mathds{1}_{\Delta_{u^*}}$ and another two in the set must be equal to $\mathds{1}_{\Delta_{v^*}}$. If $\left\langle h_{\sigma(5)}, h_{\sigma(6)} \right\rangle_{\mathcal{H}_1^d}$ is equal to the third or the fourth quantity in \eqref{56choices}, we have the same count with $u$ and $v$ switched. \par
If $\left\langle h_{\sigma(5)}, h_{\sigma(6)} \right\rangle_{\mathcal{H}_1^d} = R \begin{pmatrix}
\Delta_{i^*} \\
\Delta_{j^*}
\end{pmatrix}$, then without loss of generality,
\begin{align*}
h_{\sigma(1)} \: \mathrm{and} \: h_{\sigma(2)} = \mathds{1}_{\Delta_{u^*}}, \quad
h_{\sigma(3)} \: \mathrm{and} \: h_{\sigma(4)} = \mathds{1}_{\Delta_{v^*}}.
\end{align*}
\item
$N_u^{(a, b, c)} N_v^{(a, b, c)}$ terms: \par
If $\left\langle h_{\sigma(5)}, h_{\sigma(6)} \right\rangle_{\mathcal{H}_1^d}
= R \begin{pmatrix}
\Delta_{u^*} \\
\Delta_{u^*}
\end{pmatrix}$, then one of $\left\{ h_{\sigma(1)}, h_{\sigma(2)}, h_{\sigma(3)}, h_{\sigma(4)} \right\}$ must equal $\mathds{1}_{\Delta_{u^*}}$ and another two in the set must equal $\mathds{1}_{\Delta_{v^*}}$. By switching $u$ and $v$, we can resolve the only other case $\left\langle h_{\sigma(5)}, h_{\sigma(6)} \right\rangle_{\mathcal{H}_1^d}
= R \begin{pmatrix}
\Delta_{v^*} \\
\Delta_{v^*}
\end{pmatrix}$ similarly.
\item
$M_u^{(a, b, c)} N_v^{(a, b, c)}$ terms: \par
There are only three possibilities
\begin{align*}
\left\langle h_{\sigma(5)}, h_{\sigma(6)} \right\rangle_{\mathcal{H}_1^d}
= R \begin{pmatrix}
\Delta_{u^*} \\
\Delta_{i^*}
\end{pmatrix},
R \begin{pmatrix}
\Delta_{v^*} \\
\Delta_{v^*}
\end{pmatrix} \: \mathrm{or} \:
R \begin{pmatrix}
\Delta_{v^*} \\
\Delta_{i^*}
\end{pmatrix},
\end{align*}
and we need not consider the case $R \begin{pmatrix}
\Delta_{u^*} \\
\Delta_{u^*}
\end{pmatrix}$ since $y_1 \neq y_2$ in \eqref{typeMN}. If $\left\langle h_{\sigma(5)}, h_{\sigma(6)} \right\rangle_{\mathcal{H}_1^d}$ is equal to $R \begin{pmatrix}
\Delta_{u^*} \\
\Delta_{i^*}
\end{pmatrix}$, then one of $\left\{ h_{\sigma(1)}, h_{\sigma(2)}, h_{\sigma(3)}, h_{\sigma(4)} \right\}$ must be equal to $\mathds{1}_{\Delta_{u^*}}$ and another two in the set must be equal to $\mathds{1}_{\Delta_{v^*}}$. If $\left\langle h_{\sigma(5)}, h_{\sigma(6)} \right\rangle_{\mathcal{H}_1^d}$ is equal to the second or third quantity, the same count applies with $u$ and $v$ switched.
\end{enumerate}
Thus in each case, applying the bounds in \eqref{qrBounds} yields
\begin{align} \label{A1Bound}
\prod_{r=1}^4 \norm{\Phi(h_{\sigma(r)})}_{q-var; [0, T]} \abs{\left\langle h_{\sigma(5)}, h_{\sigma(6)} \right\rangle_{\mathcal{H}_1^d}}
\leq C \, 2^{\frac{-2(n+k)}{\rho}} 2^{\frac{-2(n+l)}{\rho}},
\end{align}
which gives us
\begin{align*}
&\sum_{i, j= 0}^{2^n-1} \sum_{k,l=1}^m \sum_{u=0}^{2^k-1} \sum_{v=0}^{2l-1} \exptn{\mathcal{D}^4_{h_{\sigma(1)}, h_{\sigma(2)}, h_{\sigma(3)}, h_{\sigma(4)}} \psi_{t^n_i} \psi_{t^n_j}} \left\langle h_{\sigma(5)}, h_{\sigma(6)} \right\rangle_{\mathcal{H}_1^d} \\
&\fqquad \leq C \sum_{i, j= 0}^{2^n-1} \sum_{k,l=1}^m \sum_{u=0}^{2^k-1} \sum_{v=0}^{2l-1} 2^{-(n+k) \frac{2}{\rho}} 2^{-(n+l)  \frac{2}{\rho}} \\
&\fqquad \leq C \, 2^{-2n \left( \frac{2}{\rho} -1 \right)} \sum_{k,l=1}^{\infty} 2^{-k \left(\frac{2}{\rho} - 1 \right)} 2^{-l \left(\frac{2}{\rho} - 1 \right)} \rightarrow 0.
\end{align*}
\end{enumerate}
For the $A_{\sigma, 2}$ terms, when we consider $\left\langle h_{\sigma(3)}, h_{\sigma(4)} \right\rangle_{\mathcal{H}_1^d}$ and $\left\langle h_{\sigma(5)}, h_{\sigma(6)} \right\rangle_{\mathcal{H}_1^d}$, we have three cases: either both, one, or none of them are equal to $R \begin{pmatrix}
\Delta_{u^*} \\
\Delta_{v^*}
\end{pmatrix}$.
\begin{enumerate}[(i)]
\item
$\left\langle h_{\sigma(3)}, h_{\sigma(4)} \right\rangle_{\mathcal{H}_1^d} \: \mathrm{and} \: \left\langle h_{\sigma(5)}, h_{\sigma(6)} \right\rangle_{\mathcal{H}_1^d} = R \begin{pmatrix}
\Delta_{u^*} \\
\Delta_{v^*}
\end{pmatrix}$: \par
(Note that this does not imply that they are equal to one another since $\Delta_{u^*}$ and $\Delta_{v^*}$ can be one of several intervals.) \par
Observe that
\begin{align*}
\abs{\left\langle h_{\sigma(3)}, h_{\sigma(4)} \right\rangle_{\mathcal{H}_1^d}}
\leq \abs{R \begin{pmatrix}
\Delta_{u^*} \\
\Delta_{v^*}
\end{pmatrix}}^{\frac{\rho}{2}} \abs{R \begin{pmatrix}
\Delta_{u^*} \\
\Delta_{v^*}
\end{pmatrix}}^{1 - \frac{\rho}{2}}
\leq C \, R^{\frac{\rho}{2}}_{\Delta_u^i \times \Delta_v^j} \, 2^{\frac{-(n+k)}{\rho} \left( 1 - \frac{\rho}{2} \right) },
\end{align*}
and
\begin{align*}
\abs{\left\langle h_{\sigma(5)}, h_{\sigma(6)} \right\rangle_{\mathcal{H}_1^d}}
\leq \abs{R \begin{pmatrix}
\Delta_{u^*} \\
\Delta_{v^*}
\end{pmatrix}}^{\frac{\rho}{2}} \abs{R \begin{pmatrix}
\Delta_{u^*} \\
\Delta_{v^*}
\end{pmatrix}}^{1 - \frac{\rho}{2}}
\leq C \, R^{\frac{\rho}{2}}_{\Delta_u^i \times \Delta_v^j} \, 2^{\frac{-(n+l)}{\rho} \left( 1 - \frac{\rho}{2} \right) }.
\end{align*}
Thus we obtain
\begin{align*}
\sum_{i, j = 0}^{2^n - 1} \sum_{k,l=1}^m \sum_{u = 0}^{2^k - 1} \sum_{v=0}^{2^l-1} &\exptn{\mathcal{D}_{h_{\sigma(1)}, h_{\sigma(2)}}^2 \psi_{t^n_i} \psi_{t^n_j}} \left\langle h_{\sigma(3)}, h_{\sigma(4)} \right\rangle_{\mathcal{H}_1^d} \left\langle h_{\sigma(5)}, h_{\sigma(6)} \right\rangle_{\mathcal{H}_1^d} \\
&\leq C \sum_{k,l=1}^m 2^{\frac{-(n+k)}{\rho} \left( 1 - \frac{\rho}{2} \right) } 2^{\frac{-(n+l)}{\rho} \left( 1 - \frac{\rho}{2} \right) } \sum_{i, j = 0}^{2^n - 1} \sum_{u = 0}^{2^k - 1} \sum_{v = 0}^{2^l-1} R^{\rho}_{\Delta_u^i \times \Delta_v^j} \\
&\leq C \, 2^{-2n \left( \frac{1}{\rho} - \frac{1}{2} \right)} \sum_{k,l=1}^{\infty} 2^{-k \left( \frac{1}{\rho} - \frac{1}{2} \right)} 2^{-l \left( \frac{1}{\rho} - \frac{1}{2}\right)} \norm{R}^{\rho}_{\rho-var;[0,T]^2} \rightarrow 0,
\end{align*}
since $\frac{1}{\rho} - \frac{1}{2} > 0$.
\item
WLOG, assume $\left\langle h_{\sigma(3)}, h_{\sigma(4)} \right\rangle_{\mathcal{H}_1^d} = R \begin{pmatrix}
\Delta_{u^*} \\
\Delta_{v^*}
\end{pmatrix}, \quad
\left\langle h_{\sigma(5)}, h_{\sigma(6)} \right\rangle_{\mathcal{H}_1^d} \neq R \begin{pmatrix}
\Delta_{u^*} \\
\Delta_{v^*}
\end{pmatrix}$: \par
As before, we will use the bounds in \eqref{qrBounds} to show that
\begin{align} \label{A2Bound}
\norm{\Phi(h_{\sigma(1)})}_{q-var; [0, T]} \norm{\Phi(h_{\sigma(2)})}_{q-var; [0, T]} \abs{\left\langle h_{\sigma(5)}, h_{\sigma(6)} \right\rangle_{\mathcal{H}_1^d}}
\leq C \, 2^{\frac{-(n+k)}{\rho}} 2^{\frac{-(n+l)}{\rho}}.
\end{align}
\begin{enumerate} [(a)]
\item
$M_u^{(a, b, c)} M_v^{(a, b, c)}$ terms: \par
Again we have five possibilities,
\begin{align} \label{56choices2}
\left\langle h_{\sigma(5)}, h_{\sigma(6)} \right\rangle_{\mathcal{H}_1^d}
= R \begin{pmatrix}
\Delta_{u^*} \\
\Delta_{i^*}
\end{pmatrix},
R \begin{pmatrix}
\Delta_{u^*} \\
\Delta_{j^*}
\end{pmatrix},
R \begin{pmatrix}
\Delta_{v^*} \\
\Delta_{i^*}
\end{pmatrix},
R \begin{pmatrix}
\Delta_{v^*} \\
\Delta_{j^*}
\end{pmatrix} \: \mathrm{or} \:
R \begin{pmatrix}
\Delta_{i^*} \\
\Delta_{j^*}
\end{pmatrix},
\end{align}
and we need not consider the cases $R \begin{pmatrix}
\Delta_{u^*} \\
\Delta_{u^*}
\end{pmatrix}$ or $R \begin{pmatrix}
\Delta_{v^*} \\
\Delta_{v^*}
\end{pmatrix}$ since $y_1 \neq y_2$ and $y_4 \neq y_5$ in \eqref{typeMM}. \par
If $\left\langle h_{\sigma(5)}, h_{\sigma(6)} \right\rangle_{\mathcal{H}_1^d}$ is equal to either of the first two quantities on the right of \eqref{56choices2}, then either $h_{\sigma(1)} $ or $h_{\sigma(2)} $ is equal to $\mathds{1}_{\Delta_{v^*}}$. If $\left\langle h_{\sigma(5)}, h_{\sigma(6)} \right\rangle_{\mathcal{H}_1^d}$ is equal to the third or fourth quantity, then either $h_{\sigma(1)} $ or $h_{\sigma(2)} $ is equal to $\mathds{1}_{\Delta_{u^*}}$. \par
If $\left\langle h_{\sigma(5)}, h_{\sigma(6)} \right\rangle_{\mathcal{H}_1^d} = R \begin{pmatrix}
\Delta_{i^*} \\
\Delta_{j^*}
\end{pmatrix}$, then we must have $h_{\sigma(1)} = \mathds{1}_{\Delta_{u^*}}$ and $h_{\sigma(2)} = \mathds{1}_{\Delta_{v^*}}$, or vice versa. \par
\item
$N_u^{(a, b, c)} N_v^{(a, b, c)}$ terms: \par
If $\left\langle h_{\sigma(5)}, h_{\sigma(6)} \right\rangle_{\mathcal{H}_1^d}
= R \begin{pmatrix}
\Delta_{u^*} \\
\Delta_{u^*}
\end{pmatrix}$ (resp. $R \begin{pmatrix}
\Delta_{v^*} \\
\Delta_{v^*}
\end{pmatrix}$), then both $h_{\sigma(1)} $ and $h_{\sigma(2)} $ must be equal to $\mathds{1}_{\Delta_{v^*}}$ (resp. $\mathds{1}_{\Delta_{u^*}}$).
\item
$M_u^{(a, b, c)} N_v^{(a, b, c)}$ terms: \par
There are only three possibilities,
\begin{align*}
\left\langle h_{\sigma(5)}, h_{\sigma(6)} \right\rangle_{\mathcal{H}_1^d}
= R \begin{pmatrix}
\Delta_{u^*} \\
\Delta_{i^*}
\end{pmatrix},
R \begin{pmatrix}
\Delta_{v^*} \\
\Delta_{v^*}
\end{pmatrix} \: \mathrm{or} \:
R \begin{pmatrix}
\Delta_{v^*} \\
\Delta_{i^*}
\end{pmatrix},
\end{align*}
and we need not consider the case $R \begin{pmatrix}
\Delta_{u^*} \\
\Delta_{u^*}
\end{pmatrix}$ since $y_1 \neq y_2$ in \eqref{typeMN}. If $\left\langle h_{\sigma(5)}, h_{\sigma(6)} \right\rangle_{\mathcal{H}_1^d}$ is equal to $R \begin{pmatrix}
\Delta_{u^*} \\
\Delta_{i^*}
\end{pmatrix}$, then both $h_{\sigma(1)}$ and $h_{\sigma(2)}$ are equal to $\mathds{1}_{\Delta_{v^*}}$. If $\left\langle h_{\sigma(5)}, h_{\sigma(6)} \right\rangle_{\mathcal{H}_1^d}$ is equal to $R \begin{pmatrix}
\Delta_{v^*} \\
\Delta_{v^*}
\end{pmatrix}$ or $R \begin{pmatrix}
\Delta_{v^*} \\
\Delta_{i^*}
\end{pmatrix}$, then either $h_{\sigma(1)}$ or $h_{\sigma(2)}$ is equal to $\mathds{1}_{\Delta_{u^*}}$.
\end{enumerate}
Thus we obtain
\begin{align*}
\sum_{i, j= 0}^{2^n-1} \sum_{k,l=1}^m \sum_{u=0}^{2^k-1} \sum_{v=0}^{2^l-1} &\exptn{\mathcal{D}^2_{h_{\sigma(1)}, h_{\sigma(2)}} \psi_{t^n_i} \psi_{t^n_j}} \left\langle h_{\sigma(3)}, h_{\sigma(4)} \right\rangle_{\mathcal{H}_1^d} \left\langle h_{\sigma(5)}, h_{\sigma(6)} \right\rangle_{\mathcal{H}_1^d} \\
&\leq \sum_{k,l = 1}^m \left( \sum_{i, j = 0}^{2^n - 1} \sum_{u = 0}^{2^k - 1} \sum_{v = 0}^{2^l-1} \abs{ R \begin{pmatrix}
\Delta_{u^*} \\
\Delta_{v^*}
\end{pmatrix}}^{\rho} \right)^{\frac{1}{\rho}}
\left( \sum_{i, j = 0}^{2^n - 1} \sum_{u = 0}^{2^k - 1} \sum_{v = 0}^{2^l-1} 2^{-(n+k)\left(\frac{\rho'}{\rho}\right)} 2^{-(n+l)\left(\frac{\rho'}{\rho}\right)} \right)^{\frac{1}{\rho'}} \\
&\leq \norm{R}_{\rho-var; [0,T]^2} \sum_{k,l=1}^m 2^{-2n \left( \frac{1}{\rho} - \frac{1}{\rho'} \right)} 2^{-k \left( \frac{1}{\rho} - \frac{1}{\rho'} \right)} 2^{-l \left( \frac{1}{\rho} - \frac{1}{\rho'} \right)} \\
&\leq C \norm{R}_{\rho-var; [0,T]^2} 2^{-2n \left( \frac{2}{\rho} - 1 \right)} \rightarrow 0.
\end{align*}
\item
$\left\langle h_{\sigma(3)}, h_{\sigma(4)} \right\rangle_{\mathcal{H}_1^d} \: \mathrm{and} \: \left\langle h_{\sigma(5)}, h_{\sigma(6)} \right\rangle_{\mathcal{H}_1^d}
\neq R \begin{pmatrix}
\Delta_{u^*} \\
\Delta_{v^*}
\end{pmatrix}$: \par
We will show that
\begin{align} \label{A22Bound}
\norm{\Phi \left( h_{\sigma(1)} \right)}_{q-var; [0, T]} \norm{\Phi \left( h_{\sigma(2)} \right)}_{q-var; [0, T]} \abs{ \left\langle h_{\sigma(3)}, h_{\sigma(4)} \right\rangle_{\mathcal{H}_1^d}} \abs{ \left\langle h_{\sigma(5)}, h_{\sigma(6)} \right\rangle_{\mathcal{H}_1^d}}
\end{align}
is bounded above by $C \, 2^{\frac{-2(n+k)}{\rho}} 2^{\frac{-2(n+l)}{\rho}}$, which gives us
\begin{align*}
\sum_{i, j= 0}^{2^n-1} \sum_{k,l=1}^m \sum_{u=0}^{2^k-1} \sum_{v=0}^{2^l-1} &\exptn{\mathcal{D}^2_{h_{\sigma(1)}, h_{\sigma(2)}} \psi_{t^n_i} \psi_{t^n_j}} \left\langle h_{\sigma(3)}, h_{\sigma(4)} \right\rangle_{\mathcal{H}_1^d} \left\langle h_{\sigma(5)}, h_{\sigma(6)} \right\rangle_{\mathcal{H}_1^d} \\
&\leq C \sum_{i, j= 0}^{2^n-1} \sum_{k,l=1}^m \sum_{u=0}^{2^k-1} \sum_{v=0}^{2^l-1} 2^{-(n+k) \frac{2}{\rho}} 2^{-(n+l)  \frac{2}{\rho}} \\
&\leq C \, 2^{-2n \left( \frac{2}{\rho} -1 \right)} \sum_{k,l=1}^m 2^{-k \left(\frac{2}{\rho} - 1 \right)} 2^{-l \left(\frac{2}{\rho} - 1 \right)} \rightarrow 0.
\end{align*}
\begin{enumerate}[(a)]
\item
$M_u^{(a, b, c)} M_v^{(a, b, c)}$ terms: \par
Note that in this scenario, neither $h_{\sigma(1)}$ nor $h_{\sigma(2)}$ can be equal to $\mathds{1}_{\Delta_{i^*}}$ or $\mathds{1}_{\Delta_{j^*}}$, so we essentially have two cases. \par
If $h_{\sigma(1)} \: \mathrm{and} \: h_{\sigma(2)} = \mathds{1}_{\Delta_{u^*}}$, we must have
\begin{align*}
\left\langle h_{\sigma(3)}, h_{\sigma(4)} \right\rangle_{\mathcal{H}_1^d}
= R \begin{pmatrix}
\Delta_{v^*} \\
\Delta_{i^*}
\end{pmatrix} \quad \mathrm{and} \quad \left\langle h_{\sigma(5)}, h_{\sigma(6)} \right\rangle_{\mathcal{H}_1^d}
= R \begin{pmatrix}
\Delta_{v^*} \\
\Delta_{j^*}
\end{pmatrix},
\end{align*}
or vice versa. (The case $h_{\sigma(1)} \: \mathrm{and} \: h_{\sigma(2)} = \mathds{1}_{\Delta_{v^*}}$ can be resolved similarly by swapping $u$ and $v$.) \par
If instead we have $h_{\sigma(1)} = \mathds{1}_{\Delta_{u^*}}$ and $h_{\sigma(2)} = \mathds{1}_{\Delta_{v^*}}$, or vice versa, then without loss of generality, it must be the case that
\begin{align*}
\left\langle h_{\sigma(3)}, h_{\sigma(4)} \right\rangle_{\mathcal{H}_1^d}
= R \begin{pmatrix}
\Delta_{u^*} \\
\Delta_{i^*}
\end{pmatrix} \: \mathrm{or} \:
R \begin{pmatrix}
\Delta_{u^*} \\
\Delta_{j^*}
\end{pmatrix},
\quad \left\langle h_{\sigma(5)}, h_{\sigma(6)} \right\rangle_{\mathcal{H}_1^d}
= R \begin{pmatrix}
\Delta_{v^*} \\
\Delta_{j^*}
\end{pmatrix} \: \mathrm{or} \:
R \begin{pmatrix}
\Delta_{v^*} \\
\Delta_{i^*}
\end{pmatrix}.
\end{align*}
\item
$N_u^{(a, b, c)} N_v^{(a, b, c)}$ terms: \par
Without loss of generality, we have
\begin{align*}
h_{\sigma(1)} = \mathds{1}_{\Delta_{u^*}}, \:\:
h_{\sigma(2)} = \mathds{1}_{\Delta_{v^*}}, \:\:
\left\langle h_{\sigma(3)}, h_{\sigma(4)} \right\rangle_{\mathcal{H}_1^d}
= R \begin{pmatrix}
\Delta_{u^*} \\
\Delta_{u^*}
\end{pmatrix}, \:\:
\left\langle h_{\sigma(5)}, h_{\sigma(6)} \right\rangle_{\mathcal{H}_1^d}
= R \begin{pmatrix}
\Delta_{v^*} \\
\Delta_{v^*}
\end{pmatrix}.
\end{align*}
\item
$M_u^{(a, b, c)} N_v^{(a, b, c)}$ terms: \par
Without loss of generality, either
\begin{align*}
h_{\sigma(1)}, \, h_{\sigma(2)} = \mathds{1}_{\Delta_{u^*}}, \quad
\left\langle h_{\sigma(3)}, h_{\sigma(4)} \right\rangle_{\mathcal{H}_1^d}
= R \begin{pmatrix}
\Delta_{v^*} \\
\Delta_{i^*}
\end{pmatrix} \:\: \mathrm{and} \:\: \left\langle h_{\sigma(5)}, h_{\sigma(6)} \right\rangle_{\mathcal{H}_1^d}
= R \begin{pmatrix}
\Delta_{v^*} \\
\Delta_{v^*}
\end{pmatrix},
\end{align*}
or
\begin{align*}
&h_{\sigma(1)} = \mathds{1}_{\Delta_{u^*}}, \: h_{\sigma(2)} = \mathds{1}_{\Delta_{v^*}}, \\
&\left\langle h_{\sigma(3)}, h_{\sigma(4)} \right\rangle_{\mathcal{H}_1^d}
= R \begin{pmatrix}
\Delta_{u^*} \\
\Delta_{i^*}
\end{pmatrix} \:\: \mathrm{and} \:\: \left\langle h_{\sigma(5)}, h_{\sigma(6)} \right\rangle_{\mathcal{H}_1^d}
= R \begin{pmatrix}
\Delta_{v^*} \\
\Delta_{v^*}
\end{pmatrix}.
\end{align*}
\end{enumerate}
\end{enumerate}
For the $A_{\sigma, 3}$ terms, when we consider the three inner products $\left\langle h_{\sigma(1)}, h_{\sigma(2)} \right\rangle_{\mathcal{H}_1^d}$, $\left\langle h_{\sigma(3)}, h_{\sigma(4)} \right\rangle_{\mathcal{H}_1^d}$ and $\left\langle h_{\sigma(5)}, h_{\sigma(6)} \right\rangle_{\mathcal{H}_1^d}$, we have two cases: either one of them is equal to $R \begin{pmatrix}
\Delta_{u^*} \\
\Delta_{v^*}
\end{pmatrix}$, or two or more of them are. Observe that it is not possible for none of them to equal $R \begin{pmatrix}
\Delta_{u^*} \\
\Delta_{v^*}
\end{pmatrix}$.
\begin{enumerate}[(i)]
\item
If two or more of the inner products are equal to $R \begin{pmatrix}
\Delta_{u^*} \\
\Delta_{v^*}
\end{pmatrix}$, then we can use the same computation in the first case for the $A_{\sigma, 2}$ terms to show that
\begin{align*}
\abs{ \left\langle h_{\sigma(1)}, h_{\sigma(2)} \right\rangle_{\mathcal{H}_1^d}} \abs{ \left\langle h_{\sigma(3)}, h_{\sigma(4)} \right\rangle_{\mathcal{H}_1^d} } \abs{ \left\langle h_{\sigma(5)}, h_{\sigma(6)} \right\rangle_{\mathcal{H}_1^d}}
\leq C \, R^{\rho}_{\Delta_u^i \times \Delta_v^j} \, 2^{\frac{-(n+k)}{\rho} \left( 1 - \frac{\rho}{2} \right) } 2^{\frac{-(n+l)}{\rho} \left( 1 - \frac{\rho}{2} \right) },
\end{align*}
and this gives us
\begin{align*}
\sum_{i, j = 0}^{2^n - 1} \sum_{k,l=1}^m \sum_{u = 0}^{2^k - 1} \sum_{v=0}^{2^l-1} &\exptn{\psi_{t^n_i} \psi_{t^n_j}} \left\langle h_{\sigma(1)}, h_{\sigma(2)} \right\rangle_{\mathcal{H}_1^d} \left\langle h_{\sigma(3)}, h_{\sigma(4)} \right\rangle_{\mathcal{H}_1^d} \left\langle h_{\sigma(5)}, h_{\sigma(6)} \right\rangle_{\mathcal{H}_1^d} \\
&\leq C \, 2^{-2n \left( \frac{1}{\rho} - \frac{1}{2} \right)} \sum_{k,l=1}^{\infty} 2^{-k \left( \frac{1}{\rho} - \frac{1}{2} \right)} 2^{-l \left( \frac{1}{\rho} - \frac{1}{2}\right)} \norm{R}^{\rho}_{\rho-var;[0,T]^2} \rightarrow 0.
\end{align*}
\item
Assume that $\left\langle h_{\sigma(1)}, h_{\sigma(2)} \right\rangle_{\mathcal{H}_1^d} = R \begin{pmatrix}
\Delta_{u^*} \\
\Delta_{v^*}
\end{pmatrix}, \: \mathrm{and} \:
\left\langle h_{\sigma(3)}, h_{\sigma(4)} \right\rangle_{\mathcal{H}_1^d}, \, \left\langle h_{\sigma(5)}, h_{\sigma(6)} \right\rangle_{\mathcal{H}_1^d}
\neq R \begin{pmatrix}
\Delta_{u^*} \\
\Delta_{v^*}
\end{pmatrix}$. \par
Then without loss of generality, we have:
\begin{enumerate} [(a)]
\item
$M_u^{(a, b, c)} M_v^{(a, b, c)}$ terms:
\begin{align*}
\left\langle h_{\sigma(3)}, h_{\sigma(4)} \right\rangle_{\mathcal{H}_1^d}
= R \begin{pmatrix}
\Delta_{u^*} \\
\Delta_{i^*}
\end{pmatrix} \: \mathrm{or} \:
R \begin{pmatrix}
\Delta_{u^*} \\
\Delta_{j^*}
\end{pmatrix}, \quad
\left\langle h_{\sigma(5)}, h_{\sigma(6)} \right\rangle_{\mathcal{H}^1_d}
= R \begin{pmatrix}
\Delta_{v^*} \\
\Delta_{i^*}
\end{pmatrix} \: \mathrm{or} \:
R \begin{pmatrix}
\Delta_{v^*} \\
\Delta_{j^*}
\end{pmatrix}.
\end{align*}
\item
$N_u^{(a, b, c)} N_v^{(a, b, c)}$ terms:
\begin{align*}
\left\langle h_{\sigma(3)}, h_{\sigma(4)} \right\rangle_{\mathcal{H}_1^d}
= R \begin{pmatrix}
\Delta_{u^*} \\
\Delta_{u^*}
\end{pmatrix}, \quad
\left\langle h_{\sigma(5)}, h_{\sigma(6)} \right\rangle_{\mathcal{H}_1^d}
= R \begin{pmatrix}
\Delta_{v^*} \\
\Delta_{v^*}
\end{pmatrix}.
\end{align*}
\item
$M_u^{(a, b, c)} N_v^{(a, b, c)}$ terms:
\begin{align*}
\left\langle h_{\sigma(3)}, h_{\sigma(4)} \right\rangle_{\mathcal{H}_1^d}
= R \begin{pmatrix}
\Delta_{u^*} \\
\Delta_{i^*}
\end{pmatrix}, \quad
\left\langle h_{\sigma(5)}, h_{\sigma(6)} \right\rangle_{\mathcal{H}_1^d}
= R \begin{pmatrix}
\Delta_{v^*} \\
\Delta_{v^*}
\end{pmatrix}.
\end{align*}
\end{enumerate}
In each case, applying the bounds in \eqref{qrBounds} gives us
\begin{align*}
\abs{ \left\langle h_{\sigma(3)}, h_{\sigma(4)} \right\rangle_{\mathcal{H}_1^d}} \abs{\left\langle h_{\sigma(5)}, h_{\sigma(6)} \right\rangle_{\mathcal{H}_1^d}}
\leq C \, 2^{\frac{-(n+k)}{\rho}} 2^{\frac{-(n+l)}{\rho}},
\end{align*}
which in turn yields
\begin{align*}
\sum_{i, j= 0}^{2^n-1} \sum_{k,l=1}^m \sum_{u=0}^{2^k-1} \sum_{v=0}^{2^l-1} &\exptn{\psi_{t^n_i} \psi_{t^n_j}} \left\langle h_{\sigma(1)}, h_{\sigma(2)}\right\rangle_{\mathcal{H}_1^d} \left\langle h_{\sigma(3)}, h_{\sigma(4)} \right\rangle_{\mathcal{H}_1^d} \left\langle h_{\sigma(5)}, h_{\sigma(6)} \right\rangle_{\mathcal{H}_1^d} \\
&\leq \sum_{k,l = 1}^m \left( \sum_{i, j = 0}^{2^n - 1} \sum_{u = 0}^{2^k - 1} \sum_{v = 0}^{2^l-1} \abs{ R \begin{pmatrix}
\Delta_{u^*} \\
\Delta_{v^*}
\end{pmatrix}}^{\rho} \right)^{\frac{1}{\rho}}
\left( \sum_{i, j = 0}^{2^n - 1} \sum_{u = 0}^{2^k - 1} \sum_{v = 0}^{2^l-1} 2^{-(n+k)\left(\frac{\rho'}{\rho}\right)} 2^{-(n+l)\left(\frac{\rho'}{\rho}\right)} \right)^{\frac{1}{\rho'}} \\
&\leq \norm{R}_{\rho-var; [0,T]^2} \sum_{k,l=1}^m 2^{-2n \left( \frac{1}{\rho} - \frac{1}{\rho'} \right)} 2^{-k \left( \frac{1}{\rho} - \frac{1}{\rho'} \right)} 2^{-l \left( \frac{1}{\rho} - \frac{1}{\rho'} \right)} \\
&\leq C \norm{R}_{\rho-var; [0,T]^2} 2^{-2n \left( \frac{2}{\rho} - 1 \right)} \rightarrow 0.
\end{align*}
\end{enumerate}
\end{proof}

\begin{corollary} \label{skorodhodLimitEnhanced} 
For $2 \leq p < 4$, let $Y \in \mathcal{C} ^{p-var} \left( [0, T]; \mathcal{L} (\mathbb{R}^d; \mathbb{R}^m) \right)$ denote the path-level solution to 
\begin{align*}
\mathrm{d} Y_t = V(Y_t) \strato{\mathbf{X}_t}, \quad  Y_0 = y_0,
\end{align*}
where $\mathbf{X} \in \mathcal{C}^{0, p-var} \left( [0, T]; G^{\lfloor p \rfloor} \left(\mathbb{R}^d \right)\right)$ satisfies Condition \ref{newCond1} and its covariance function satisfies
\begin{align*}
\left\| R(t, \cdot) - R (s, \cdot) \right\|_{q-var; [0, T]} 
\leq C \left| t - s\right|^{\frac{1}{\rho}}, \quad \forall s, t \in [0, T].
\end{align*}
Then if $V \in \mathcal{C}^{\lfloor p \rfloor + 4}_b \left( \mathbb{R}^{md}; \mathbb{R}^{md} \otimes \mathbb{R}^d \right)$, we have
\begin{align}  \label{enhancedSkoEq}
\lim_{\left\| \pi(n)\right\| \rightarrow 0} \norm{\sum_i V(Y_{t^n_i}) \left( \mathbf{X}^2_{t^n_i, t^n_{i+1}} - \frac{1}{2} \sigma^2 \left( t^n_i, t^n_{i+1} \right) \mathcal{I}_d \right)}_{L^2(\Omega)} = 0.
\end{align}
Furthermore, if $3 \leq p < 4$ and $V \in \mathcal{C}^9_b \left( \mathbb{R}^{md}; \mathbb{R}^{md} \otimes \mathbb{R}^d \right)$, we have
\begin{align}  \label{enhancedSkoEq2}
\lim_{\left\| \pi(n)\right\| \rightarrow 0} \norm{\sum_i \nabla V(Y_{t^n_i}) \left( V(Y_{t^n_i}) \right) \left( \mathbf{X}^3_{t^n_i, t^n_{i+1}} \right)}_{L^2(\Omega)} = 0.
\end{align}
\end{corollary}

\begin{proof}
We have to show that bounds \eqref{propCond0} and \eqref{propCond} in Proposition \ref{2ndlevel} are satisfied with 
\begin{align*}
\psi_t = \left[ V(Y_t) \right]_j \in \mathbb{R}^d \otimes \mathbb{R}^d, \quad j = 1, \ldots, m,
\end{align*}
to show \eqref{enhancedSkoEq}. Similarly, proving that bound \eqref{propCond2} in Proposition \ref{3rdlevel} is satisfied with
\begin{align*}
\psi_t = \left[ \nabla V(Y_t) \left( V(Y_t) \right) \right]_j \in \mathbb{R}^d \otimes \mathbb{R}^d \otimes \mathbb{R}^d, \quad j = 1, \ldots, m,
\end{align*}
will yield \eqref{enhancedSkoEq2}.

\eqref{propCond0} is trivially true since $V \in \mathcal{C}^1_b$. To show that the bounds hold for the higher Malliavin derivatives, recall Proposition \ref{dir der est}, which states that almost surely we have
\begin{align}  
\norm{\mathcal{D}_{h_1, \ldots, h_n}^n Y_{\cdot}}_{\infty} 
\leq P_{d(n)} \left( \norm{\mathbf{X}}_{p-var; [0, T]}, \exp \left( C \, N^{\mathbf{X}}_{1; [0, T]} \right)\right) \prod\limits_{i=1}^{n} \left\| \Phi(h_{i})\right\|_{q-var; [0, T]}.
\end{align}
As both $\norm{\mathbf{X}}_{p-var; [0, T]}$ and $\exp \left( C \, N^{\mathbf{X}}_{1; [0, T]} \right)$ belong to $\bigcap_{r>0}L^{r}\left( \Omega \right)$, we have 
\begin{align} \label{bound}
\norm{\mathcal{D}_{h_1, \ldots, h_n}^n Y_t}_{L^{r}\left( \Omega \right) }
\leq C_{n, q} \prod \limits_{i=1}^{n} \norm{\Phi(h_i)}_{q-var; \left[ 0,T\right] }
\end{align}
for any $r > 0$.
Now we simply use the product and chain rule of Malliavin differentiation in conjunction with the fact that $V$ has bounded derivatives up to the appropriate order.
\end{proof}

\section{Correction formula} \label{corr formula}

We are now ready to prove the main result of the paper. As before, $\pi(n) := \left\{  t^{n}_{i} \right\} , t^{n}_{i} := \frac{iT}{2^{n}}$, denotes the sequence of dyadic partitions on $[0, T]$.

\subsection{Main theorem}

\begin{theorem} \label{mainThm3rdLevel} 
For $3 \leq p < 4$, let $Y \in \mathcal{C}^{p-var} \left( [0, T]; \mathcal{L} (\mathbb{R}^d; \mathbb{R}^m) \right)$ denote the path-level solution to 
\begin{align*}
\mathrm{d} Y_t = V(Y_t) \strato{\mathbf{X}_t}, \quad  Y_0 = y_0,
\end{align*}
where $V \in \mathcal{C}^9_b \left( \mathbb{R}^{md}; \mathbb{R}^{md} \otimes \mathbb{R}^d \right)$, and $\mathbf{X} \in \mathcal{C}^{0, p-var} \left( [0, T]; G^{\lfloor p \rfloor} \left(\mathbb{R}^d \right)\right)$ is a Volterra process which satisfies Condition \ref{newCond1}, and whose kernel satisfies Condition \ref{amnCond} with $\alpha < \frac{1}{p}$. Furthermore, we assume the covariance function satisfies 
\begin{align} 
\left\| R(t, \cdot) - R (s, \cdot) \right\|_{q-var; [0, T]} \leq C \left| t - s\right|^{\frac{1}{\rho}},
\end{align}
for all $s, t \in [0, T]$, and $\norm{R(\cdot)}_{q-var; [0, T]} < \infty$. 
 Then almost surely, we have 
\begin{align} \label{cflr}
\int_0^T Y_t \circ \mathrm{d} \mathbf{X}_t 
&= \int_0^T Y_t \, \mathrm{d} X_t + \sum_{j=1}^m \left( \frac{1}{2} \int_0^T \mathrm{tr} \left[ V(Y_s) \right]_j \, \mathrm{d} R(s) + U^{(j)}_T \right) e_j,
\end{align}
where for $j = 1, \ldots, m$, $U^{(j)}_T$ is the limit in $L^2(\Omega)$ of 
\begin{align} \label{2ndTermApprox}
\sum_{i} \int_0^{t^n_i} \mathrm{tr} \left[J^{\mathbf{X}}_{t^n_i \leftarrow s} V(Y_s) - V \left( Y_{t^n_i} \right) \right]_j R(\Delta^n_i, \mathrm{d}s) 
\end{align}
along the dyadic partitions $\{ t^n_i\}$ of $[0, T]$.
\end{theorem}

\begin{proof}
Using bounds \eqref{controlledBound2}, \eqref{boundedVF2} together with the integrability of $\mathbf{X}$, we can apply dominated convergence theorem to \eqref{controlledRPdefn2} in Theorem \ref{controlledThm2} to show that $\int_0^T Y_t \strato{\textbf{X}_t} $ is the $L^2(\Omega)$ limit of
\begin{align*}
\lim_{n \rightarrow \infty} \sum_i
\, Y_{t^n_i} \left( X_{t^n_i, t^n_{i+1}} \right) + V(Y_{t^n_i}) \left( \mathbf{X}^2_{t^n_i, t^n_{i+1}} \right) + \nabla V \left( Y_{t^n_i} \right) \left( V \left( Y_{t^n_i} \right) \right) \left( \mathbf{X}^3_{t^n_i, t^n_{i+1}} \right).
\end{align*}
Now applying Proposition \ref{skorohodLimit3rdLevel} in conjunction with Corollary \ref{skorodhodLimitEnhanced} gives us
\begin{align*}
\int_0^T Y_t \wrt{X_t} 
&= \lim_{n \rightarrow \infty} \sum_i \left[ Y_{t^n_i} \left( X_{t^n_i, t^n_{i+1}} \right) - \sum_{j=1}^m \left( \int_0^{t^n_i} \mathrm{tr} \left[J^{\mathbf{X}}_{t^n_i \leftarrow s} V(Y_s)\right]_j R\left( \Delta^n_i, \mathrm{d} s \right) \right) e_j + A_i \right], 
\end{align*}
where the limit is also in $L^2(\Omega)$ and
\begin{align*}
A_i := V(Y_{t^n_i}) \left( \left( \mathbf{X}^2_{t^n_i, t^n_{i+1}} \right) - \frac{1}{2} \sigma^2 \left( t^n_i, t^n_{i+1} \right) \mathcal{I}_d \right) + \nabla V \left( Y_{t^n_i} \right) \left( V \left( Y_{t^n_i} \right) \right) \left( \mathbf{X}^3_{t^n_i, t^n_{i+1}} \right).
\end{align*}
Following the procedure in Theorem 6.1 of \cite{cl2018}, subtracting the two integrals and re-balancing the terms gives us
\begin{align} \label{diff}
\begin{split}
\int_0^T Y_t \strato{\textbf{X}_t} \, - &\int_0^T Y_t \wrt{X_t} \\
&= \sum_{j=1}^m \left( \lim_{n \rightarrow \infty} \sum_i \int_0^{t^n_i} \mathrm{tr} \left[ J^{\mathbf{X}}_{t^n_i \leftarrow s} V(Y_s) \right]_j R(\Delta^n_i, \wrt{s} ) + \frac{1}{2} \sigma^2 \left( t^n_i, t^n_{i+1} \right) \mathrm{tr} \left[ V(Y_{t^n_i}) \right]_j \right) e_j, \\
&= \sum_{j=1}^m \left( \lim_{n \rightarrow \infty} \sum_i \int_0^{t^n_i} \mathrm{tr} \left[ J^{\mathbf{X}}_{t^n_i \leftarrow s} V(Y_s) - V\left(Y_{t^n_i} \right) \right]_j R(\Delta^n_i, \wrt{s} ) \right. \\
&\left. \fqquad+ \frac{1}{2} \lim_{n \rightarrow \infty} \sum_i \mathrm{tr} \left[ V(Y_{t^n_i}) \right]_j \big( R\left( t^n_{i+1}, t^n_{i+1} \right) - R \left(t^n_i, t^n_i \right) \big) \right) e_j.
\end{split}
\end{align}
The second term in the last line of the expression above is dominated by
\begin{align*}
C \norm{V(Y_{\cdot})}_{p-var; [0, T]} \norm{R(\cdot)}_{q-var; [0, T]}
\end{align*}
by Young's inequality, and thus converges in $L^2 (\Omega)$ to
\begin{align*}
\frac{1}{2} \int_0^T \mathrm{tr} \left[ V(Y_s) \right]_j \, \mathrm{d} R(s).
\end{align*}
This in turn guarantees the convergence of the first term in $L^2(\Omega)$ to the random variable $U_T^{(j)}$. Now extracting an almost sure subsequence allows us to equate both sides of \eqref{cflr} almost surely, and the proof is thus complete.
\end{proof}

In the more regular case $2 \leq p < 3$, we can be more precise in identifying the second term \eqref{2ndTermApprox}.

\begin{proposition} \label{volterra2ndTerm}
For $2 \leq p < 3$, let $Y \in \mathcal{C}^{p-var} \left( [0, T]; \mathcal{L} (\mathbb{R}^d; \mathbb{R}^m) \right)$ denote the path-level solution to $\mathbf{X}$ 
\begin{align*}
\mathrm{d} Y_t = V(Y_t) \strato{\mathbf{X}_t}, \quad  Y_0 = y_0,
\end{align*}
where $\mathbf{X} \in \mathcal{C}^{0, p-var} \left( [0, T]; G^{\lfloor p \rfloor} \left(\mathbb{R}^d \right)\right)$ is a Volterra process satisfying Condition \ref{newCond1} with some $\rho \in \left[ 1, \frac{3}{2} \right)$, and whose kernel satisfies Condition \ref{amnCond} with $\alpha < \frac{1}{2p}$. \par 
Furthermore, assume that $V \in \mathcal{C}^6_b \left( \mathbb{R}^{md}; \mathbb{R}^{md} \otimes \mathbb{R}^d \right)$, and the covariance function satisfies 
\begin{align} 
\left\| R(t, \cdot) - R (s, \cdot) \right\|_{\rho-var; [0, T]} 
\leq C \left| t - s\right|^{\frac{1}{\rho}},
\end{align}
for all $s, t \in [0, T]$. Then almost surely we have 
\begin{align*}
\int_0^T Y_t \circ \mathrm{d} \mathbf{X}_t 
&= \int_0^T Y_t \, \mathrm{d} X_t + Z_T,
\end{align*}
where the correction term is given by
\begin{align} \label{Zdef}
\begin{split}
Z^{(j)}_T 
&= \frac{1}{2} \int_0^T \mathrm{tr} \left[ V(Y_s) \right]_j \, \mathrm{d} R(s) + \int_{[0, T]^2} h_j(s, t) \, \mathrm{d} R(s,t) \\
&= \frac{1}{2} \int_0^T \mathrm{tr} \left[ V(Y_s) \right]_j \, \mathrm{d} R(s) + \int_0^T \mathcal{K}^* \otimes \mathcal{K}^* h_j (r, r) \wrt{r}, \quad j = 1, \ldots, m.
\end{split}
\end{align}
with
\begin{align*}
h_j (s, t) := \mathds{1}_{[0, t)} (s) \, \mathrm{tr} \left[J^{\mathbf{X}}_{t \leftarrow s} V(Y_s) - V(Y_t) \right]_j.
\end{align*}
\end{proposition}

\begin{proof}
Under the conditions of the theorem, we can invoke Theorem 6.1 from \cite{cl2018} to obtain the first line of \eqref{Zdef}. To obtain the second line, we will use Proposition 4.3 from \cite{lim2018}, which states that if $\phi:[0,T]^{2} \rightarrow \mathbb{R}$ is a $\lambda$-H\"{o}lder bi-continuous function (one that satisfies Definition \ref{biHolderDef} without necessarily satisfying \eqref{biHolder2}) with $\lambda > 2 \alpha$, then
\begin{align} \label{volEqDR}
\int_{[0, T]^{2}}\phi(s, t)\,\mathrm{d} R(s, t) = \int_{0}^{T}\mathcal{K}^* \otimes\mathcal{K}^* \phi(r,r)\,\mathrm{d}r.
\end{align}
Thus, the proof is complete once we show that $h_j (s, t)$ is $\frac{1}{p}$-H\"{o}lder bi-continuous for all $j = 1, \ldots, m$ since $\frac{1}{p} > 2\alpha$. Using the fact that 
\begin{align*}
\tilde{h}_j (s, t) := \mathrm{tr} \left[J^{\mathbf{X}}_{t \leftarrow s} V(Y_s) - V(Y_t) \right]_j
\end{align*}
is $\frac{1}{p}$-H\"{o}lder bi-continuous, we have, assuming $v_2 > v_1$ without loss of generality,
\begin{align*}
\abs{h_j (u, v_2) - h_j (u, v_1)} 
&= \abs{\tilde{h}_j (u, v_2 \vee u) - \tilde{h}_j (u, v_1 \vee u)}, \quad u, v_1, v_2 \in [0, T], \\
&\leq C_1 \abs{v_2 \vee u - v_1 \vee u}^{\frac{1}{p}} \\
&\leq C_2 \abs{v_2 - v_1}^{\frac{1}{p}},
\end{align*}
and similarly, 
\begin{align*}
\abs{h_j (u_2, v) - h_j (u_1, v)} \leq C \abs{u_2 - u_1}^{\frac{1}{p}}, \quad v, u_1, u_2 \in [0, T].
\end{align*}

\end{proof}

In the case $3 \leq p < 4$, due to the lack of complementary regularity, we cannot apply Theorem \ref{2Dintegral}  although $\mathds{1}_{[0, t)} (s) \mathrm{tr} \left[J^{\mathbf{X}}_{t \leftarrow s} V(Y_s) - V \left( Y_t \right) \right]_j$ is continuous almost surely on $[0, T]^2$. Furthermore, although the integrand is strongly $\frac{1}{p}$-H\"{o}lder bi-continuous away from the diagonal, one can check that in general, \eqref{biHolder2} fails at the diagonal, which means that we cannot employ \eqref{volEqDR} from Proposition 4.3 in \cite{lim2018} (it can also be verified that there would be insufficient H\"{o}lder regularity in the weaker sense). Hence, we can only show convergence in $L^2(\Omega)$ rather than almost surely. The question of whether the second part of the correction term can be identified as a proper 2D Young-Stieltjes integral requires further investigation. \par 

An interesting special case of Theorem is when the vector fields defining the RDE commute. In this situation the $U_{T}$ terms in the correction formula \eqref{cflr} disappear.

\begin{corollary} \label{thmSimple} 
Under the conditions of Theorem \ref{mainThm3rdLevel}, if in addition the vector fields commute, i.e. $[V_{i},V_{j}]=0$ for all $i,j=1,\ldots,d$, then
\begin{align*}
\int_{0}^{T}Y_{t}\circ\mathrm{d}\mathbf{X}_{t}=\int_{0}^{T}Y_{t} \, \mathrm{d}X_{t}+\frac{1}{2}\sum_{j=1}^{m}\left(  \int_{0}^{T}\mathrm{tr} \left[ V(Y_{s})\right]_{j}\,\mathrm{d}R(s)\right) e_{j},
\end{align*}
\end{corollary}

\begin{proof}
For any vector field $W \in \mathcal{C}^1 \left(\mathbb{R}^{md}; \mathbb{R}^{md} \right)$, we have
\begin{align*}
\left( J^{\mathbf{X}}_t\right)^{-1} W(Y_t) = W(y_0) + \sum_{i=1}^d \int_0^t \left( J^{\mathbf{X}}_s \right)^{-1} [V_i, W] \strato{\mathbf{X}^{(i)}_s},
\end{align*}
which can be computed using the RDEs satisfied by $Y$ and $\left( J^{\mathbf{X}} \right)^{-1}$, cf. Chapter 20 (Section 4.2) in \cite{fv2010b}. Hence, if the $V_i$'s commute, then each $V_i$ is invariant under the flow of $Y$, and we have
\begin{align*}
J^{\mathbf{X}}_{t \leftarrow s} V(Y_s) = V(Y_t), \quad 0 \leq s < t \leq T.
\end{align*}
\end{proof}

\subsection{Applications of the correction formula}

We present applications of the main theorem to two important special cases. The first is to fractional Brownian motion in the regime $H > \frac{1}{4}$. Given the interest in this in recent years, especially in volatility models in mathematical finance, this result may also find practical uses among the wider areas of applied probability. The second application is to use the commuting case discussed in Corollary \ref{thmSimple} to obtain It\^{o} formulas for Gaussian processes. This links our correction formula to the prolific stream of recent work mentioned in the introduction.

\begin{theorem} [Correction formula fBM, $H > \frac{1}{4}$]
For $1 \leq p < 4$, let $Y \in \mathcal{C}^{p-var} \left( [0, T]; \mathcal{L} (\mathbb{R}^d; \mathbb{R}^m) \right)$ denote the path-level solution to 
\begin{align*}
\mathrm{d} Y_t = V(Y_t) \strato{\mathbf{X}_t}, \quad  Y_0 = y_0,
\end{align*}
where we assume that $V \in \mathcal{C}^k_b \left( \mathbb{R}^{md}; \mathbb{R}^{md} \otimes \mathbb{R}^d \right)$, with
\begin{align} \label{kDefn}
k = \begin{cases}
\, 2,  \quad 1 \leq p < 2, \\
\, 6, \quad 2 \leq p < 3, \\
\, 9, \quad 3 \leq p < 4,
\end{cases}
\end{align}
and $\mathbf{X} \in \mathcal{C}^{0, p-var} \left( [0, T]; G^{\lfloor p \rfloor} \left(\mathbb{R}^d \right)\right)$ is the geometric rough path constructed from the limit of the piecewise-linear approximations of standard fractional Brownian motion with Hurst parameter $H > \frac{1}{4}$. Then almost surely, we have 
\begin{align*}
\int_0^T Y_t \circ \mathrm{d} \mathbf{X}_t 
&= \int_0^T Y_t \, \mathrm{d} X_t + Z_T,
\end{align*}
where the correction term $Z_T = \left( Z^{(1)}_T, \ldots, Z^{(m)}_T \right)$ is given by
\begin{align} \label{cTerm}
\begin{split}
Z^{(j)}_T 
&= H \int_0^T \mathrm{tr} \left[ V(Y_s) \right]_j \, s^{2H - 1} \wrt{s} + \int_{[0, T]^2} h_j(s,t) \, \mathrm{d} R(s, t), \quad j = 1, \ldots, m, \\
&= H \int_0^T \mathrm{tr} \left[ V(Y_s) \right]_j \, s^{2H - 1} \wrt{s} + \int_0^T \mathcal{K}^*\otimes \mathcal{K}^* h_j (r, r) \wrt{r}, \quad \left( \mathrm{when} \, \frac{1}{3} < H \leq \frac{1}{2} \right),
\end{split}
\end{align}
with
\begin{align*}
h_j (s,t) := \mathds{1}_{[0, t)} (s) \mathrm{tr} \left[J^{\mathbf{X}}_{t \leftarrow s} V(Y_s) - V(Y_t) \right]_j, \quad j = 1, \ldots, m.
\end{align*}
\end{theorem}

\begin{remark}
For simplicity, we use the same notation for the second term of $Z_{T}^{(j)}$ for all $H > \frac{1}{4}$, with the understanding that it denotes the $L^{2}(\Omega)$ limit of \eqref{2ndTermApprox} when $\frac{1}{4} < H \leq \frac{1}{3}$.
\end{remark}

\begin{proof}
The proof rests entirely on Proposition \ref{fBMCond}, which tells us that fractional Brownian motion fulfills all the requirements needed to apply Theorem 6.1 of \cite{cl2018} when $H > \frac{1}{3}$, and Theorem \ref{mainThm3rdLevel} when $\frac{1}{4} < H \leq \frac{1}{3}$.
\end{proof}We now show that we can recover It\^{o}'s formulas.

\begin{theorem} [It\^{o} formulas for Gaussian processes]
For $1 \leq p < 4$, let $\mathbf{X} \in \mathcal{C}^{0, p-var} \left( [0, T]; G^{\lfloor p \rfloor} \left( \mathbb{R}^d \right)\right)$ satisfy Condition \ref{newCond1}.  Depending on $p$, we further impose the following conditions:
\begin{enumerate} [(i)]
\item $1 \leq p < 2$: $\sigma^2(s, t) \leq C \abs{t -s}^{\theta}$ for some $\theta > 1$ and $\norm{R(\cdot)}_{q-var; [0, T]} < \infty$.
\item $2 \leq p < 3$: The covariance function satisfies 
\begin{align} \label{cond2}
\left\| R(t, \cdot) - R (s, \cdot) \right\|_{q-var; [0, T]} 
\leq C \left| t - s\right|^{\frac{1}{\rho}},
\end{align}
for all $s, t \in [0, T]$.
\item $3 \leq p < 4$: $\mathbf{X}$ is a Volterra process whose kernel satisfies Condition \ref{amnCond} with $\alpha < \frac{1}{p}$. Furthermore, its covariance function satisfies \eqref{cond2} and $\norm{R(\cdot)}_{q-var; [0, T]} < \infty$.
\end{enumerate}
Then almost surely, for $f \in \mathcal{C}^{k+2}_b \left( \mathbb{R}^d; \mathbb{R} \right)$, $k$ defined as in \eqref{kDefn}, we have 
\begin{align*}
f(X_T) - f(0) = \int_0^T \left\langle \nabla f(X_t), \circ \mathrm{d} \mathbf{X}_t \right\rangle 
&= \int_0^T \left\langle \nabla f(X_t), \, \mathrm{d} X_t \right\rangle + \frac{1}{2} \int_0^T \Delta f (X_t) \, \mathrm{d} R(t).
\end{align*}
\end{theorem}

\begin{proof}
Let $Y_t = \left( Y^{(1)}_t, \ldots, Y^{(2d)}_t \right)$ denote the augmented process
\begin{align*}
\left( \pd{f}{e_1} (Y_t), \ldots, \pd{f}{e_d} (Y_t), X^{(1)}_t, \ldots X^{(d)}_t \right),
\end{align*}
in which case $Y$ satisfies the RDE
\begin{align*}
\mathrm{d} Y_t = V(Y_t) \strato{\mathbf{X}_t}, \quad Y_0 = \left( y_0, 0 \right),
\end{align*}
where  $V(Y) \in \mathbb{R}^{2d} \otimes \mathbb{R}^d$ is represented by the $2d$-by-$d$ matrix
\begin{align*}
V(Y_t) = \begin{bmatrix}
\begin{array} {c}
\nabla^2 f (Y_t) \\
\hdashline
\mathcal{I}_d
\end{array}
\end{bmatrix}.
\end{align*}
Now one can check that $[V_i, V_j] = 0$ for all $i, j = 1, \ldots, d$, and apply Corollary \ref{thmSimple}. \par
Alternatively, since $\nabla V(Y_t) \strato{\mathbf{X}_t}$ is the upper-triangular $2d$-by-$2d$ matrix
\begin{align*}
\nabla V(Y_t) \strato{\mathbf{X}_t}
= \begin{bmatrix}
\begin{array}{c:c}
0  & \nabla^3 f(X_t) \strato{\mathbf{X}_t} \\
\hdashline
\phantom{123456} 0 \phantom{123456} & 0
\end{array}
\end{bmatrix},
\end{align*}
where
\begin{align*}
\left( \nabla^3 f (X_t) \strato{\mathbf{X}_t} \right)_{ij} = \sum_{k=1}^d \frac{\partial^3 f}{\partial e_k \partial e_i \partial e_j} (X_t) \strato{\mathbf{X}^{(k)}_t}, \quad i, j = 1, \ldots, d,
\end{align*}
one can directly compute the solution to the RDE satisfied by the Jacobian as
\begin{align*}
J^{\mathbf{X}}_t = \begin{bmatrix}
\begin{array}{c:c}
\mathcal{I}_d & \nabla^2 f(X_t) - \nabla^2 f(0) \\
\hdashline
\phantom{1234567} 0 \phantom{1234567} & \mathcal{I}_d
\end{array}
\end{bmatrix}, \quad t \in [0, T].
\end{align*}
Now since the inverse is given by
\begin{align*}
\left( J^{\mathbf{X}}_t \right)^{-1} = \begin{bmatrix}
\begin{array}{c:c}
\mathcal{I}_d & \nabla^2 f(0) - \nabla^2 f (X_t) \\
\hdashline
\phantom{1234567} 0 \phantom{1234567} & \mathcal{I}_d
\end{array}
\end{bmatrix}, \quad t \in [0, T],
\end{align*}
we also obtain, for $0 \leq s < t \leq T$,
\begin{align*}
J^{\mathbf{X}}_{t \leftarrow s} V(Y_s) = V(Y_t).
\end{align*}
\end{proof}

\end{document}